\documentclass{imsart}

\RequirePackage{amsthm,amsmath,amsfonts,amssymb}
\RequirePackage[authoryear]{natbib}%% uncomment this for author-year citations
\RequirePackage[colorlinks,citecolor=blue,urlcolor=blue, linkcolor=blue]{hyperref}%% uncomment this for coloring bibliography citations and linked URLs
\RequirePackage{graphicx}%% uncomment this for including figures

\usepackage[capitalise]{cleveref}  
\usepackage{adjustbox, array}
\usepackage{bbm}
\usepackage{psfrag}
\usepackage[dvipsnames]{xcolor}
\usepackage{enumerate}
\usepackage{url} 
\usepackage{bm}
\usepackage[bottom]{footmisc}
\usepackage{booktabs}
\usepackage{color}
\usepackage{comment}
\usepackage{mathtools}
\usepackage{multirow} 
\usepackage{enumitem}
\usepackage{float}
\usepackage{epsfig}

\startlocaldefs
\theoremstyle{definition}
\newtheorem{lemma}{Lemma}[section]
\newtheorem{remark}{Remark}[section]
\newtheorem{corollary}{Corollary}[section]

\newtheorem{theorem}{Theorem}[section]

\newtheorem{assumption}{Assumption}
\newtheorem*{assumption*}{Assumption}

\newcommand{\field}[1]{\mathbb{#1}}

\newcommand{\p}{\field{P}}

\newcommand{\bphi}{\boldsymbol{\phi}}

\DeclareMathOperator*{\tr}{tr}

\newcommand{\R}{\mathbb{R}}

\newcommand{\pr}{\mathbb{P}}
\renewcommand{\Pr}{\mathbb{P}}
\newcommand{\E}{\mathbb{E}}

\DeclareMathOperator*{\Cov}{Cov}

\newcommand{\FF}{{\cal F}}

\newcommand{\dee }{ \text{d}}

\newcommand{\Ind}{\mathbbm{1}}
\newcommand{\bigO}{O}

\newcommand{\um}{\underline{m}}

\newcommand{\clbm}{\textcolor{red}{\underline{m}}}

\newcommand{\RED}[1]{\textcolor{red}{#1}}

\newcommand{\PURPLE}[1]{\textcolor{purple}{#1}}

\def\TITLE{Wasserstein and Convex Gaussian Approximations for Non-stationary Time Series of Diverging Dimensionality}
\endlocaldefs

\begin{document}

\begin{frontmatter}

\title{\TITLE{}}
%\title{A sample article title with some additional note\thanksref{T1}}
\runtitle{Wasserstein and Convex Gaussian Approximations for Time Series}
%\thankstext{T1}{A sample of additional note to the title.}

\begin{aug}
%%%%%%%%%%%%%%%%%%%%%%%%%%%%%%%%%%%%%%%%%%%%%%%
%% Only one address is permitted per author. %%
%% Only division, organization and e-mail is %%
%% included in the address.                  %%
%% Additional information can be included in %%
\author[A]{\fnms{Miaoshiqi} ~\snm{Liu}\ead[label=e1]{miaoshiqi.liu@mail.utoronto.ca}},
\author[B]{\fnms{Jun} \snm{Yang}\ead[label=e2]{jy@math.ku.dk}}
\and
\author[A]{\fnms{Zhou} \snm{Zhou}\ead[label=e3]{zhou.zhou@utoronto.ca}}
%%%%%%%%%%%%%%%%%%%%%%%%%%%%%%%%%%%%%%%%%%%%%%
%% Addresses                                %%
%%%%%%%%%%%%%%%%%%%%%%%%%%%%%%%%%%%%%%%%%%%%%%
\address[A]{Department of Statistical Sciences, University of Toronto, Canada.\printead[presep={,\ }]{e1,e3}}
\address[B]{Department of Mathematical Sciences, University of Copenhagen, Denmark.\printead[presep={,\ }]{e2}}
\end{aug}

\begin{abstract}
In high-dimensional time series analysis, Gaussian approximation (GA) schemes under various distance measures or on various collections of subsets of the Euclidean space play a fundamental role in a wide range of statistical inference problems. To date, most GA results for high-dimensional time series are established on hyper-rectangles and their equivalence. In this paper, by considering the 2-Wasserstein distance and the collection of all convex sets, we establish a general GA theory for a broad class of high-dimensional non-stationary (HDNS) time series, extending the scope of problems that can be addressed in HDNS time series analysis. For HDNS time series of sufficiently weak dependence and light tail, the GA rates established in this paper are either nearly
optimal with respect to the dimensionality and time series length, or they are nearly identical
to the corresponding best-known GA rates established for independent data. A multiplier bootstrap procedure is utilized and theoretically justified to implement our GA theory. We demonstrate by two previously undiscussed time series applications the use of the GA theory and the bootstrap procedure as unified tools for a wide range of statistical inference problems in HDNS time series analysis.
\end{abstract}



\end{frontmatter}
%%%%%%%%%%%%%%%%%%%%%%%%%%%%%%%%%%%%%%%%%%%%%%
%% Please use \tableofcontents for articles %%
%% with 50 pages and more                   %%
%%%%%%%%%%%%%%%%%%%%%%%%%%%%%%%%%%%%%%%%%%%%%%
%\tableofcontents

%%%%%%%%%%%%%%%%%%%%%%%%%%%%%%%%%%%%%%%%%%%%%%
%%%% Main text entry area:
\section{Introduction}
%The central limit theorem (CLT) and its extensions, such as the Berry--Esseen theorem, are among the most useful results in classic statistics. In high dimensions, the analogous results to the CLT are Gaussian approximation (GA) schemes on various collections of subsets of the multi-dimensional Euclidean space $\R^d$. %\textcolor{OliveGreen}{One widely-explored topic is to consider Gaussian approximation over the collection of hyper-rectangles, where explicit error bounds have been derived and improved by \citet{chernozhukov2013bootstrap, Chernozhukov_Chetverikov_Kato_2017}, and many others. Another avenue of research focuses on controlling the GA bound over the class of convex sets in $\R^d$ \citep{Nagaev_1976, Gotze_1991}, with a Berry-Esseen bound later established by \citet{bentkus2003dependence} for isotropic distributions. More generally, \citet{Zaitsev1987,zaitsev1987estimates} presented GA results applicable to Borel sets under the L\'evy-Prohorov metric. Additional studies include \citet{Zhai_2018,eldan2020clt} and the references therein. } 
For high dimensional complex time series analysis in modern data sciences, it is fundamental to establish Gaussian approximation (GA) results of the sample mean under various distance measures or on various collections of subsets of the multi-dimensional Euclidean space, much like the fundamental role the central limit theorem (CLT) played in classic statistical inference.   
While there has been a recent surge of interest in Gaussian approximations for high-dimensional time series, existing results are predominantly limited to investigating the maximum or ${\cal L}^\infty$ norm, which essentially quantifies the GA bound over hyper-rectangles. The most related work can be found in, for example, \citet{zhang2017gaussian} for stationary time series and \citet{zhang2018} and \citet{zhou2020frequency} for non-stationary time series. %and \citet{Kurisu_Kato_Shao_2021} on spatial data. 
Recently, \citet{chang2021central} extended the framework of \citet{Chernozhukov_Chetverikov_Kato_2017} to establish a GA bound for high-dimensional time series over simple convex sets and sparsely convex sets, which can be well approximated by hyper-rectangles of possibly moderately higher dimensions after certain transformations. Nevertheless, a wide class of statistical inference procedures for time series, such as those based on quadratic forms, U-statistics, eigen-analysis and thresholding, lie beyond the scope of hyper-rectangles and their equivalence. In this article, we shall establish GA results under  2-Wasserstein distance and over all Euclidean convex sets for a general class of non-stationary time series of diverging dimensionality, aiming to broaden the range of problems that can be addressed in high-dimensional non-stationary (HDNS) time series analysis. For non-stationary time series of sufficiently weak dependence and light tail, the GA rates established in this paper are either (nearly) optimal with respect to the dimensionality and time series length, or they are nearly identical to the corresponding best-known GA rates established for independent data. Details will be discussed below in Sections \ref{sec:convex} and \ref{sec:borel}.

Throughout this article, let $\{ x_{i}^{[n]}\}_{i=1}^n$ be a centered $d$-dimensional non-stationary times series, where $d=d_n$ diverges to infinity with $n$. $\{ x_{i}^{[n]}\}_{i=1}^n$ is modeled as a triangular array of time-varying Bernoulli shifts; see Equation \eqref{eq:bernoulli}. By the Rosenblatt transform \citep{rosenblatt1952remarks}, every triangular array of non-stationary time series can be represented as a triangular array of time-varying Bernoulli shift processes. Hence the latter modeling of the non-stationary process is very general.
Let $\{ y_{i}^{[n]}\}_{i=1}^n$ be a centered $d$-dimensional non-stationary {\it Gaussian} times series that preserves the covariance structure of $\{ x_{i}^{[n]}\}_{i=1}^n$. Write $X_n:=\sum_{i=1}^n x_{i}^{[n]}$ and $Y_n:=\sum_{i=1}^n y_{i}^{[n]}$. %The first main purpose of this paper is to establish a GA theory that controls 
\subsection{GA and bootstrap on high-dimensional convex sets}\label{sec:convex}
%GA on hyper-rectangles are suitable for statistical inference based on maximum deviations or the ${\cal L}^\infty$ norm. Nevertheless, a wide class of statistical inference procedures for HDNS time series, such as those based on quadratic forms or the ${\cal L}^2$ norm, cannot be tackled as a problem on hyper-rectangles. Instead, many of these procedures boil down to investigating $\p(X_n\in A)$, where $A$ is a convex set. 
%Gaussian approximation over the collection of hyper-rectangles has long served as a foundation for statistical inference based on maximum deviations, or the $\cal{L}^{\infty}$ norm, with substantial advancements made by \citet{chernozhukov2013bootstrap, Chernozhukov_Chetverikov_Kato_2017}, and many others. 
We have demonstrated that numerous time series inference problems fall outside the framework of GA on hyper-rectangles, indicating that the class of hyper-rectangles may be too small for some applications. A useful richer class is the collection of all convex sets, where many important topics in time series analysis can be viewed as inference on such sets. 
See for instance the mixed ${\cal{L}}^{2}$ and $\cal{L}^\infty$ statistics for regression inference in \cref{section:application:sub2}, where GA on the intersection of high-dimensional spheres and hyper-rectangles plays a key role for the asymptotic investigation of such statistics. It is worth noting that the aforementioned intersections are convex sets with relatively complex shapes.
%\textcolor{OliveGreen}{\textit{(Question: Do $U$ statistics fall into this realm? If so, will add here)}} Many of these procedures boil down to investigating $\p(X_n\in A)$, where $A$ is a convex set. 
In light of this, we shall investigate bounds for the following convex variation distance
\begin{equation}\label{eq:1}
d_c(X_n,Y_n):=\sup_{A\in{\cal A}}\big|\p(X_n/\sqrt{n}\in A)-\p(Y_n/\sqrt{n}\in A)\big|,
\end{equation}   
where ${\cal A}$ is the collection of all convex sets of ${\mathbb R}^d$. 

Tracing back to the independent case, early works on convex Gaussian approximation can be found in, for example, %$\beta:=\E|\boldsymbol{u}_1|^3$, 
\citet{Nagaev_1976,Bentkus_1986} and \citet{Gotze_1991}. %refined %the dependence of the bound on $d$ 
%the bound from $d^{5/2} n^{-1/2}\beta$ \citep{Sazonov_1972} to $dn^{-1/2}\beta$. This result was later improved to $d^{1/2} n^{-1/2}\beta$ by 
For independent data with finite third moments, \citet{bentkus2003dependence} established a high-dimensional convex GA bound which converges to 0 for $d$ as large as $o(n^{2/7})$ in typical cases. Recently, \citet{fang2024large} established the best-known convex GA bound for independent data with finite fourth moments using Stein's method, where in typical cases the bound decays as fast as $\bigO(n^{-1/2})$ with the sample size and converges to 0 for $d$ as large as $o(n^{2/5})$, up to a logarithmic factor. For dependent data, using Stein's method, \citet{Fang2016} established the same convex GA rate as that in \citet{bentkus2003dependence} for bounded random vectors with short neighborhood or decomposable dependence structure. Later, \citet{Cui_Zhou_2023} extended the results in \citet{bentkus2003dependence} and \citet{Fang2016} to stationary time series and applied them to simultaneous inference of time series functional linear regression. % Note that $\beta$ typically depends on $d$ in the high-dimensional setting, so the explicit dependence on $d$ would vary case by case. As a more general result, \citet{Fang2015} obtained the order of $d^{7/4}$ for dependent data that admits a Stein coupling. 
%We also refer readers to \citet{Cui_Zhou_2023} for the same order achieved for stationary and weakly dependent time series.

Our main contributions to time series convex GA lie in two aspects. First, we establish explicit convex GA bounds for a general class of HDNS time series of short memory. In particular, for time series with sufficiently high moments and sufficiently short range dependence, the convex GA bound of \eqref{eq:1} decays at $\bigO(n^{-1/2+\epsilon})$ with the sample size and the established bound converges to $0$ for $d$ as large as $\bigO(n^{2/5-\delta})$, where $\epsilon$ and $\delta$ are arbitrarily small positive constants. See Theorem \ref{main_convex} for details. Hence, under sufficiently short memory and light tail, our GA bound is nearly optimal regarding $n$ and it converges to $0$, allowing $d$ to diverge at a nearly best-known rate. To our knowledge, this is the first convex GA result for HDNS time series with nearly optimal/best-known rates. %The main goal of our present paper is to develop a new approximation technique that enables us to extend the convex approximation results to non-stationary time series with moderately high dimensions. %With $M$-dependence approximation and truncation techniques, the convex approximation results for independent and $M$-dependent data established in e.g.~\citet{bentkus2003dependence,Fang2016} will be extended to non-stationary time series with moderately high dimensions. 
Second, a high-dimensional extension of the multiplier bootstrap proposed in \citet{zhou2013heteroscedasticity} 
is developed and theoretically verified to approximate $\p(X_n/\sqrt{n} \in A)$ uniformly over all convex sets $A$. To our knowledge, the asymptotic validity of the multiplier bootstrap for uniform convex approximation of HDNS time series has not been theoretically verified before. The bootstrap is essential for practical implementations of the convex GA results. The convex GA theory and method are illustrated through an example of combined ${\cal{L}}^{2}$ and $\cal{L}^\infty$ inference for moderately high-dimensional time series in \cref{section:application:sub2}. Due to the complex dependence between the ${\cal{L}}^{2}$ and $\cal{L}^\infty$ statistics, the combined inference is difficult to tackle in high dimensions. Leveraging our convex GA results, this paper, to the best of our knowledge, represents the first attempt at high-dimensional combined-norm inference for time series. %Under some mild conditions, the conditional convex variation distance between the bootstrap statistic and $X_n$ is shown to converge to 0 in probability for $d$ as large as $o(n^{4/15})$. %\textcolor{OliveGreen}{Most notably, the main ingredient in our proof leverages the convergence bound in Wasserstein distance for i.i.d. random vectors developed in \citet{eldan2020clt}, which leads to new GA results in convex sets for high-dimensional non-stationary time series.}
    
\subsection{Wasserstein GA and bootstrap for HDNS time series}\label{sec:borel} Multiple important applications go beyond inference on convex sets. 
Prominent examples include hard and soft thresholding of wavelet coefficients and eigenvalue investigations of sample covariance matrices. In Section \ref{section:r2}, we analyze an example of threshold inference for high-dimensional time series regression, where the thresholding function is non-convex. Oftentimes, the statistics considered in those applications are continuous functions of $X_n$.  In such cases, controlling the physical distance $|X_n-Y_n|$ (on a possibly different probability space) is helpful for the theoretical investigation, where $|\cdot|$ denotes the Euclidean norm.  In this paper, we explore GA for HDNS time series via the %L\'evy-Prokhorov distance 
2-Wasserstein distance
% \begin{eqnarray}\label{eq:levy}
%     \rho(X_n,Y_n,\epsilon)=\sup_{A\in{\cal A}}\max\{\p(X_n\in A)-\p(Y_n\in A^{\epsilon}), \p(Y_n\in A)-\p(X_n\in A^{\epsilon})\},
% \end{eqnarray}
\begin{eqnarray}\label{eq:wasserstein}
\mathcal{W}_2(X_n/\sqrt{n},Y_n/\sqrt{n})=\sqrt{\inf_{\tilde{X}\stackrel{D}{=}X_n, \, \tilde{Y}\stackrel{D}{=} Y_n}\mathbb{E}[|\tilde{X}/\sqrt{n} - \tilde{Y}/\sqrt{n}|^2]},
\end{eqnarray}
%where ${\cal A}$ is the collection of all Borel sets and $A^{\epsilon}$ denotes the $\epsilon$-neighborhood of $A$. By Strassen's Theorem \citep[Ch.10]{Pollard2001}, controlling $\rho(X_n,Y_n,\epsilon)$ is equivalent to controlling the Euclidean distance $|X_n-Y_n|$ on a possibly richer probability space. For independent data with moderately high dimensions, we refer readers to \citet{Zaitsev1987} for discussions and results on $\rho(X_n,Y_n,\epsilon)$. 
where $\tilde{X}\stackrel{D}{=}X_n, \, \tilde{Y}\stackrel{D}{=} Y_n$ takes into account all the $\tilde{X}, \tilde{Y}$ that share the same distribution as $X_n, Y_n$, respectively. 

Before we state our main contributions to Wasserstein GA for time series, we shall briefly review the current literature on Wasserstein GA and closely related topics. For bounded and i.i.d. random vectors,   \citet{zhai2018high} and \citet{eldan2020clt} derived nearly optimal bounds of \eqref{eq:wasserstein} in both $d$ and $n$. In particular, the bounds decay at $\bigO(n^{-1/2})$ with the sample size and typically they converge to 0 for $d$ as large as $o(n^{1/2})$, up to a logarithm factor. Recently, using a blocking technique, \citet{mies2023sequential} extended the results of \citet{eldan2020clt} to sequential Wasserstein GA of independent but non-identically distributed random vectors and HDNS time series. For HDNS time series of sufficiently high moments and short memory, in typical cases the bound in \citet{mies2023sequential} decays at $\bigO(\sqrt{\log n}n^{-1/6})$ with the sample size and converges to 0 for $d$ as large as $\bigO(n^{1/4-\alpha})$ for arbitrarily small $\alpha>0$.

On the other hand, we point out that the closely related problem of establishing in probability and almost sure bounds of $|X_n-Y_n|$ has a long history in the probability and statistics communities under terminologies such as strong approximation and invariance principle. According to Strassen's Theorem \citep[Ch.10]{Pollard2001}, controlling the in probability bounds of $|X_n-Y_n|$ is equivalent to controlling the L\'evy--Prohorov metric
\begin{eqnarray}\label{eq:lp_metric}
	\rho(X_n,Y_n):=\inf\{\epsilon>0: \p(X_n\in B)\leq   \p(Y_n\in B^{\epsilon})+\epsilon \mbox{ for all } B\in \mathcal{B}({\mathbb R}^d) \},
\end{eqnarray}
where $\mathcal{B}({\mathbb R}^d)$ denotes the Borel $\sigma$-field of ${\mathbb R}^d$ and $B^{\epsilon}$ denotes the $\epsilon$-neighborhood of $B$. In this sense, Wasserstein GA and invariance principles are closely related to GA on all Borel sets, which is richer compared to the convex case in Section \ref{sec:convex} and therefore may be used to tackle a wider range of statistical problems. For univariate i.i.d. sample with finite $p$-th moment, \citet{komlos1975approximation, Komlos_1976} established the optimal bound $|\frac{X_n}{\sqrt{n}}-\frac{Y_n}{\sqrt{n}}|=o_{a.s.}(n^{\frac{1}{p}-\frac{1}{2}})$. We also refer to \citet{einmahl1989extensions} and \citet{zaitsev2007estimates} among many others for contributions to strong approximations/invariance principles in the independent case. For time series data, to our knowledge, to date all in probability and almost sure GA results for $|X_n-Y_n|$ and its sequential version are for the fixed-dimensional case. Among others, \citet{berkes2014komlos} and \citet{bonnerjee2024gaussian} established the optimal sequential GA rates for univariate stationary time series and multivariate non-stationary time series, respectively.  
    
%\textcolor{OliveGreen}{\textit{(May add more explanation about this distance, and how it compares to other metrics. For example, it is stronger than the previous result under L\'evy-Prokhorov distance).}} \textcolor{OliveGreen}{Under a different notion of Gaussian coupling, a sequential GA result was also established by \citet{mies2023sequential} for the partial sum processes of possibly non-stationary time series.} 

%To the best of our knowledge, there have been no results to bound %$\rho(X_n,Y_n,\epsilon)$
%$\mathcal{W}_2(X_n/\sqrt{n},Y_n/\sqrt{n})$ for high-dimensional time series. In this work, we extend the result of \citet{Zaitsev1987,zaitsev1987estimates} and \citet{eldan2020clt} to moderately HDNS time series using the conditioning technique considered in \citep{Karmakar2020,bonnerjee2024gaussian}. Furthermore, a bootstrap approximation theory is established using Stein's method and a Yurinskii coupling technique \citep[Ch.10]{Pollard2001}. \textcolor{OliveGreen}{\textit{(Will double check later whether we need more details regarding the techniques used in the proof.)}}

Our main contributions to HDNS time series Wasserstein GA are twofold. First, we establish nearly optimal bounds for \eqref{eq:wasserstein} for a general class of short-memory HDNS time series (see Theorem \ref{main_borel}). Specifically, for HDNS time series with finite $p$-th moment and sufficiently short memory, the established GA bound decays at $\bigO(n^{\frac{1}{p}-\frac{1}{2}})$ with the sample size, up to a logarithm factor; for HDNS time series with sufficiently short memory and high moments, the bound converges to 0 for $d$ as large as $\bigO(n^{\frac{1}{2}-\alpha})$ for arbitrarily small $\alpha>0$. For HDNS time series with finite exponential moment and exponentially decaying memory, $d$ can be as large as $o(n^{\frac{1}{2}})$ for the bound to vanish, up to a logarithm factor. Furthermore, our GA results are based on explicit constructions of the Gaussian random vectors, which enables tangible implementations of the theory. To our knowledge, nearly optimal Wasserstein/in probability/almost sure GA results for high-dimensional time series have not been established in the literature before. Our proof strategy is mainly based on the martingale embedding technique in \citet{eldan2020clt} and the boundary conditioning method used in \citet{berkes2014komlos, Karmakar2020}, together with blocking and M-dependence/martingale approximation techniques.

Second, a unified bootstrap algorithm as that in the convex GA case is theoretically justified to ``physically" approximate $X_n$ for practical implementations. The bootstrap approximation theory is established using Stein's method and a Yurinskii coupling technique \citep[Ch.10]{Pollard2001} which enables implementation of the methodology by comparing either the Frobenius or the maximum norm of the corresponding covariance matrices, whichever is more convenient to use in practice. As an illustration of our Wasserstein GA theory and methodology, in Section \ref{section:r2} we consider the problem of threshold inference for HDNS time series regression. Due to the non-convexity of the soft-thresholding functions, the diverging dimensionality and the temporal dependence, the latter problem was difficult to tackle and, to our knowledge, has not been discussed in the time series literature before. Utilizing our Wasserstein GA theory and the Lipchitz continuity of the threshold function, we theoretically verified the validity of the multiplier bootstrap for the threshold inference. Simulations in Section \ref{sec:simu_threshold} illustrate good finite sample performance of the bootstrap methodology. 

%we provide explicit GA bounds over convex and Borel sets, enriching the Gaussian approximation results for high-dimensional non-stationary time series. The dependence of the derived bounds on sample size $n$ and dimension $d$ is regarded as optimal or sub-optimal, improving the existing results within the same regime. Secondly, an easy-to-implement multiplier bootstrap is proposed, which facilitates the subsequent statistical inference. Last but not least, we present detailed methodologies for regression inference that utilize combined norms and thresholds, which leverage the GA results over convex sets and Borel sets, respectively. These applications offer an intuitive yet powerful demonstration of the GA results, addressing a gap in the primarily theory-focused existing literature.} 

The rest of the paper is organized as follows. \cref{section:dependence} introduces the preliminaries of HDNS time series, including the physical dependence measures and basic assumptions. The main theorems for GA of HDNS time series on convex sets and with respect to the 2-Wasserstein distance are presented in \cref{section:main}. The corresponding general bootstrap procedures for HDNS time series are introduced in \cref{section:bootstrap}. \cref{section:application} explores two applications of the proposed GA theory and bootstrap procedures, with their finite-sample performance demonstrated via simulations in \cref{section:simulation}. Technical proofs of the theoretical results are deferred to the appendices.

%\textcolor{OliveGreen}{Since the bound is relatively long, I'm not sure which way is better in terms of readability. Just for comparison:}\\
%Using $\min$: \qquad  \qquad $\min\left\{d^{\frac{7}{4}}n^{\frac{9}{2p}-\frac{1}{2}}(\log n)^2, d^{\frac{5}{6}}n^{\frac{2}{3p}-\frac{1}{3}}(\log n)^{\frac{1}{3}}\right\}$\\
%Using big wedge:\qquad  $\left(d^{\frac{7}{4}}n^{\frac{9}{2p}-\frac{1}{2}}(\log n)^2\right) \bigwedge \left( d^{\frac{5}{6}}n^{\frac{2}{3p}-\frac{1}{3}}(\log n)^{\frac{1}{3}}\right)$

\section{Preliminary}\label{section:dependence}

Throughout this article, we write $a \lesssim b$ if $a$ is smaller than or equal to $b$ up to a universal positive constant. For a random vector $\boldsymbol{v}=(v_1, \cdots, v_d)^\top$, we denote $|\cdot|$ as the Euclidean norm, $|\cdot|_{\infty}$ as the maximum norm, and $\left\|\cdot\right\|_q$ as the $\mathcal{L}^q$ norm. For a random matrix $A = (a_{ij})$, let $|\cdot |$ denote the operator norm, $|\cdot|_F$ denote the Frobenius norm, and $|\cdot|_{\max}$ denote the entry-wise max norm. For two random vectors $\boldsymbol{u}, \boldsymbol{v}$, write $\boldsymbol{u} \stackrel{D}{=} \boldsymbol{v}$ if they share the same distribution. %\RED{JY: maybe use $a\wedge b$ for $\min\{a,b\}$? there's big wedge $a\bigwedge b$ too.}
We consider the following $d$-dimensional triangular array of HDNS time series 
\begin{eqnarray}\label{eq:bernoulli}
x_{i}^{[n]} := \mathcal{G}^{[n]}_{i}(\mathcal{F}^{[n]}_i), \quad i=1, \cdots, n,
\end{eqnarray}
where $\mathcal{G}^{[n]}_{i}: \R^{\infty} \to \R^d$ is a measurable function, $\mathcal{F}^{[n]}_i = (\dots,e^{[n]}_{i-1},e^{[n]}_i)$ is a filtration generated by i.i.d.\ random elements $\{e^{[n]}_i\}_{i\in\mathbb{Z}}$. Furthermore, we denote the $j$-th coordinate $x_{i,j}^{[n]}$ as $\mathcal{G}^{[n]}_{i,j}(\mathcal{F}^{[n]}_i)$. By the Rosenblatt transform \citep{rosenblatt1952remarks}, every triangular array of non-stationary time series can be represented as a triangular array of time-varying Bernoulli shift processes. To quantify the temporal dependence of the time series, we use the following physical dependence measure as introduced by \citet{wu2005nonlinear}:  
\[ 
\theta_{k,j,q} := \sup_n\max_{1 \leq i \leq n} \| \mathcal{G}^{[n]}_{i,j}(\mathcal{F}^{[n]}_i) -\mathcal{G}^{[n]}_{i,j}(\mathcal{F}^{[n]}_{i,i-k}) \|_q,
\] 
where $\mathcal{F}^{[n]}_{i,i-k} := (\dots,\hat{e}^{[n]}_{i - k},...,e^{[n]}_i)$ and $\{\hat{e}^{[n]}_i\}_{i\in\mathbb{Z}}$ is an i.i.d. copy of $\{e^{[n]}_i\}_{i\in\mathbb{Z}}$. Observe that $\mathcal{F}^{[n]}_{i,i-k}$ differs from the original filtration $\mathcal{F}^{[n]}_i$ by replacing $e^{[n]}_{i-k}$ with $\hat{e}^{[n]}_{i-k}$. 
Intuitively, $\theta_{k,j,q}$ measures the influence of the innovations of the underlying data-generating mechanism $k$ steps ahead on the current observation uniformly across time. Some literature uses the notation $\delta_{k,q} = \theta_{k,1,q}$ for univariate case where $x_i \in \R$. For the sake of convenience, we generalize the notation as 
\[
\delta_{k,q} := \max_{1 \leq j \leq d}\theta_{k,j,q}
\] 
for $x_i \in \R^d$. We also define  
\[   \Theta_{k,q}:= \max_{1 \leq j \leq d}\sum_{l=k}^{\infty} \theta_{l,j,q},	
\] 
which quantifies the cumulative tail dependence structure. In the sequel, whenever no ambiguity is caused, we shall omit the superscript ${[n]}$. For example, we will write $x_{i}$ for $x_i^{[n]}$ and $y_{i}$ for $y_i^{[n]}$ for simplicity.%and enables one to control the auto-covariance structure of $x_i$. 

%\begin{remark}
%{\color{red} change this example} One example of such a non-stationary time series model is the piece-wise locally stationary linear process with $r$ breakpoints $0=s_0 <s _1 <\cdots < s_r < s_{r+1} = 1$, where \[x_i=G_j(i/n,\mathcal{F}_i)\Ind_{\{s_j < i/n \leq s_{j + 1}\}}, \quad G_j(t,\mathcal{F}_i) = \sum_{k = 0}^\infty a_{j,k}(t) e_{i - j}, \quad s_j < t \leq s_{j + 1},  \] and $a_{j,k}(\cdot)$ are Lipschitz continuous functions. Note here we consider $x_i\in \mathbb{R}$. The data-generating mechanism is linear with smoothly varying coefficients between the breakpoints $s_j$ and $s_{j+1}$. If $\| e_0 \|_q < \infty $, the physical dependence measure of the system can be shown as 
%\[\delta_{k,q} = \theta_{k,1,q} = \bigO \left( \max_{0 \leq i \leq r} \sup_{s_i \leq t \leq s_{i+1}} |a_{i,k}(t)| \right).\] 
%We refer to \citep{zhou2013heteroscedasticity,zhou2014inference} for more discussions and examples of the PLS process.
%\end{remark}

\begin{remark}
One example of such a high-dimensional time series is the Vector Moving Average (VMA) model of infinite order. By using time-dependent coefficient matrices, the following VMA$(\infty)$ model is non-stationary, where
\[x^{[n]}_i = \sum_{l = 0}^\infty A^{[n]}_{i,l}e^{[n]}_{i-l}, \quad A^{[n]}_{i,l} \in \R^{d\times d}, \quad \left\|e^{[n]}_i\right\|_q = 1\]
and $e^{[n]}_i \in \R^d$ are i.i.d. random vectors. The physical dependence measure of this system can be shown as
\[\delta_{k,q} = \bigO\left(\sup_n \max_{1 \le i \le n} \left|A^{[n]}_{i,k}\right|_{\infty}\right),\]
where $\left|A^{[n]}_{i,k}\right|_{\infty} = \max_{1 \le j \le d}\left|A^{[n]}_{i,k,j}\right|_1$ refers to the maximum absolute row sum of the matrix $A^{[n]}_{i,k}$, typically known as the infinity matrix norm induced by vectors.
\end{remark}

Presented below are assumptions required for the GA theory. 
%\textcolor{OliveGreen}{\textit{(Modified Assumption 2 by omitting the part about uniform integrability.)}}
\begin{assumption}\label{asm:moment}
%\begin{enumerate}
%\leavevmode
Recall that $X_n: = \sum_{i = 1}^n x_i$. We assume the smallest eigenvalue of $\Cov\left(\frac{1}{\sqrt{n}}X_n\right)$ is bounded below by some constant $\lambda_*>0$. Furthermore, we assume that for all $1 \le i \le n$ and $1 \le j \le d$,  one of the following holds:
\begin{itemize}
    \item[(a)] For some $p>2$, $x^{[n]}_i$ has finite $p$-th order moment uniformly:
    \[
    \sup_n\max_{i,j} \E[|x^{[n]}_{i,j}|^p]<C_p<\infty;
    \]
%    \PURPLE{\item[(b)] $x_i$ satisfies the $p$-th order uniform integrability condition: for $K \to \infty$,
% \[
%     \sup_{i,j}\E\left[|x_{i,j}|^p\Ind_{\{|x_{i,j}|>K\}}\right]\to 0;
% \]}
    \item[(b)] 
    $x^{[n]}_i$ has finite exponential moment uniformly: \[
    \sup_n\max_{i,j} \E[\exp(|x^{[n]}_{i,j}|)]<C_{\infty}<\infty.
    \]
\end{itemize} 
%\end{enumerate}
\end{assumption}

\begin{assumption} \label{asm:dependence}
Assume $\sup_n\max_{1\le i\le n}\|x_i^{[n]}\|_p<\infty$ for some $p>2$. The dependence measure $\theta_{m,j,p}$ satisfies one of the following:
\begin{itemize}
    \item[(a)] For $1 \le j \le d$, $\theta_{m,j,p}$  decreases polynomially as 
\[
\theta_{m,j,p}=\bigO((m+1)^{-(\chi+1)}(\log(m+1))^{-A})
\]
uniformly in $j$ for some constants $\chi > 1$ and %$A > 0$
$A > \sqrt{\chi + 1}$. Under this assumption, we have $\Theta_{m,p}=\bigO(m^{-\chi}(\log m)^{-A})$;
\item[(b)]  $\theta_{m,j,p}$ decreases exponentially as
\[
\theta_{m,j,p}=\bigO\left(\exp(-C (m+1))\right)
\]
uniformly in $j$ for some constant $C>0$. Under this assumption, we have $\Theta_{m,p}=\bigO(\exp(-C m))$.
\end{itemize}
\end{assumption}

Assumptions \ref{asm:moment} and \ref{asm:dependence} are standard moment and short-range dependence assumptions on the HDNS time series $\{x_i\}$. Additionally, Assumption \ref{asm:moment} requires the covariance matrix of $X_n/\sqrt{n}$ to be uniformly non-degenerate, which is a relatively mild condition.

\begin{comment}
\begin{assumption} \label{asm:dependence}
Throughout this paper, we assume for some finite $p$-th moment with $p>2$ that the dependence measure $\Theta_{m,p}$ satisfies one of the following:
\begin{itemize}
    \item[(a)] $\Theta_{m,p}$  decreases polynomially as 
\[
\Theta_{m,p}=\bigO(m^{-\chi}(\log m)^{-A})
\]
for some constant $\chi$ and $A$ (note that the constant $A$ is only used for simplicity of the proof)
\item[(b)]  $\Theta_{m,p}$ decreases exponentially as
\[
\Theta_{m,p}=\bigO(\exp(-C m))
\]
for some constant $C>0$. 
\end{itemize}
\end{assumption}   
\end{comment}

%\subsection{Summary of Main Results}
%In this work, we focus on three important classes of ${\cal A}$: the hyper-rectangles, all convex sets and all Borel sets.% (with a slight modification of distance).  %All the theory will be established under the PLS framework. %a very general physical representation of the non-stationary time series, that is, 
%\begin{eqnarray}\label{eq:repre}
%x_{i,n}=G_{i,n}(\FF_i),\, i=1,2,\cdots,n,
%\end{eqnarray}  
%where $G_{i,n}$ is an $h$-dimensional function that represents the time-varying underlying data generating mechanism, $\FF_i=(\cdots,\eta_{i-1},\eta_{i})$ and $\{\eta_i\}_{i\in\mathcal{Z}}$ are i.i.d. random elements that represent the innovations or shocks which drive the temporal system. Observe that \cref{eq:repre} is a very general representation and includes the PLS class (Zhou, 2013) as a special case.  

\section{Gaussian Approximation for HDNS Time Series} \label{section:main}
This section provides theoretical results for Wasserstein and convex GA bounds for HDNS time series.
\subsection{Wasserstein GA for HDNS time series}\label{sec:WGA}
This subsection is devoted to controlling the %L\'evy-Prokhorov distance 
2-Wasserstein distance as defined in \eqref{eq:wasserstein}. We focus on the dependence of this Gaussian approximation bound on both $n$ and $d$, which is presented below by \cref{main_borel}. %By Strassen's theorem, it is equivalent to show
% \[
%     \left|X_n^c-Y_n^c\right| =o_{\Pr}(\tau_{n,d}),
% \]
% where $X_n^c\stackrel{D}{=} X_n$, $Y_n^c\stackrel{D}{=} Y_n$, and $X_n^c$ and $Y_n^c$ are defined in the same probability space $(\Sigma_c,A_c, P_c)$. %We know $\tau_n=o(n^{1/p})$ is optimal in $n$ where $p$ is the number of finite moments of $x_i$ \textcolor{OliveGreen}{by Markov inequality}. We focus more on the dependence on $d$.
% We focus on the dependence of $\tau_{n,d}$ on both $n$ and $d$.

%\RED{Let $\mathcal{W}_2(\cdot,\cdot)$ be the $2$-Wasserstein distance. Then we have the following main result.}
\begin{comment}
\begin{assumption} \label{A2}
We assume that for all $i$, every element $\{x_i\}$ has finite $p$-th moment uniformly $\sup_{j=1,\dots,d} \E[|x_{i,j}|^p]<C_p<\infty$. Moreover, we assume the following uniform integrability condition: if $K\to\infty$ then
\[
    \sup_{i,j}\E\left[|x_{i,j}|^p\Ind_{\{|x_{i,j}|>K\}}\right]\to 0.
\]
Finally, we assume the smallest eigenvalue of $\Cov(\frac{1}{\sqrt{n}}X_n)$ is bounded below by some constant $\lambda_*>0$. 
\end{assumption}   
\end{comment}

% \begin{remark}
% The uniform integrability condition in \cref{A2} has been used in existing literature such as \citep{Karmakar2020}, It is slightly stronger than the finite $p$-th moment condition, which implies uniform integrability of $|x_{i,j}|^q$ for any $q<p$, because as $K\to\infty$
% \[
% \sup_{i,j} \E\left[|x_{i,j}|^q\Ind_{\{|x_i|>K\}}\right]
% 	\le \frac{\sup_{i,j} \E\left[|x_{i,j}|^p \right]}{K^{p-q}}\le C_p K^{q-p}\to 0.
% \]
% \end{remark}

\begin{theorem}\label{main_borel} 
Recall $Y_n\in\mathbb{R}^d$ is a Gaussian vector with $\Cov\left(\frac{1}{\sqrt{n}}Y_n \right)=\Cov\left(\frac{1}{\sqrt{n}}X_n\right)$.  %If $\{x_i\}_{i = 1}^n$ satisfy \cref{asm:dependence} and \cref{asm:moment}, then 
We have
%and \cref{asm:moment}, then one could construct $X_n^c$ and $Y_n^c$, $X_n^c\stackrel{D}{=} X_n$, $Y_n^c\stackrel{D}{=} Y_n$, and $X_n^c$ and $Y_n^c$ are defined in the same probability space, such that:
\begin{itemize}
\item If $\{x_i\}$ satisfies \cref{asm:moment}(a), 
\begin{equation}\label{eq:311}
     \mathcal{W}_2\left(\frac{1}{\sqrt{n}}X_n, \frac{1}{\sqrt{n}}Y_n\right) = \begin{cases}\bigO(d{n^{\frac{1}{r}-\frac{1}{2}}}(\log n)),% + \min\{dn^{\frac{1}{r}-\frac{1}{2}}, d^{5/8}n^{\frac{1}{2r}-\frac{1}{4}}\}) 
     \quad \textrm{under \cref{asm:dependence}(a),}\\ %with  A>\sqrt{\chi+1},\\
     \bigO(d{n^{\frac{1}{p}-\frac{1}{2}}}(\log n)), 
     \quad \textrm{under \cref{asm:dependence}(b)},
     \end{cases}
\end{equation}
where $1/r= \max\left\{1/p, \frac{1}{\sqrt{\chi+1}p}+\left(\frac{1}{2}-\frac{1}{\sqrt{\chi+1}p}\right)\max\left\{\frac{2}{\sqrt{\chi+1}p}, \frac{1}{\chi}\left(\frac{1}{2}-\frac{1}{p}\right)\right\}\right\}$.
\item If $\{x_i\}$ satisfies \cref{asm:moment}(b),
\begin{equation}\label{eq:312}
     \mathcal{W}_2\left(\frac{1}{\sqrt{n}}X_n, \frac{1}{\sqrt{n}}Y_n\right) = \begin{cases}\bigO(d{n^{\frac{1}{4\chi}-\frac{1}{2}}}(\log n)^4) 
     \quad \textrm{under \cref{asm:dependence}(a),}\\ %with  A>\sqrt{\chi+1},\\
     \bigO(d{n^{-\frac{1}{2}}}(\log n)^5),
     \quad \textrm{under \cref{asm:dependence}(b)}.
     \end{cases}
\end{equation}
\end{itemize}
\end{theorem}
%\begin{proof}
%    See \cref{proof_main_borel}. 
%\end{proof}

% ----------------------------------------------
% The previous version of \cref{main_borel}
\begin{comment}
Suppose that $\{x_i\}_{i = 1}^n$ satisfy \cref{asm:dependence}(a) with $A>\sqrt{\chi+1}$ and the $p$-th order uniform integrability condition in \cref{asm:moment}(b), then one can construct $X_n^c$ and $Y_n^c$, such that $X_n^c\stackrel{D}{=} X_n$, $Y_n^c\stackrel{D}{=} Y_n$, and $X_n^c$ and $Y_n^c$ are defined in the same probability space. Furthermore, we have
% \[
%     \left|\frac{1}{\sqrt{n}}X_n^c-\frac{1}{\sqrt{n}}Y_n^c\right|=\RED{\bigO_{\pr}(\min\{d^{5/2}(\log d),d\log(n)\}{n^{\frac{1}{r}-\frac{1}{2}}})},
% \]
\[
    \left|\frac{1}{\sqrt{n}}X_n^c-\frac{1}{\sqrt{n}}Y_n^c\right|=\RED{\bigO_{\pr}(d{n^{\frac{1}{r}-\frac{1}{2}}}\log n)},
\]
where {$1/r= \max\left\{1/p, \frac{1}{\sqrt{\chi+1}p}+\left(\frac{1}{2}-\frac{1}{\sqrt{\chi+1}p}\right)\max\left\{\frac{2}{\sqrt{\chi+1}p}, \frac{1}{\chi}\left(\frac{1}{2}-\frac{1}{p}\right)\right\}\right\}$}
\RED{and $1/r=1/p$ if \cref{asm:dependence}(b) is satisfied instead of \cref{asm:dependence}(a).}
\textcolor{OliveGreen}{Removed the minimum to simplify the order based on Zhou's comment, as the first part is always larger than the second part for diverging $d$.}
\end{comment}
% End of The previous version of \cref{main_borel}
% ----------------------------------------------

As we discussed in the Introduction, when the dimension is fixed, under the finite $p$-th moment condition the optimal bound for the in probability/almost sure GA with respect to $n$ is $o(n^{1/p-1/2})$. See \citet{Komlos_1976, berkes2014komlos,bonnerjee2024gaussian}. Since the ${\cal W}_2$ distance is stronger than the in probability distance, the optimal ${\cal W}_2$ GA bound with respect to $n$ must be at most $o(n^{1/p-1/2})$. When  $\chi$ is sufficiently large, the rate in \eqref{eq:311} of \cref{main_borel} becomes $\bigO(n^{\frac{1}{p}-\frac{1}{2}})$ with respect to the sample size, up to a logarithm term, and hence must be nearly optimal in terms of $n$ and $p$, within a factor of logarithm. %The sub-optimal regime when $\chi$ is small corresponds to the slow decay of the dependence measure. 
We now discuss our bound with respect to the dimension $d$. For bounded and i.i.d random vectors, \citet{eldan2020clt} established a nearly optimal ${\cal W}_2$ bound, which in typical cases is proportional to $dn^{-1/2}\sqrt{\log n}$. Furthermore, \citet{eldan2020clt} demonstrated that $d$ cannot exceed $o(n^{1/2})$ in general for the bound to vanish. Hence, we expect a factor of $d$ to be optimal when multiplied into the bound. Observe that in \eqref{eq:311} and \eqref{eq:312}, the bounds depend on $d$ through a factor of $d$. In particular, when $p$ and $\chi$ are sufficiently large, \eqref{eq:311} allows $d$ to be as large as $\bigO(n^{1/2-\alpha})$ for arbitrarily small $\alpha>0$ for the bound to vanish. When $\{x_i\}$ has exponentially decaying memory and finite exponential tail, the bound in \eqref{eq:312} vanishes for $d$ as large as $o(n^{1/2})$, up to a logarithm factor. Therefore, the bounds established in Theorem \ref{main_borel} are nearly optimal for HDNS time series of sufficiently short memory and high moments. Finally, we point out that our proof of Theorem \ref{main_borel} is constructive. That is, we explicitly construct random vectors $Y_n$ so that $\{\E(|X_n-Y_n|^2)\}^{1/2}/\sqrt{n}$ achieves the bounds in \eqref{eq:311} and \eqref{eq:312}. This explicit construction is helpful in applications.

\subsection{GA on high-dimensional convex sets}
To investigate GA over the collection of all convex sets, recall that we consider the convex variation distance in \cref{eq:1}: %\textcolor{OliveGreen}{\textit{Why are we using $B$ instead of $A$ here?}}
\begin{equation*}
d_c(X_n,Y_n):=\sup_{A\in\mathcal{A}}\left|\Pr\left(\frac{1}{\sqrt{n}}X_n\in A\right) - \Pr\left(\frac{1}{\sqrt{n}}Y_n\in A\right) \right|,
\end{equation*}
where $\mathcal{A}$ is the collection of all convex sets in $\R^d$. The convex variation distance serves as a tool to compare the distribution of two high-dimensional random vectors over all convex sets. Our convex GA theory is summarized in the following theorem.

\begin{theorem}\label{main_convex}
%Assume $\frac{1}{\sqrt{n}}Y_n$ is a multivariate Gaussian random vector with $\Cov\left(\frac{1}{\sqrt{n}}Y_n \right)=\Cov\left(\frac{1}{\sqrt{n}}X_n\right)$.  If $\{x_i\}_{i = 1}^n$ satisfy \cref{asm:dependence} and \cref{asm:moment}, then:
We have the following:
\begin{itemize}
\item If $\{x_i\}$ %has up to $p$-th order finite moment uniformly, 
satisfies \cref{asm:moment}(a),
\begin{equation}\label{eq:321}	
d_c\left(X_n,Y_n\right) = \begin{cases}
		\bigO\left(\min\left\{d^{\frac{7}{4}}n^{\frac{9}{2p}+(2-\frac{3}{p})\frac{1}{\chi}-\frac{1}{2}}, d^{\frac{5}{6}}n^{\frac{2}{3r}-\frac{1}{3}}(\log n)^{\frac{1}{3}}\right\}\right),&\quad \textrm{under \ref{asm:dependence}(a)},\\
		\bigO\left(\min\left\{d^{\frac{7}{4}}n^{\frac{9}{2p}-\frac{1}{2}}(\log n)^2, d^{\frac{5}{6}}n^{\frac{2}{3p}-\frac{1}{3}}(\log n)^{\frac{1}{3}}\right\}\right),&\quad
		\textrm{under \ref{asm:dependence}(b)}.
	\end{cases}
    \end{equation}
Recall that $1/r= \max\left\{1/p, \frac{1}{\sqrt{\chi+1}p}+\left(\frac{1}{2}-\frac{1}{\sqrt{\chi+1}p}\right)\max\left\{\frac{2}{\sqrt{\chi+1}p}, \frac{1}{\chi}\left(\frac{1}{2}-\frac{1}{p}\right)\right\}\right\}$.    
\item If $\{x_i\}$ %has finite exponential moment uniformly, 
satisfies \cref{asm:moment}(b), 
%\[ d_c\left(\frac{1}{\sqrt{n}}X_n,Y\right) = \bigO(d^{7/4}n^{-1/2}(\log n)^3m^2), \]
\begin{equation}\label{eq:322}	
    d_c\left(X_n,Y_n\right) = \begin{cases}
			\bigO\left(\min\left\{d^{\frac{7}{4}}n^{\frac{2}{\chi}-\frac{1}{2}}(\log n)^3, d^{\frac{5}{6}}n^{\frac{1}{6\chi}-\frac{1}{3}}(\log n)^{\frac{8}{3}}\right\}\right),&\quad \textrm{under \ref{asm:dependence}(a)},\\
			\bigO\left(\min\left\{d^{\frac{7}{4}}n^{-\frac{1}{2}}(\log n)^5,d^{\frac{5}{6}}n^{-\frac{1}{3}}(\log n)^{\frac{10}{3}}\right\}\right),&\quad
			\textrm{under \ref{asm:dependence}(b)}.
	\end{cases}
    \end{equation}
\end{itemize}
\end{theorem}
%\begin{proof}
%See \cref{proof_main_convex}. 
%\end{proof}

%\textcolor{OliveGreen}{\textit{(Deleted the result using uniform integrability, and combined two paths under exponential moment using the minimum of them. Note: the first part of the $\min()$ does not need the $A > \sqrt{\chi +1}$ condition, may consider add this to the remark...)}}

The proof of Theorem \ref{main_convex} relies on an extension of Stein's method used in \citet{Bentkus_1986,Fang2016} to HDNS time series, together with the Wasserstein GA results established in Theorem \ref{main_borel}. 
In particular, the rates of \cref{main_convex} are obtained using an $m$-dependence approximation technique by choosing $m$ in the dependence measure $\Theta_{m,p}$ such that $\Theta_{m,p}= o(n^{\frac{3}{2p}-1})$, where $p$ is the highest order of finite moments of $\{x_i\}$ or $\Theta_{m,p} =  o(n^{-1})$ when $x_i$ has finite exponential moment. When $p$ and $\chi$ are sufficiently large, the first terms on the right hand side of  \eqref{eq:321} and \eqref{eq:322} converge at the rate  $\bigO(n^{-1/2+\alpha})$ for arbitrarily small $\alpha>0$, which is nearly optimal in terms of $n$. Regarding the dimension $d$, when $p$ and $\chi$ are sufficiently large, the last terms on the right hand sides of \eqref{eq:321} and \eqref{eq:322} become $d^{5/6}n^{-1/3+\beta}$ for arbitrarily small $\beta>0$, within a factor of logarithm. Therefore the bounds \eqref{eq:321} and \eqref{eq:322} vanishes for $d$ as large as $\bigO(n^{2/5-\alpha})$ for arbitrarily small $\alpha>0$ . If the time series has exponentially decaying memory and finite exponential moments, \eqref{eq:322} admits $d=o(n^{2/5})$, up to a logarithm factor. As we discussed in the Introduction, in typical cases the best-known divergence rate of $d$ for independent data with finite fourth moment is $o(n^{2/5})$ which is established in \citet{fang2024large}. In this sense, our result with respect to $d$ is nearly the best-known rate for sufficiently short memory and light-tailed HDNS time series.

\section{Bootstrap Procedure for HDNS Time Series}\label{section:bootstrap}
The main results from \cref{section:main} enable one to bound the error for approximating $X_n=\sum_{i=1}^n x_{i}$ with its Gaussian counterpart $Y_n$. Nevertheless, the construction of $Y_n$ is difficult in practice due to the complex dependence structure of HDNS time series. To fill the gap, we present a general bootstrap strategy for approximating $Y_n$ in practice. Define the multiplier bootstrap statistics 
%\[ \tau_L := \sum_{i = 1}^{\textcolor{OliveGreen}{n-L+1}} B_i \psi_{i,L}/\textcolor{OliveGreen}{\sqrt{n-L+1}} \]
\begin{equation}\label{eq:boot_statistic}
    \tau_{n,L} := \sum_{i = 1}^{n-L+1} B_i \psi_{i,L},
\end{equation}
where $B_i$'s are i.i.d. standard normal random variables independent of the data $\{x_i\}$, $L$ is a bandwidth/window size parameter chosen by the user and $\psi_{i,L} := \sum_{j = i}^{{i + L-1}} x_{j}/\sqrt{L}$. The practical selection of $L$ will be discussed in Section \ref{section:simulation}. Our bootstrap procedure uses $\tau_{n,L}$ to approximate $Y_n$. Note that the short-range block sum $\psi_{i,L}$ allows $\tau_{n,L}$ to approximate the correlation structure of $\{x_{i} \}_{i = 1}^n$ up to lag $L$. Due to the short range dependence assumption on $\{x_i\}$, if $L\rightarrow\infty$, then  $\tau_{n,L}$ captures most of the covariance information in $\{x_i\}$. Two more observations can help further understand the multiplier bootstrap procedure. First, conditional on the data $\{x_{i} \}_{i = 1}^n$, $\tau_{n,L}$ is a Gaussian vector, which makes the comparison between distributions of $\tau_{n,L}$ and $Y_n$ relatively easy. Second, by generating independent sets of i.i.d. standard normal variables $B_i$, we can approximate the distribution of $Y_n$ (and hence that of $X_n$) based on the conditional empirical quantiles of $\tau_{n,L}$.  

The following lemma quantifies the distance between the covariance matrices of $X_n/\sqrt{n}$ and $\tau_{n,L}/\sqrt{n-L+1} \mid \{x_i\}_{i=1}^n$, which serves as an intermediate step to prove the consistency of the multiplier bootstrap procedure.

\begin{lemma}\label{lemma_bootstrap_delta} 
Suppose $\{x_i\}_{i=1}^n$ satisfy \cref{asm:moment}(a) and \cref{asm:dependence}(a) with $p \ge 4$. %\textcolor{OliveGreen}{$\max_{1 \le i \le n}\max_{1 \le j \le d}\E|x_{i,j}|^{2r} < C$ and the $2r$-th ($r \ge 2$) dependence measure is exponentially decaying or polynomially decaying with $\chi > 2$}, 
Recall the bootstrap statistic $\tau_{n,L}$ introduced in \eqref{eq:boot_statistic}. Define 
\[  \Delta_n^{(1)}  :=  \left|\Cov\left[ X_n /\sqrt{n} \right]- \Cov\left[ \tau_{n,L}/\sqrt{n-L+1} \mid \{x_i\}_{i=1}^n \right] \right|_F,\]
and
\[  \Delta_n^{(2)}  :=  \left|\Cov\left[ X_n /\sqrt{n} \right]- \Cov\left[ \tau_{n,L}/\sqrt{n-L+1} \mid \{x_i\}_{i=1}^n \right] \right|_{\max}.\]
%where $|\cdot|_F$ denotes the Frobenius norm $|\cdot|_F$, and $|\cdot|_{\max}$ refers to the maximum norm. 
Then we have
\[\Delta_n^{(1)}  = \bigO_\pr \left(d\left(\sqrt{\frac{L}{n}}+\frac{1}{L}\right)\right) \quad \text{and} \quad \Delta_n^{(2)}   = \bigO_{\pr}\left(\sqrt{\frac{L}{n}}d^{4/p} + \frac{L}{n} + \frac{1}{L}\right).\]
\end{lemma}

\subsection{Bootstrap for Wasserstein and in probability GA} \label{sec:boot_borel} 
We first prove the following two important lemmas for two Gaussian vectors $X$ and $Y$, which are useful for proving our main results on the bootstrap.
\begin{lemma}\label{key_lemma3}
Suppose $X$ and $Y$ are zero mean Gaussian vectors in $\mathbb{R}^d$, with covariance matrices $\Sigma$ and $\hat{\Sigma}$, respectively. Suppose the smallest eigenvalue of either $\Sigma$ or $\hat{\Sigma}$ is bounded below by $\lambda_*>0$, then one can construct $X^c$ and $Y^c$ in the same probability space such that $X^c\stackrel{D}{=}X$ and $Y^c\stackrel{D}{=}Y$.
Furthermore, we have
\[
\E\left[|X^c-Y^c|^2\right]\le \lambda_*^{-1} |\Sigma-\hat{\Sigma}|_F^2.
\]
%which immediately implies $\Pr\left(|X^c-Y^c|> \epsilon \right) \le \frac{1}{\lambda_*\epsilon^2}|\Sigma-\hat{\Sigma}|_F^2$.
\end{lemma}
%\begin{proof}
%See \cref{proof_key_lemma3}.
%\end{proof}

\begin{lemma}\label{key_lemma2}
Suppose $X$ and $Y$ are zero mean Gaussian vectors in $\mathbb{R}^d$, with covariance matrices $\Sigma$ and $\hat{\Sigma}$, respectively. Then one can construct $X^c$ and $Y^c$ in the same probability space such that $X^c$ has the same distribution as $X$ and $Y^c$ has the same distribution as $Y$ such that for any given $\epsilon>0$,
\[
\Pr\left(|X^c-Y^c|> \epsilon \right) \le C d^{5/2}\frac{1}{\epsilon^2}|\Sigma-\hat{\Sigma}|_{\max},
\]
where $|\cdot|_{\max}$ denote the maximum over all elements.
\end{lemma}
%\begin{proof}
%See \cref{proof_key_lemma2}.
%\end{proof}

\cref{key_lemma3} involves explicitly constructing $X^c$ and $Y^c$ and a Frobenius norm inequality in \citet{van1980inequality}. The proof of \cref{key_lemma2} is based on Stein's method and a Yurinskii coupling technique \citep[Ch.10]{Pollard2001}. \cref{key_lemma3} and \ref{key_lemma2} enable one to validate the multiplier bootstrap procedure via either the Frobenius or the maximum norm of the covariance matrix, whichever is more convenient to use in practice. 
Next, by combining \cref{main_borel} with \cref{key_lemma3} and \cref{key_lemma2} respectively, we have the following theoretical results for asymptotic consistency of the bootstrap procedure. %in terms of L\'evy-Prohorov distance.

%\RED{JY: leaving the results below to Shiqi.}\\
%\textcolor{OliveGreen}{\textit{Also adding a remark incorporating the information we have regarding the max norm and Frobenius norm.}}
    
\begin{theorem}  \label{boot:borel}
Suppose that $\{x_i\}_{i = 1}^n$ satisfy the assumptions of \cref{main_borel}. %and the smallest eigenvalue of either $\Cov\left[ X_n /\sqrt{n} \right]$ or $\Cov\left[ \tau_{n,L}/\sqrt{n-L+1} \mid \{x_i\}_{i=1}^n \right]$ is bounded below by $\lambda_*>0$. 
%we consider
%\[ \Delta_n^{(1)}  :=  \left|\Cov\left[ X_n /\sqrt{n} \right]- \Cov\left[ \tau_{n,L}/\sqrt{n-L+1} \mid \{x_i\}_{i=1}^n \right] \right|_F\]
%and 
%\[  \Delta_n^{(2)}  :=  \left|\Cov\left[ X_n / \sqrt{n} \right]- \Cov\left[ \tau_{n,L} / \sqrt{n-L+1} \mid \{x_i\}_{i=1}^n \right] \right|_{\max}\]
%as defined in \cref{lemma_bootstrap_delta}.
%where $|\cdot|_{\max}$ denotes the maximum over all elements. 	
Assume $L \rightarrow\infty$ and $L/n\rightarrow 0$. Denote by $Q$ the distribution of $\tau_{n,L}/\sqrt{n-L+1}$ conditional on $\{x_i\}_{i=1}^n$ and by $P$ the distribution of $X_n/\sqrt{n}$. Note that $Q$ is a random Gaussian distribution. We have the following:
\begin{enumerate}
    \item[(a)]
One can construct $\boldsymbol{u}_n^{(1)}$ and $\boldsymbol{v}_n^{(1)}$ in the same probability space such that $\boldsymbol{u}_n^{(1)} \sim P$ and $\boldsymbol{v}_n^{(1)} \sim Q$. For any sequence of real numbers $\{a_n\}$ satisfying $\Pr(\Delta_n^{(1)} \leq a_n ) \rightarrow 1$, on the event $A_n:=\{\Delta_n^{(1)} \leq a_n\}$,
\begin{itemize}
\item If $x_i$ satisfies \cref{asm:moment}(a), then uniformly for all $\omega\in A_n$,
\begin{equation}
     \sqrt{\E\left|\boldsymbol{u}_n^{(1)} - \boldsymbol{v}_n^{(1)}\right|^2} = \begin{cases}\bigO( a_n+ d{n^{\frac{1}{r}-\frac{1}{2}}}(\log n))% + \min\{dn^{\frac{1}{r}-\frac{1}{2}}, d^{5/8}n^{\frac{1}{2r}-\frac{1}{4}}\}) 
     \quad \textrm{under \cref{asm:dependence}(a)},\\
     \bigO(a_n+d{n^{\frac{1}{p}-\frac{1}{2}}}(\log n))
     \quad \textrm{under \cref{asm:dependence}(b)},
     \end{cases}
\end{equation}
where $r$ is defined as in Theorem \ref{main_borel}. %$1/r= \max\left\{1/p, \frac{1}{\sqrt{\chi+1}p}+\left(\frac{1}{2}-\frac{1}{\sqrt{\chi+1}p}\right)\max\left\{\frac{2}{\sqrt{\chi+1}p}, \frac{1}{\chi}\left(\frac{1}{2}-\frac{1}{p}\right)\right\}\right\}$.
\item If $x_i$ satisfies \cref{asm:moment}(b), then uniformly for all $\omega\in A_n$
\begin{equation}
     \sqrt{\E\left|\boldsymbol{u}_n^{(1)} - \boldsymbol{v}_n^{(1)}\right|^2}  = \begin{cases}\bigO( a_n+d{n^{\frac{1}{4\chi}-\frac{1}{2}}}(\log n)^4) 
     \quad \textrm{under \cref{asm:dependence}(a)},\\
     \bigO(a_n+d{n^{-\frac{1}{2}}}(\log n)^5)
     \quad \textrm{under \cref{asm:dependence}(b)}.
     \end{cases}
\end{equation}

\end{itemize}

\item[(b)]
%One can also construct $\boldsymbol{u}_n^{(2)}$ and $\boldsymbol{v}_n^{(2)}$ in the same probability space such that $\boldsymbol{u}_n^{(2)} \sim P$ and $\boldsymbol{v}_n^{(2)} \sim Q$. 
For any sequence of real numbers $\{b_n\}$ satisfying $\Pr(\Delta_n^{(2)} \leq b_n ) \rightarrow 1$, on the event $B_n:=\{\Delta_n^{(2)} \leq b_n\}$,
\begin{itemize}
\item If $x_i$ satisfies \cref{asm:moment}(a), then uniformly for all $\omega\in B_n$
\begin{equation}
     \rho(P, Q) = \begin{cases}\bigO(d^{\frac{5}{4}} b_n^{\frac{1}{2}}+ d{n^{\frac{1}{r}-\frac{1}{2}}}(\log n)),% + \min\{dn^{\frac{1}{r}-\frac{1}{2}}, d^{5/8}n^{\frac{1}{2r}-\frac{1}{4}}\}) 
     \quad \textrm{under \cref{asm:dependence}(a)},\\
     \bigO(d^{\frac{5}{4}} b_n^{\frac{1}{2}}+d{n^{\frac{1}{p}-\frac{1}{2}}}(\log n)), 
     \quad \textrm{under \cref{asm:dependence}(b)}.
     \end{cases}
\end{equation}

\item If $x_i$ satisfies \cref{asm:moment}(b), then uniformly for all $\omega\in B_n$
\begin{equation}
     \rho(P, Q)  = \begin{cases}\bigO(d^{\frac{5}{4}} b_n^{\frac{1}{2}}+d{n^{\frac{1}{4\chi}-\frac{1}{2}}}(\log n)^4) 
     \quad \textrm{under \cref{asm:dependence}(a)},\\
     \bigO(d^{\frac{5}{4}} b_n^{\frac{1}{2}}+d{n^{-\frac{1}{2}}}(\log n)^5)
     \quad \textrm{under \cref{asm:dependence}(b)}.
     \end{cases}
\end{equation}
\end{itemize}
%where $r$ is given in \cref{main_borel}. and $\rho(\cdot,\cdot)$ denotes the L\'evy-Prohorov distance.
\end{enumerate}
\end{theorem}
%\begin{proof}
%See \cref{proof:boot_borel}. 
%\end{proof}

Observe that $Q$ is a random Gaussian measure which depends on elements $\omega$ in the sample space of $\{x_i\}$. Recall the L\'evy--Prohorov metric $\rho(\cdot,\cdot)$ defined in \eqref{eq:lp_metric}. By Strassen's Theorem, controlling $\rho(P,Q)$ in part (b) of \cref{boot:borel} is equivalent to controlling $|P-Q|$ in probability (on a possibly different probability space).
\cref{boot:borel} presents two perspectives for quantifying the bootstrap error, based on the Frobenius norm and the maximum norm of the covariance difference, respectively. In practice, one can choose the version that best suits their preference. %Incorporating the magnitude of $\Delta_n^{(1)}$ and $\Delta_n^{(2)}$ as presented in \cref{lemma_bootstrap_delta}, the first version in \cref{boot:borel} usually provides a tighter bound. 
For example, by Lemma \ref{lemma_bootstrap_delta}, the event $\left\{\Delta_n^{(1)} \le c_nd\left(\sqrt{\frac{L}{n}}+\frac{1}{L}\right)\right\}$ has probability converging to 1 for any sequence $c_n\rightarrow\infty$.  Therefore under\cref{asm:dependence}(b) and \cref{asm:moment}(b),  the bootstrap error is controlled by $d\left(c_n\left(\sqrt{\frac{L}{n}}+\frac{1}{L}\right)+\frac{(\log n)^5}{\sqrt{n}}\right)$, which simplifies to $dc_n n^{-1/3}$ by setting $L = \bigO(n^{1/3})$ and $c_n$ diverging to $\infty$ arbitrarily slowly.

% ---------------------------------------------------------------
%    Previous version of theorem 4.2 using Levy-Prohorov distance
% ---------------------------------------------------------------
% \begin{theorem}  \label{boot:borel}
% Suppose that $\{x_i\}_{i = 1}^n$ satisfy the assumptions of \cref{main_borel}, 
% we define
% \[  \Delta_n^{(2)}  :=  \left|\Cov\left[ X_n / \sqrt{n} \right]- \Cov\left[ \tau_{n,L} / \sqrt{n-L+1} \mid \{x_i\}_{i=1}^n \right] \right|_{\max}.\]	
% Suppose $L \rightarrow\infty$ and $L/n\rightarrow 0$. Denote $Q$ as the distribution of $\tau_{n,L}/\sqrt{n}$ conditional on $\{x_i\}_{i=1}^n$ and $P$ as the distribution of $X_n/\sqrt{n}$, then for any sequence of real numbers $\{a_n\}$ satisfying $\Pr(\Delta_n^{(2)} \leq a_n ) \rightarrow 1$, under the event $\{\Delta_n^{(2)} \leq a_n\}$,
% \begin{equation}
% \label{eqn:bootborel}
% \rho(P,Q) = \bigO\left( d^{5/4} a_n^{1/2} +dn^{\frac{1}{r} - \frac{1}{2}}\log(n)\right),
% \end{equation}
% where $r$ is given in \cref{main_borel} and $\rho(\cdot,\cdot)$ denotes the L\'evy-Prohorov distance.
% \end{theorem}
% Previous version of Theorem 4.2 ends here
% ---------------------------------------------------------------

\subsection{Bootstrap for GA on high-dimensional convex sets}   \label{sec:boot_convex}
%The major contribution of our work is to present a general strategy for comparing the distribution $X_n$ with the same transformation on its Gaussian counterpart $Y_n$ over different test sets, which plays a vital role in many inference scenarios. Let $\mathcal{A} \subset 2^{\R^d}$ be a class of test sets, we will quantify the distance between the distributions of $X_n$ and $\tau_{n,L}$ by the metric given below \textcolor{OliveGreen}{(But it is only used for convex sets, may well just only state this for convex sets)}
Over the collection of all convex sets $\mathcal{A}$, we quantify the following convex variation distance 
\begin{eqnarray} \label{eqn:boot:metric}
d_{c}(\mathcal{L}(X_n),\mathcal{L}(\tau_{n,L}\mid \{x_i\}_{i=1}^n)) :=  \sup_{ A \in \mathcal{A} }\big|\p(X_n \in A )-\p(\tau_{n,L} \in A \mid \{x_i\}_{i=1}^n )\big| .
\end{eqnarray}  

We first state the following useful lemma for $d_c(X,Y)$ when $X$ and $Y$ are both zero-mean Gaussian vectors in $\mathbb{R}^d$, which later contributes to the validation of the bootstrap.

\begin{lemma}\label{lemma_convex_bootstrap} 
Suppose $X$ and $Y$ are zero mean Gaussian vectors in $\mathbb{R}^d$, with positive definite covariance matrices $\Sigma$ and $\hat{\Sigma}$, respectively. Suppose the smallest eigenvalue of either $\Sigma$ or $\hat{\Sigma}$ is bounded below by $\lambda_*>0$, then 
%\[
%d_c(X,Y)=\bigO (d^{1/8}\lambda_*^{-1/4}|\Sigma-\hat{\Sigma}|_F^{1/2}).
%\]
\[
d_c(X,Y)\le \frac{3}{2} \min\left\{1, \lambda_*^{-1}|\Sigma-\hat{\Sigma}|_F\right\}.
\]
\end{lemma}
%\begin{proof}
%See \cref{proof_lemma_convex_bootstrap}. 
%\end{proof}
	
Combining \cref{lemma_convex_bootstrap} and \cref{main_convex}, we have the following %theoretical guarantee for the bootstrap procedure.

%\textbf{\textcolor{OliveGreen}{In \cref{boot:convex}, it looks like we need to use $\tau_n/\sqrt{n - L + 1}$ for the sake of consistency, this will slightly change the bound in $a_n$. The same thing happens in \cref{boot:borel}.}}\\

%\textbf{\textcolor{OliveGreen}{On the other hand, if we use the form $\tau_n/\sqrt{n}$, then we will have $a_n = O\left(\sqrt{\frac{L}{n}}d^{\frac{2}{r}}  + \frac{1}{L}\right)$ as some parts will cancel off, where $r$ needs to satisfy that the $2r$-th moment of $x_i$ is finite and the $2r$-th dependence measure is exponentially decaying or polynomially decaying with $\chi > 2$}}

\begin{theorem} \label{boot:convex}
%Suppose $\{x_i\}_{i=1}^n$ satisfy Assumptions \ref{asm:moment} and \ref{asm:dependence}. %and the smallest eigenvalue of either $\Cov\left[ X_n /\sqrt{n} \right]$ or $\Cov\left[ \tau_{n,L}/\sqrt{n-L+1} \mid \{x_i\}_{i=1}^n \right]$ is bounded below by $\lambda_*>0$.  
%we consider
%\[  \Delta_n^{(1)}  :=  \left|\Cov\left[ X_n /\sqrt{n} \right]- \Cov\left[ \tau_{n,L}/\sqrt{n-L+1} \mid \{x_i\}_{i=1}^n \right] \right|_F\]
%as defined in \cref{lemma_bootstrap_delta}, where $|\cdot|_F$ denotes the Frobenius norm.  Suppose $L \rightarrow\infty$ and $L/n\rightarrow 0$. Then, 
For any sequence of real numbers $\{a_n\}$ satisfying $\Pr(\Delta_n^{(1)} \leq a_n ) \rightarrow 1$ and $a_n\rightarrow 0$, on the event $A_n:=\{\Delta_n^{(1)} \leq a_n\}$, we have the following:
%begin{align*}
%	& d_c (  X_n /\sqrt{n} , \tau_L/\sqrt{n} ) = \bigO( d^{1/8} a_n^{1/2}   ) +  \bigO(d^{7/4}n^{\frac{9}{2p}-\frac{1}{2}}m^2).
%	\end{align*} 
\begin{itemize}
\item If $\{x_i\}$ satisfies \cref{asm:moment}(a), then uniformly for all $\omega\in A_n$,
\begin{equation}\label{eq:thm421}	
%\begin{aligned}
d_c \left(P, Q \right)= \begin{cases}
\bigO\left(a_n + \min\left\{d^{\frac{7}{4}}n^{\frac{9}{2p}+(2-\frac{3}{p})\frac{1}{\chi}-\frac{1}{2}}, d^{\frac{5}{6}}n^{\frac{2}{3r}-\frac{1}{3}}(\log n)^{\frac{1}{3}}\right\}\right),&\textrm{ under \ref{asm:dependence}(a)},\\
\bigO\left(a_n+\min\left\{d^{\frac{7}{4}}n^{\frac{9}{2p}-\frac{1}{2}}(\log n)^2, d^{\frac{5}{6}}n^{\frac{2}{3p}-\frac{1}{3}}(\log n)^{\frac{1}{3}}\right\}\right),&
\textrm{ under \ref{asm:dependence}(b)}.
\end{cases}
%\end{aligned}
\end{equation}
\item If $\{x_i\}$ satisfies \cref{asm:moment}(b), then uniformly for all $\omega\in A_n$,
%	\[ d_c\left(\frac{1}{\sqrt{n}}X_n,Y\right) = \bigO(d^{7/4}n^{-1/2}(\log n)^3m^2), \]
\begin{equation}\label{eq:thm422}	
\begin{aligned}
d_c \left(P, Q \right) = \begin{cases}
\bigO\left(a_n + \min\left\{d^{\frac{7}{4}}n^{\frac{2}{\chi}-\frac{1}{2}}(\log n)^3, d^{\frac{5}{6}}n^{\frac{1}{6\chi}-\frac{1}{3}}(\log n)^{\frac{8}{3}}\right\}\right),&\textrm{ under \ref{asm:dependence}(a)},\\
\bigO\left(a_n+\min\left\{d^{\frac{7}{4}}n^{-\frac{1}{2}}(\log n)^5,d^{\frac{5}{6}}n^{-\frac{1}{3}}(\log n)^{\frac{10}{3}}\right\}\right),&
\textrm{ under \ref{asm:dependence}(b)}.
\end{cases}
\end{aligned}
\end{equation}
% \item If $x_i$ %has finite exponential moment uniformly, 
% satisfies \cref{asm:moment}(c),
% \begin{equation}
%      d_c \left(\frac{X_n}{\sqrt{n}}, \frac{\tau_{n,L}}{\sqrt{n-L+1}}\mid \{x_i\} \right) = \begin{cases}\bigO_\Pr(d^{\frac{1}{8}} a_n^{\frac{1}{2}}+d^{\frac{5}{6}}n^{\frac{1}{6\chi}-\frac{1}{3}}(\log n)^{\frac{8}{3}}),% + \min\{dn^{\frac{1}{r}-\frac{1}{2}}, d^{5/8}n^{\frac{1}{2r}-\frac{1}{4}}\}) 
%      \quad \textrm{under \cref{asm:dependence}(a) with } A>\sqrt{\chi+1},\\
%      \bigO_\Pr(d^{\frac{1}{8}} a_n^{\frac{1}{2}}+d^{\frac{5}{6}}n^{-\frac{1}{3}}(\log n)^{\frac{10}{3}}),% + \min\{dn^{\frac{1}{r}-\frac{1}{2}}, d^{5/8}n^{\frac{1}{2r}-\frac{1}{4}}\}) 
%      \quad \textrm{under \cref{asm:dependence}(b)}.
%      \end{cases}
% \end{equation}
\end{itemize}
\end{theorem}
%\begin{proof}
%See \cref{proof:boot_convex}.
%\end{proof}
Theorem \ref{boot:convex} establishes the asymptotic validity of the multiplier bootstrap for approximating $X_n/\sqrt{n}$ uniformly on all convex sets. The bounds established in the theorem are comprised of two parts. The first part, $O(a_n)$, controls the distance $d_c(Y_n/\sqrt{n},\frac{\tau_{n,L}}{\sqrt{n-L+1}}\mid \{x_i\})$ on $A_n$. The second part on the right hand side of \eqref{eq:thm421} and \eqref{eq:thm422} reflects the convex GA error $d_c(X_n,Y_n)$. According to Lemma \ref{lemma_bootstrap_delta}, the sequence $a_n$ can be chosen as $a_n= c_nd\left(\sqrt{\frac{L}{n}}+\frac{1}{L}\right)$ for $c_n\rightarrow\infty$ arbitrarily slowly. Setting $L$ proportional to $n^{1/3}$, $a_n$ is simplified to $c_ndn^{-1/3}$.

\section{Applications to Statistical Inference of HDNS Times Series}\label{section:application}

\subsection{Example: combined $\mathcal{L}^2$ and $\mathcal{L}^\infty$ inference (GA on high-dimensional convex sets)} \label{section:application:sub2}
This section presents an application of our convex GA results. Consider the following multiple linear regression
\begin{equation}\label{eq:lm}
    y_i=x_i^T\beta+\epsilon_i,\quad x_i=G_i(\mathcal{F}_i) \in \R^{d}, \quad \epsilon_i= H_i(\mathcal{F}_i)\in \R, 
\end{equation} 
where $\E(\epsilon_i \mid x_i) = 0$ for $i = 1, 2, \cdots,n$. We aim to perform statistical inference on the regression coefficients $\beta \in \R^d$ regarding the hypothesis:
\begin{equation}
H_0: \beta=\beta_0,\quad H_a: \beta \neq \beta_0.  \label{eq:hyp}
\end{equation} 
%\textcolor{OliveGreen}{[Do we want to use the bold version to denote the $x_i$'s and $\beta$?]}
Denote by $\hat{\beta}$ the ordinary least squares estimate of $\beta$. Here we implicitly require that $d<n$ to ensure that $\hat\beta$ is well defined. It is intuitively clear that if $\sqrt{n}(\hat{\beta} - \beta_0)$ is small in some norms, then there is evidence in favor of the null hypothesis. In practice, the $\mathcal{L}^2$ and $\mathcal{L}^\infty$ norms are among the most popular choices. On the other hand, it is well-known that $\mathcal{L}^2$ and $\mathcal{L}^\infty$ tests are sensitive to different types of alternatives. In particular, $\mathcal{L}^2$ tests are sensitive to alternatives deviating from the null relatively evenly across dimensions, while $\mathcal{L}^\infty$ tests are powerful against alternatives that are sparse with sharp differences in a few dimensions. To combine the strengths of both tests, there has been a recent surge of interest in combined $\mathcal{L}^2$ and $\mathcal{L}^\infty$ inference for high-dimensional data. See for instance, \citet{fan2015power} and \citet{feng2022asymptotic} among many others. In \citet{feng2022asymptotic}, it is shown that the $\mathcal{L}^2$ and $\mathcal{L}^\infty$ tests are asymptotically independent for i.i.d. high-dimensional Gaussian random vectors under certain conditions. However, due to non-stationarity and time series dependence, in general, such independence does not hold for HDNS time series, which poses great challenges for the inference. To our knowledge, combined $\mathcal{L}^2$ and $\mathcal{L}^\infty$ testing has not been discussed for high-dimensional time series. In this example, we shall tackle the mixed $\mathcal{L}^2$ and $\mathcal{L}^\infty$ testing problem from a convex GA perspective.
We consider the following test statistic
\begin{eqnarray}\label{eq:combined_test}
T := \max\left\{\frac{|\hat{\beta}-\beta_0|_{\infty}\sqrt{n}}{\sqrt{ 2\log d}}, |\hat{\beta}-\beta_0| \sqrt{\frac{n}{d}}\right\}.    
\end{eqnarray}
Elementary calculations yield that for standard  Gaussian vector $Z \in \R^d$, $|Z|$ and $|Z|_{\infty}$ are concentrated around $\sqrt{d}$ and $\sqrt{2\log d}$ respectively. Therefore in \eqref{eq:combined_test} we normalized the $\mathcal{L}^2$ and $\mathcal{L}^\infty$ statistics by $\sqrt{d}$ and $\sqrt{2\log d}$ respectively to roughly balance the magnitudes of these norms. 
Observe that the test statistic $T$ is a convex transformation on $\sqrt{n}(\hat{\beta} - \beta_0) $. In particular, the event $\{T\le x\}$ is equivalent to $\sqrt{n}(\hat{\beta} - \beta_0)$ belonging to the intersection of a high dimensional sphere and a high dimensional hyper-rectangle. %Moreover, $T$ combines the ${\cal{L}}^2$ and ${\cal{L}}^\infty$ norm to capture both extreme and aggregate deviations. The normalizing factors are chosen to balance the magnitudes of these norms: for a standard Gaussian vector $Z \in \R^d$, $\E|Z| = \sqrt{2}\Gamma(\frac{d+1}{2})/\Gamma(\frac{d}{2}) = \bigO(\sqrt{d})$ and $\E|Z|_{\infty} = \bigO(\sqrt{2\log d})$ \citep{Galambos_1987, boucheron2013concentration}. For future refinements, the normalization term can be selected based on the corresponding norms of the bootstrapped sample. We will later show through simulations that the test statistic $T$ also combines the advantages of both norms in practice.
Write 
\begin{equation} 
\Gamma_n := \sqrt{n}(\hat{\beta} - \beta) = \frac{1}{\sqrt{n}} \sum_{i = 1}^n  \left(\frac{1}{n}  X^TX\right)^{-1} x_i \epsilon_i,
\label{eq:sum}
\end{equation}  
where $X^T := [x_1,x_2,\dots,x_n]$. Observe that $\sqrt{n}(\hat{\beta} - \beta)$ is a weighted sum of the $d$-dimensional time series $\{x_i\epsilon_i\}$. Based on \cref{main_convex} and \cref{boot:convex}, we can design a multiplier bootstrap statistic that mimics the behavior of $T$ with high probability under $H_0$. 

To proceed, let $\{\hat{\epsilon}_i\}_{i = 1}^n $ be residuals from the OLS estimation and $\{ B_i \}_{i = 1}^n$ be a sequence of i.i.d standard normal random variables independent from the data. Define 
\begin{align}
    {\hat{z}_i:=\left(\frac{1}{n}  X^TX\right)^{-1}x_i\hat{\epsilon}_i}, \quad \text{and} \quad \hat{G}_n:= \frac{1}{\sqrt{n-L+1}} \sum_{i=1}^{n-L+1}\left(\frac{1}{\sqrt{L}}\sum_{j=i}^{i+L-1}\hat{z}_j \right) B_i. \label{eq:hG}
\end{align} 
Observe that $\hat G_n$ should mimic the behavior of $\sqrt{n}(\hat\beta-\beta)$ according to our discussions in Section \ref{section:bootstrap}. We construct the multiplier bootstrap statistics as
\[	T^B := \max\left\{\frac{|\hat{G}_n|_{\infty}}{\sqrt{ 2\log d}}, \frac{|\hat{G}_n|}{\sqrt{d}}\right\}.\]
%Since the error component $\epsilon_i$ is not directly observed in practice, we use the residuals $\hat{\epsilon}_i$ as an approximation. 
The hyper-parameter $L$ in the bootstrap can be chosen by in a data-driven way according to Section \ref{section:simulation}. By generating a large number of replications of $T^B$, we can obtain its $(1 - \alpha)$-th sample quantile. Since the distribution of $T^B$ and $T$ should be equal asymptotically under $H_0$, the $(1 - \alpha)$-th sample quantile of $T^B$ can then serve as the critical value of $T$ at significance level $\alpha$. 

The next theorem states the consistency of the proposed multiplier bootstrap. To facilitate the statement, we define $\tilde{z}_i:= \left(\frac{1}{n} \E(X^\top X)\right)^{-1}x_i{\epsilon}_i$ and consider an intermediate version of $\Gamma_n:=\frac{1}{\sqrt{n}}\sum_{i=1}^n z_i$ based on $\tilde{z}_i$ as \[  \tilde{\Gamma}_n:=\frac{1}{\sqrt{n}}\sum_{i=1}^n \tilde{z}_i.\] 
Similar to $\hat{G}_n$, we define 
\[\tilde{G}_n :=\frac{1}{\sqrt{n-L+1}} \sum_{i=1}^{n-L+1}\left(\frac{1}{\sqrt{L}}\sum_{j=i}^{i+L} \tilde{z}_j\right)  B_i,\] 
where $B_i$'s are i.i.d.~standard normal random variables independent of the data.

{\begin{theorem} \label{convex:consistency}
Suppose $\{x_i\}$ and $\{x_i\epsilon_i\}$ satisfy \cref{asm:moment}(a) and \cref{asm:dependence}(b) for some $p > 4$. Assume that the smallest eigenvalues of $\E\left(\frac{1}{n}X^\top X\right)$ and $\Cov\left(\tilde{\Gamma}_n\right)$ are bounded below by $\lambda_*>0$. If $L \rightarrow\infty$ with $L/n\rightarrow 0$ and $d^2/n\rightarrow 0$, then we have
\[  \Delta_n  := \left| \Cov\left[ \tilde{\Gamma}_n \right]- \Cov\left[ \tilde{G}_n | \{(x_i, y_i)\}_{i=1}^n \right] \right|_F  =  \bigO_{\pr}( d \sqrt{L/n}  + d /L ) .\]
%Furthermore, define the event $A_n = \{ \Delta_n \leq  (d \sqrt{L/n} + d/L  )h_n \}$, where $h_n \to \infty$ at an arbitrarily slow rate, then $\Pr(A_n) = 1 - o(1)$, and on the event $A_n$, 
Furthermore, %for any sequence $h_n \to \infty$ at an arbitrarily slow rate, 
we have under $H_0$,
\begin{align*}
& \quad \sup_{x \in \R}\Big|\Pr(T \le x) - \Pr(T^B \le x \mid \{x_i, y_i\}_{i=1}^n)\Big| \le  d_c(\Gamma_n, \hat{G}_n \mid \{(x_i, y_i)\}_{i=1}^n) \\
%& = \bigO_\pr\left(d^{\frac{5}{8}+\frac{1}{2p}}n^{\frac{1}{p}-\frac{1}{4}}+d^{\frac{7}{4}}n^{\frac{9}{2p}-\frac{1}{2}}(\log n)^2+d^{\frac{5}{8}}\left(\left(\frac{L}{n}\right)^{\frac{1}{4}} +\left(\frac{1}{L}\right)^{\frac{1}{2}}\right)h_n^{\frac{1}{2}} + d^{\frac{11}{8}}n^{-\frac{1}{4}}\right).
& =  \bigO_\pr\left(d^{\frac{7}{8}}n^{\frac{1}{p}-\frac{1}{4}}+d^{\frac{7}{4}}n^{\frac{9}{2p}-\frac{1}{2}}(\log n)^2+d\left(\sqrt{\frac{L}{n}}
+\frac{1}{L}\right)\right).
\end{align*} 
\end{theorem}} 
%\textcolor{OliveGreen}{(Note: right now we are assuming $p$-th order finite moment with the exponential decaying dependence measure, and we are using the first term of the error bound in \cref{main_convex}. Changing this to the latter term doesn't affect the result, as the $d^{\frac{7}{8}}n^{\frac{1}{p}-\frac{1}{4}}$ already requires $d \sim o(n^{\frac{2}{7}})$).}

%\textcolor{OliveGreen}{Question: Should we use $\bigO$ or $\bigO_\pr$ here?}
%\begin{proof}
%See \cref{proof-convex:consistency}. 
%\end{proof}

\begin{remark}\label{rem:combined_norm}
To optimize the order, one can choose $L = \bigO(n^{\frac{1}{3}})$, which reduces the corresponding term $\sqrt{\frac{L}{n}} +\frac{1}{L}$ to $n^{-\frac{1}{3}}$. %The current result allows for $d \sim o(n^{\frac{2}{7}})$ when $p \to \infty$. 
\end{remark}

\subsection{Example: Thresholded Inference (Wasserstein GA)} \label{section:r2} In this section, we delve into the same regression model and inference problem as shown in \cref{eq:lm} and \cref{eq:hyp}. As a different approach, we consider a soft threshold transformation of the estimated regression coefficient as $\hat{\beta}_j^{ST} := \delta_{\lambda_i }(\hat{\beta}_j ) $, where 
\[ \delta_{\lambda}(x) = \begin{cases} \text{sgn}(x) \cdot(|x| - \lambda), & \text{ if } |x| \geq \lambda, \\
0, & \text{ otherwise }. \end{cases} \]
Observe that the function $\delta_\lambda(x)$ is not convex in $x$ and therefore the convex GA techniques are not directly useful here. The threshold $\lambda_j$ is chosen to be \[ \lambda_j =  \hat{\sigma}_j\sqrt{\frac{2\log(n)}{n}}, \] where $\hat{\sigma}_j$ is the standard deviation estimator of $\hat{\beta}_j$. In practice, $\hat{\sigma}_j$ can be estimated through bootstrap simulation of $\hat{\beta}$, independent of the bootstrap simulation for the test statistic.

The soft threshold transformation has been widely applied in statistics, signal processing and many other fields for purposes such as denoising and dimension reduction. For example, it is commonly employed for denoising the empirical wavelet coefficients \citep{Donoho_Johnstone_1995}, which pushes the coefficients towards the origin by certain thresholds. Such wavelet thresholding has proved to be highly effective in a range of applications, including audio-noise removal \citep{Stephane_2009}, image enhancement \citep{Wang_Wu_Castleman_2023}, and other related domains. In particular, thresholding estimators achieve high efficiency when the underlying signal is sparse. On the other hand, however, statistical inference under thresholding has rarely been addressed for time series and other stochastic processes. In this section, we shall tackle the latter inference problem from our Wasserstein GA theory and the associated bootstrap algorithm. 

First, following the threshold transformation, we choose the test statistic as
\[ T := \sqrt{n} \left| \hat{\beta}^{ST} - \beta_0 \right|_{\infty}  .\] 
Note that we chose the $\mathcal{L}^\infty$ norm in $T$. Asymptotic theory for any $\mathcal{L}^p$ norm ($p\ge 1$) is similar. Correspondingly, we shall design a multiplier bootstrap statistic $T^{B}$ that converges in distribution to $T$ with high probability under $H_0$. Define the bootstrap statistic
%\[T^{B} : = \sqrt{n} \max_{1 \leq j \leq d} \left| \delta_{\lambda_j}(\beta_{0,j} + \hat{G}_{n,j}/\sqrt{n}) - \beta_{0,j}  \right|,\]
\[T^B := \sqrt{n}\left|\left(\beta_0 + \hat{G}_n/\sqrt{n}\right)^{ST} - \beta_0\right|_{\infty},\]
where $\hat{G}_n$ is defined the same way as in \cref{eq:hG}. The following theorem establishes the asymptotic validity of the multiplier bootstrap for thresholded inference using the Wasserstein GA and bootstrap theory established in Section \ref{sec:WGA}.
%Compared to the convex transformation, the test statistic $T$ employed in this example is much more complicated. Due to this added layer of complication, we need \cref{main_borel} to prove its consistency. 

\begin{theorem}\label{consistency_coupling}
Suppose $x_i$ and $Y_i=x_i\epsilon_i$ satisfy \cref{asm:moment}(a) and \cref{asm:dependence}(b) for $p > 4$. Also suppose that the smallest eigenvalue of $\E\left(\frac{1}{n}X^\top X\right)$ and $\Cov\left(\tilde{\Gamma}_n\right)$ are bounded below by $\lambda_*>0$. If $L \rightarrow\infty$ with $L/n\rightarrow 0$ and $d^2/n\rightarrow 0$, then we have 
%\[  \Delta_n  := \left|\Cov\left( \tilde{\Gamma}_n \right)- \Cov\left( \tilde{G}_n \mid \{(x_i,y_i)\}_{i=1}^n \right) \right|_{\max} =  \bigO_{\pr}( d^{4/p} \sqrt{L/n} +L/n + 1 /L ) .\]
\[  \Delta_n  := \left| \Cov\left[ \tilde{\Gamma}_n \right]- \Cov\left[ \tilde{G}_n | {\{(x_i, y_i)\}_{i=1}^n} \right] \right|_F  =  \bigO_{\pr}( d \sqrt{L/n}  + d /L ) .\]
Furthermore, %for any sequence $h_n \to \infty$ at an arbitrarily slow rate, 
we have under $H_0$,
%Furthermore, define the event $A_n = \{ \Delta_n \leq  (d^{4/p} \sqrt{L/n} +L/n + 1 /L )h_n \}$, where $h_n \to \infty$ at an arbitrarily slow rate, then $\Pr(A_n) = 1 - o(1)$, and on the event $A_n$,
% \begin{align*}
% & | \Gamma_n - \hat{G}_n | = \bigO_{\pr} \left(d^{\frac{3}{2}}n^{\frac{2}{p} -\frac{1}{2}} + d\log(n){n^{\frac{1}{r}-\frac{1}{2}}} + d^{\frac{5}{4}}\left(d^{\frac{2}{p}}\left(\frac{L}{n}\right)^{\frac{1}{4}} + \left(\frac{L}{n}\right)^{\frac{1}{2}}
% +\left(\frac{1}{L}\right)^{\frac{1}{2}}\right)h_n^{\frac{1}{2}}\right),
% \end{align*} 
%where $r$ is given in \cref{main_borel}.
\begin{align*}
& |T - T^B| = \bigO_\pr\left(d^{\frac{3}{2}}n^{\frac{2}{p} -\frac{1}{2}} + d\log(n){n^{\frac{1}{p}-\frac{1}{2}}} + d\left(\sqrt{\frac{L}{n}}
+\frac{1}{L}\right)\right).
\end{align*} 
\end{theorem}
%\begin{proof}
%See \cref{proof-consistency_coupling}. 
%\end{proof}

\begin{remark}\label{rem:threshold}
To optimize the order, one can choose $L = \bigO(n^{\frac{1}{3}})$, which reduces the corresponding term $\sqrt{\frac{L}{n}}
+\frac{1}{L}$ to $n^{-\frac{1}{3}}$. When $p \to \infty$, the upper bound is asymptotically negligible as long as $d = o(n^{1/3})$. 
\end{remark}

%\newpage
\section{Simulations}\label{section:simulation}
%\textcolor{OliveGreen}{Shiqi: will check whether to add a subsection or remark, talking about the choice of window size $m$.\\}
This section performs finite sample Monte Carlo  simulations of the examples considered in Section \ref{section:application} under different non-stationary time series models. We select models based on various structure complexities, including time-varying MA/AR process and piece-wise locally stationary time series with normal/heavy-tail innovations. Each subsection provides accuracy and power analysis for the bootstrap-based tests described in \cref{section:application}, respectively.

In terms of the window size $L$ involved in the bootstrap in \cref{section:application}, we choose it via the plug-in block size selection method as introduced in \citet{zhou2020frequency}. %Following similar arguments in \citet{zhou2020frequency}, one can show that the selected window size minimizes the asymptotic MSE of the corresponding long-run variance estimators for the associated time series. 
%The detailed proofs are beyond the scope of this paper and 
We refer readers to \citet{zhou2020frequency} for the detailed description of the method and more discussions. As demonstrated in the simulation results, the plug-in method performs reasonably well in both accuracy and power analysis.

\subsection{Examples of HDNS time series} \label{subsec:model}
For each model below, let $x_{i} \in \R^d$ be the $i$-th time index entry, and $\{\xi_i\}_{i=1}^n \in \R^d$ be i.i.d. standard multivariate Gaussian vectors. 
\begin{enumerate}[label=(M\arabic*)]
    \item  \label{m1}  Define generating function  \[x_i = 0.6 \cos(2\pi i /n ) x_{i -1} + \xi_{i}. \]
    %\begin{remark}
    Model \ref{m1} is similar to an auto-regressive (AR) model with a time-dependent AR coefficients, which is the simplest non-stationary time series model we consider. Each coordinate is generated independently. Due to its simplicity and resemblance to an AR process, we use model \ref{m1} as the benchmark for power analysis hereafter.    
    %\end{remark}
    
    \item  \label{m2} Let $A$ be a $d \times d $ matrix with diagonal entries being $1$, subdiagonal and superdiagonal entries being $1/5$, and $0$ elsewhere. Define generating function  \[  x_i =  0.6 \cos(2\pi i /n )A x_{i-1}  + \xi_i. \]  %\begin{remark}
    Model \ref{m2} is a time-varying vector AR(1) model with cross-sectional dependence. The model is included to illustrate the performance when the coordinates of the random vectors are  correlated. 
   %\end{remark}
    
    \item \label{m3} Define generating function   \[ x_i = [0.25 \cos(2\pi i/n) \Ind_{ [0,0.75 n) }(i) + (i/n - 0.3) \Ind_{ [0.75 n,n] }(i)] x_{i-1} + \xi_{i}. \]
    
    \item \label{m4} Define generating function  
    \begin{equation*}
    \begin{split}
    x_i  =  \left[0.25 \cos(2\pi i/n) \Ind_{ [0,0.25 n) }(i) + (i/n - 0.6) \Ind_{ [0.25 n ,0.6 n) }(i) \right. \qquad \\
    \left.\qquad + 0.6 e^{-50(i/n - 0.75)^2} \Ind_{ [0.6 n , n] }(i) \right] x_{i-1}  +  \xi_{i} .        
    \end{split}
    \end{equation*}
    
    %\begin{remark}
    Model \ref{m3} and \ref{m4} are known as piece-wise locally stationary time series models. Unlike the previous models constructed from continuously time varying data generating mechanism, \ref{m3} and \ref{m4} experience break point(s) in their data-generating mechanism.
    %\end{remark}
    %\[G_{i}(\FF_i) := 0.6 \cos(2\pi i /n ) G_{i-1}(\FF_{i - 1}) + t_{i}. \]
    \item \label{m5} Let $t_{i} \in \R^d$ follows multivariate student's $t$ distribution with degrees of freedom $\nu = 5$, location parameter $\boldsymbol{\mu} = \boldsymbol{0}$, and $\boldsymbol{\Sigma} = \boldsymbol{I}_d$. Define generating function \[x_{i}:= 0.6 \cos(2\pi i /n ) x_{i-1}  + t_{i}/\sqrt{5/3}, \] 
    where $\sqrt{5/3}$ is the standard deviation of each entry of $t_i$.
    %\begin{remark} 
    Model \ref{m5} is a non-stationary time series where the innovations follow a heavy-tail distribution. 
    %\end{remark}
\end{enumerate}

Throughout this section, we generate the data using sample size $n = 500$. The bootstrap size is fixed at $B = 1000$, and the simulated rejection rates are computed under $2000$ repetitions.

\subsection{Example: Combined $\mathcal{L}^\infty$ and $\mathcal{L}^{2}$ Inference} In this subsection, we present simulated Type I error and power for methodologies described in \cref{section:application:sub2}. Following the linear model defined in \cref{eq:lm}, we generate $x_i := G_i(\FF_i)$ as predictors based on \ref{m1}-\ref{m5}. The error term of the linear model is defined as  \[ \epsilon_i := \mathcal{H}_i({\mathcal{F}_i})= [14 (i /n)^2(1 - i/n)^2 - 0.5] \epsilon_{i - 1} + \eta_i, \] where $\eta_i$'s are i.i.d. standard normal random variables independent of $\{x_i\}_{i = 1}^n$. Let $\beta_0 = \mathbf{1}$, where $\mathbf{1} = [1, 1,\dots,1]^\top$ is the vector of all ones in $\R^d$. We aim to test the hypothesis 
\[H_0: \beta=\beta_0,\quad H_a: \beta\neq \beta_0. \] as specified in \cref{eq:hyp}.

\subsubsection{Simulation Under Null}
We generate the data under $H_0$ by taking $\beta = \beta_0 = \mathbf{1}$. As shown in Table \ref{tab:convex}, the simulated Type I error aligns reasonably well with the significance level $\alpha$ in all cases for $d\le 25$, where the highest simulated dimensionality 25 is selected to be at the order of $\sqrt{n}$.

\begin{table}[ht]
\centering
\caption{\label{tab:convex} \small Simulated Type I error of combined testing in $\%$ for different dimension $d$ and model \ref{m1} - \ref{m5}.}
%\begin{adjustbox}{width=\columnwidth}
\renewcommand{\arraystretch}{1.05}
\begin{tabular}{|>{\centering}m{0.05\textwidth}|*{5}{m{0.05\textwidth}m{0.05\textwidth}|}}
\hline 
 &\multicolumn{2}{c|} {M1}& \multicolumn{2}{c|} {M2} & \multicolumn{2}{c|} {M3} & \multicolumn{2}{c|} {M4}& \multicolumn{2}{c|} {M5} \\ 
\hline 
\(d\) & \multicolumn{2}{c|}{\(5 \%\)\quad\(10 \%\)}  & \multicolumn{2}{c|}{\(5 \%\)\quad\(10 \%\)} & \multicolumn{2}{c|}{\(5 \%\)\quad\(10 \%\)}  & \multicolumn{2}{c|}{\(5 \%\)\quad\(10 \%\)} & \multicolumn{2}{c|}{\(5 \%\)\quad\(10 \%\)}\\ 
\hline
\(5\) &
\multicolumn{2}{l|}{\(5.10\)\quad\(9.65\)} & 
\multicolumn{2}{l|}{\(5.20\)\quad\(9.60\)} & 
\multicolumn{2}{l|}{\(5.15\)\quad\(9.90\)} &  
\multicolumn{2}{l|}{\(5.60\)\quad\(10.80\)} &  
\multicolumn{2}{l|}{\(4.45\)\quad\(9.20\)} \\
\(10\) & 
\multicolumn{2}{l|}{\(4.30\)\quad\(9.00\)} & 
\multicolumn{2}{l|}{\(4.30\)\quad\(8.90\)} & 
\multicolumn{2}{l|}{\(4.85\)\quad\(9.25\)} &  
\multicolumn{2}{l|}{\(5.65\)\quad\(10.45\)} &  
\multicolumn{2}{l|}{\(4.50\)\quad\(9.15\)} \\
\(15\) & 
\multicolumn{2}{l|}{\(4.65\)\quad\(9.50\)} & 
\multicolumn{2}{l|}{\(4.70\)\quad\(9.80\)} & 
\multicolumn{2}{l|}{\(4.85\)\quad\(9.40\)} &  
\multicolumn{2}{l|}{\(5.45\)\quad\(11.05\)} &  
\multicolumn{2}{l|}{\(4.35\)\quad\(10.20\)} \\
\(20\) & 
\multicolumn{2}{l|}{\(5.65\)\quad\(10.45\)} & 
\multicolumn{2}{l|}{\(5.65\)\quad\(10.85\)} & 
\multicolumn{2}{l|}{\(5.25\)\quad\(10.00\)} &  
\multicolumn{2}{l|}{\(5.90\)\quad\(11.10\)} &  
\multicolumn{2}{l|}{\(5.65\)\quad\(12.20\)} \\
\(25\) & 
\multicolumn{2}{l|}{\(3.75\)\quad\(8.35\)} & 
\multicolumn{2}{l|}{\(4.70\)\quad\(9.30\)} & 
\multicolumn{2}{l|}{\(5.05\)\quad\(10.35\)} &  
\multicolumn{2}{l|}{\(6.25\)\quad\(12.70\)} &  
\multicolumn{2}{l|}{\(5.40\)\quad\(11.00\)} \\
\hline
\end{tabular}
%\end{adjustbox}
%\caption{\label{tab:convex} \small Simulated Type I error for different dimension $d$ and model \ref{m1} - \ref{m5}.}
\end{table}

\subsubsection{Simulation Under the Alternative}
For power analysis, we consider two different types of alternatives. The first type is \[\beta = \beta_0 + \delta \cdot e_1 ,\] where $e_i$ is the $i$-th standard basis of the $\R^d$ vector space. In this case, the hypothesized $\beta_0$ and the true parameter $\beta$ only differ in one coordinate (by $\delta$), and we shall refer to this as the sparse case. The second type is \[\beta = (1 + \delta ) \cdot \beta_0,\] and shall be referred to as the uniform case since the deviation of $\beta$ from $\beta_0$ is uniform across all coordinates. 

%\newpage
\begin{figure}[H]
\centering
\includegraphics[width = \linewidth, keepaspectratio]{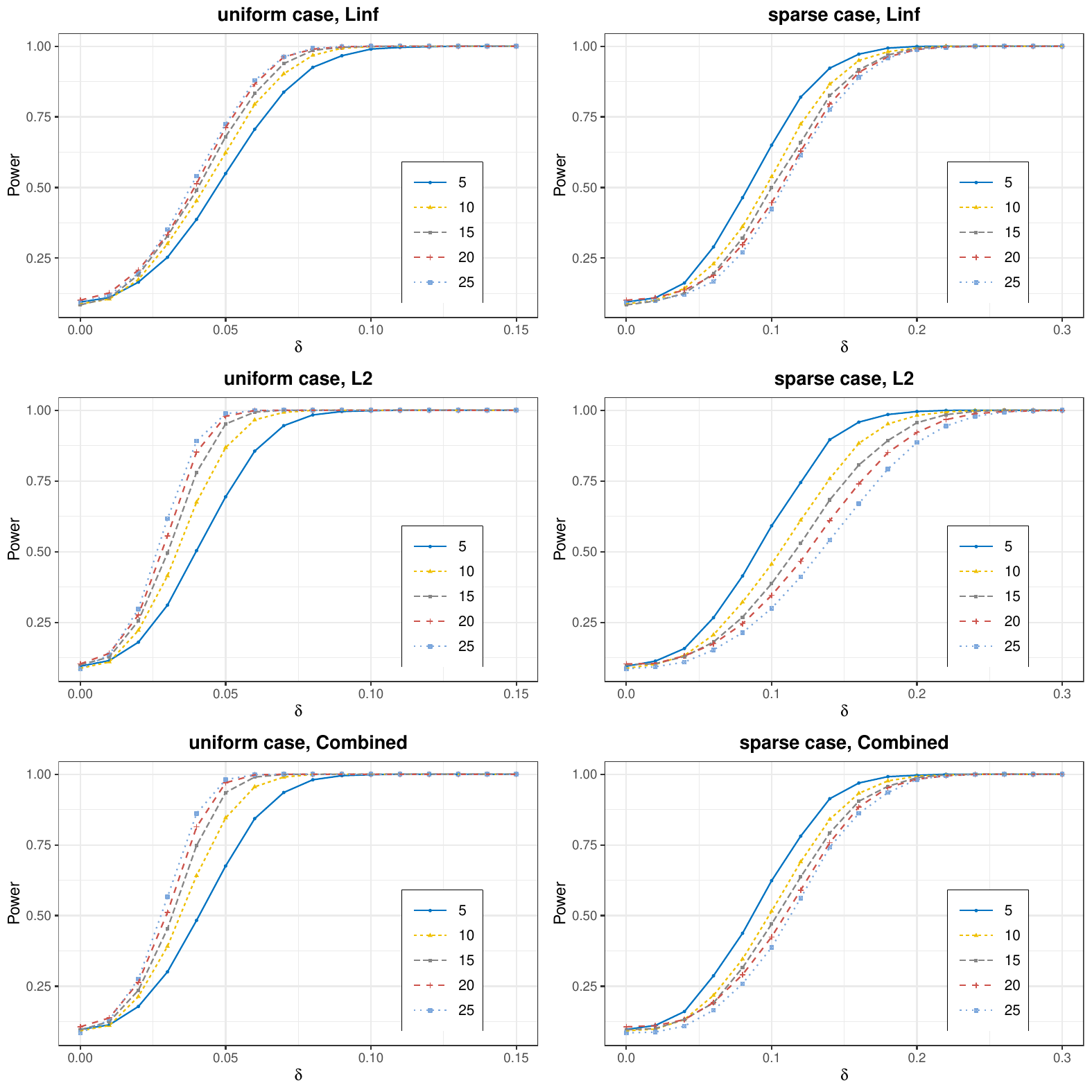}
\caption{Simulated power versus the size of deviation $\delta$ under specific alternatives and test statistics.}
\label{fig:convex_power_v1}
\end{figure}

We plot the simulated rejection rates against the size of deviation $\delta$ under the same conditions as in the null case for model \ref{m1}.
% \begin{figure}[H]
% \centering
% \includegraphics[height = 0.6\textheight, keepaspectratio]{PowerComparison_Dimension_DiffScale.pdf}
% \caption{Simulated power versus the size of deviation $\Delta$ under specific alternatives and test statistics. \RED{JY: Could the font size in the figure be larger?}}
% \label{fig:convex_power_v1}
% \end{figure}
 \cref{fig:convex_power_v1} compares the simulated power for the uniform case (1st column) against the sparse case (2nd column) as the dimension $d$ increases. The first 2 rows show the power performances of the $\mathcal{L}^2$ and $\mathcal{L}^\infty$ tests, respectively; while the last row presents the power of the proposed test statistic $T$. Comparing the uniform case and sparse case, we can see that there is a reversal in performance relative to the dimension size. For the uniform case, the power tends to increase as the dimension increases, while for the sparse case, the opposite tendency applies. This could be attributed to the different nature of the alternatives. Considering that the uniform alternative deviates from the null for each entry, the increasing rejection rate becomes intuitively correct, as the deviation gets enhanced with a larger dimension. On the other hand, the sparse alternative differs from the null only for the first entry; such deviation is `diluted' as the dimension increases, which may lead to a smaller rejection rate.

\begin{figure}
\centering
\includegraphics[width = \linewidth,keepaspectratio]{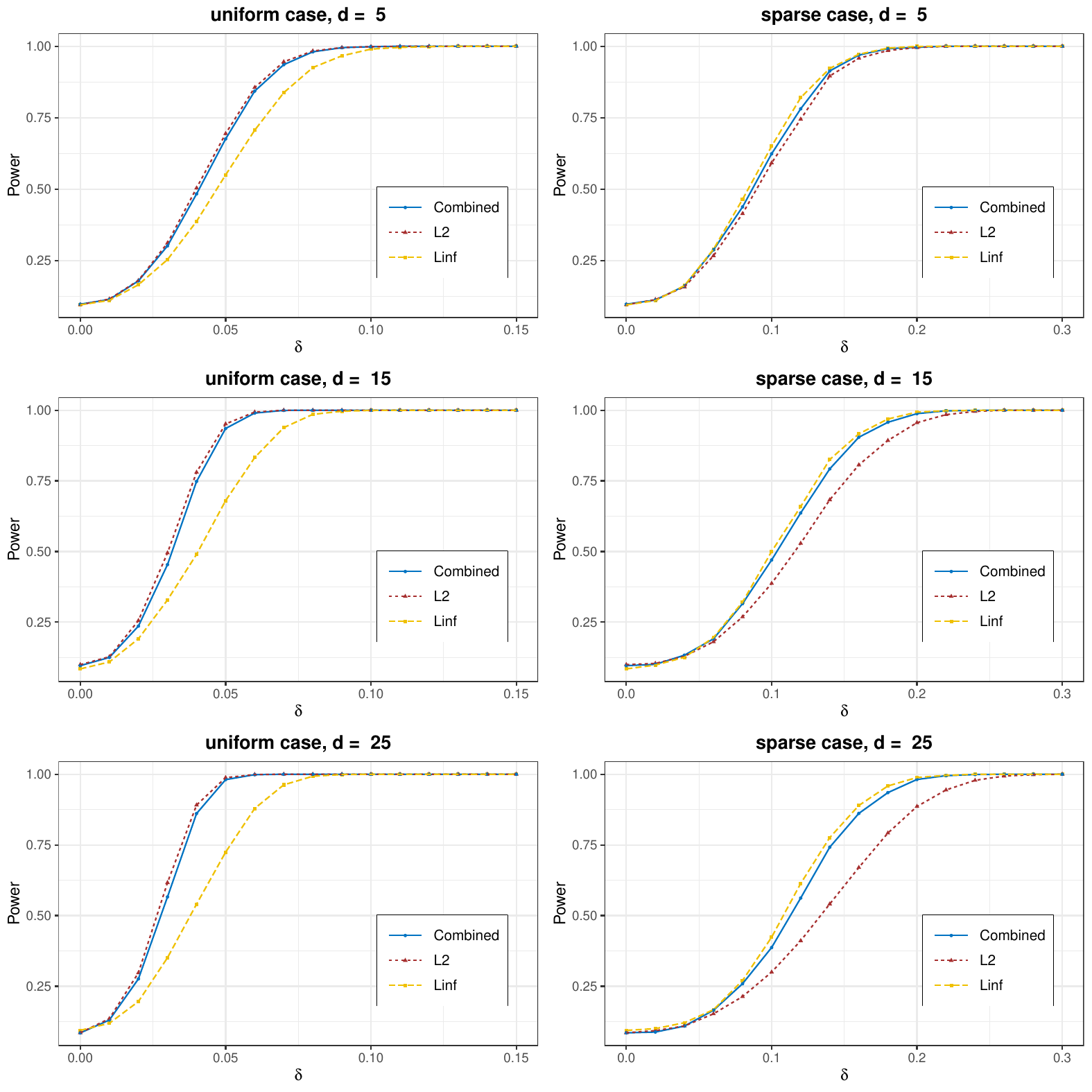}
\caption{Simulated power versus the size of deviation $\delta$ under specific alternatives and dimensions.}
\label{fig:convex_power_v2}
\end{figure}

% \begin{figure}
% \centering
% \includegraphics[height=\textheight,keepaspectratio]{PowerComparison_Method_DiffScale.pdf}
% \caption{Simulated power versus the size of deviation $\Delta$ under specific alternatives and dimensions. \RED{JY: Could the font size in the figure be larger?}}
% \label{fig:convex_power_v2}
% \end{figure}
\cref{fig:convex_power_v2} further compares the power performances of the $\mathcal{L}^2$, $\mathcal{L}^\infty$ and the combined  tests under various choices of $d$ and types of alternatives. The rows of the power plots correspond to different dimensions $d = 5, 15, 25$, respectively. Observe that the $\mathcal{L}^\infty$ statistic performs better in the sparse case and the $\mathcal{L}^2$ statistic performs better in the uniform case. Most notably, the proposed statistic $T$, a convex combination of $\mathcal{L}^2$ and $\mathcal{L}^\infty$ statistics, is a close second in all cases. Thus, the convex combination of both statistics yields a robust  power performance in finite samples. 

\subsection{Example: Thresholded Inference }\label{sec:simu_threshold}
In this subsection, we present the finite sample performance of the threshold statistic as described in \cref{section:r2}. The simulation condition is the same as the previous example except for the value of $\beta$ and $\beta_0$. Let $\beta_0 = [\mathbf{0}_{r}, \mathbf{1}_{d - r} ] $, where $\mathbf{0}_{r}= [0,\dots,0]^\top \in \R^r$, $\mathbf{1}_{d - r} = [1,\dots,1]^\top \in \R^{d - r} $. Using this construction, by changing the value of $r$, we aim to illustrate the ability of our test under different sparsity of $\beta_0$.

\subsubsection{Simulation Under Null.}
\begin{table}[H]
\centering
\caption{\label{tab:thres} \small Simulated Type I error of threshold testing in $\%$ for different dimension $d$ and model \ref{m1} - \ref{m5}.}
%\begin{adjustbox}{width=\columnwidth}
\renewcommand{\arraystretch}{1.05}
\begin{tabular}{|>{\centering}m{0.05\textwidth}|*{5}{m{0.05\textwidth}m{0.05\textwidth}|}}
\hline 
 &\multicolumn{2}{c|} {M1}& \multicolumn{2}{c|} {M2} & \multicolumn{2}{c|} {M3} & \multicolumn{2}{c|} {M4}& \multicolumn{2}{c|} {M5} \\ 
\hline 
\(d\) & \multicolumn{2}{c|}{\(5 \%\)\quad\(10 \%\)}  & \multicolumn{2}{c|}{\(5 \%\)\quad\(10 \%\)} & \multicolumn{2}{c|}{\(5 \%\)\quad\(10 \%\)}  & \multicolumn{2}{c|}{\(5 \%\)\quad\(10 \%\)} & \multicolumn{2}{c|}{\(5 \%\)\quad\(10 \%\)}\\ 
\hline
& \multicolumn{10}{c|}{$r = 0$} \\
\hline
\(5\) &
\multicolumn{2}{l|}{\(4.80\)\quad\(10.15\)} & 
\multicolumn{2}{l|}{\(5.20\)\quad\(9.45\)} & 
\multicolumn{2}{l|}{\(5.65\)\quad\(10.10\)} &  
\multicolumn{2}{l|}{\(6.15\)\quad\(11.60\)} &  
\multicolumn{2}{l|}{\(3.95\)\quad\(8.55\)} \\
\(10\) & 
\multicolumn{2}{l|}{\(4.70\)\quad\(9.70\)} & 
\multicolumn{2}{l|}{\(4.30\)\quad\(9.25\)} & 
\multicolumn{2}{l|}{\(5.00\)\quad\(9.75\)} &  
\multicolumn{2}{l|}{\(5.85\)\quad\(10.00\)} &  
\multicolumn{2}{l|}{\(5.20\)\quad\(9.45\)} \\
\(15\) & 
\multicolumn{2}{l|}{\(4.75\)\quad\(8.50\)} & 
\multicolumn{2}{l|}{\(4.55\)\quad\(8.80\)} & 
\multicolumn{2}{l|}{\(4.10\)\quad\(8.70\)} &  
\multicolumn{2}{l|}{\(5.00\)\quad\(10.65\)} &  
\multicolumn{2}{l|}{\(4.80\)\quad\(9.85\)} \\
\(20\) & 
\multicolumn{2}{l|}{\(5.25\)\quad\(10.35\)} & 
\multicolumn{2}{l|}{\(5.40\)\quad\(10.60\)} & 
\multicolumn{2}{l|}{\(5.20\)\quad\(10.10\)} &  
\multicolumn{2}{l|}{\(6.00\)\quad\(11.45\)} &  
\multicolumn{2}{l|}{\(5.95\)\quad\(11.85\)} \\
\(25\) & 
\multicolumn{2}{l|}{\(4.00\)\quad\(9.70\)} & 
\multicolumn{2}{l|}{\(4.85\)\quad\(10.65\)} & 
\multicolumn{2}{l|}{\(5.80\)\quad\(10.90\)} &  
\multicolumn{2}{l|}{\(6.30\)\quad\(12.85\)} &  
\multicolumn{2}{l|}{\(5.35\)\quad\(12.05\)} \\
%\hline
%& \multicolumn{10}{c|}{\text{Proportion of $\beta$'s coordinate being zero is $0.2$}} \\
% \hline
% \(5\) &
% \multicolumn{2}{l|}{\(5.15\)\quad\(10.00\)} & 
% \multicolumn{2}{l|}{\(5.20\)\quad\(9.65\)} & 
% \multicolumn{2}{l|}{\(5.50\)\quad\(9.85\)} &  
% \multicolumn{2}{l|}{\(6.20\)\quad\(11.75\)} &  
% \multicolumn{2}{l|}{\(4.10\)\quad\(8.45\)} \\
% \(10\) & 
% \multicolumn{2}{l|}{\(4.65\)\quad\(9.70\)} & 
% \multicolumn{2}{l|}{\(4.10\)\quad\(9.60\)} & 
% \multicolumn{2}{l|}{\(5.05\)\quad\(9.60\)} &  
% \multicolumn{2}{l|}{\(5.65\)\quad\(10.20\)} &  
% \multicolumn{2}{l|}{\(5.20\)\quad\(9.85\)} \\
% \(15\) & 
% \multicolumn{2}{l|}{\(4.55\)\quad\(8.40\)} & 
% \multicolumn{2}{l|}{\(4.40\)\quad\(8.75\)} & 
% \multicolumn{2}{l|}{\(3.70\)\quad\(8.05\)} &  
% \multicolumn{2}{l|}{\(4.85\)\quad\(10.45\)} &  
% \multicolumn{2}{l|}{\(4.85\)\quad\(10.00\)} \\
% \(20\) & 
% \multicolumn{2}{l|}{\(5.25\)\quad\(10.20\)} & 
% \multicolumn{2}{l|}{\(5.70\)\quad\(10.55\)} & 
% \multicolumn{2}{l|}{\(5.15\)\quad\(9.95\)} &  
% \multicolumn{2}{l|}{\(5.70\)\quad\(11.00\)} &  
% \multicolumn{2}{l|}{\(5.95\)\quad\(12.10\)} \\
% \(25\) & 
% \multicolumn{2}{l|}{\(4.15\)\quad\(9.65\)} & 
% \multicolumn{2}{l|}{\(4.85\)\quad\(10.45\)} & 
% \multicolumn{2}{l|}{\(5.30\)\quad\(10.80\)} &  
% \multicolumn{2}{l|}{\(6.70\)\quad\(12.50\)} &  
% \multicolumn{2}{l|}{\(5.70\)\quad\(11.70\)} \\
\hline
& \multicolumn{10}{c|}{$r = 0.4d$} \\
\hline
\(5\) &
\multicolumn{2}{l|}{\(5.10\)\quad\(9.65\)} & 
\multicolumn{2}{l|}{\(5.25\)\quad\(9.75\)} & 
\multicolumn{2}{l|}{\(5.45\)\quad\(9.85\)} &  
\multicolumn{2}{l|}{\(5.80\)\quad\(11.85\)} &  
\multicolumn{2}{l|}{\(4.30\)\quad\(8.55\)} \\
\(10\) & 
\multicolumn{2}{l|}{\(5.00\)\quad\(9.70\)} & 
\multicolumn{2}{l|}{\(4.60\)\quad\(9.90\)} & 
\multicolumn{2}{l|}{\(4.90\)\quad\(9.55\)} &  
\multicolumn{2}{l|}{\(5.40\)\quad\(10.30\)} &  
\multicolumn{2}{l|}{\(4.75\)\quad\(9.95\)} \\
\(15\) & 
\multicolumn{2}{l|}{\(4.50\)\quad\(8.80\)} & 
\multicolumn{2}{l|}{\(4.20\)\quad\(8.75\)} & 
\multicolumn{2}{l|}{\(4.05\)\quad\(8.45\)} &  
\multicolumn{2}{l|}{\(4.95\)\quad\(10.30\)} &  
\multicolumn{2}{l|}{\(5.05\)\quad\(10.35\)} \\
\(20\) & 
\multicolumn{2}{l|}{\(5.20\)\quad\(9.75\)} & 
\multicolumn{2}{l|}{\(5.65\)\quad\(10.45\)} & 
\multicolumn{2}{l|}{\(5.00\)\quad\(9.90\)} &  
\multicolumn{2}{l|}{\(5.70\)\quad\(10.85\)} &  
\multicolumn{2}{l|}{\(5.70\)\quad\(12.05\)} \\
\(25\) & 
\multicolumn{2}{l|}{\(4.30\)\quad\(9.65\)} & 
\multicolumn{2}{l|}{\(5.05\)\quad\(9.90\)} & 
\multicolumn{2}{l|}{\(5.30\)\quad\(10.75\)} &  
\multicolumn{2}{l|}{\(6.50\)\quad\(12.60\)} &  
\multicolumn{2}{l|}{\(5.55\)\quad\(11.75\)} \\
% \hline
% & \multicolumn{10}{c|}{\text{Proportion of $\beta$'s coordinate being zero is $0.6$}} \\
% \hline
% \(5\) &
% \multicolumn{2}{l|}{\(5.10\)\quad\(9.55\)} & 
% \multicolumn{2}{l|}{\(5.20\)\quad\(9.65\)} & 
% \multicolumn{2}{l|}{\(5.30\)\quad\(9.60\)} &  
% \multicolumn{2}{l|}{\(5.75\)\quad\(11.45\)} &  
% \multicolumn{2}{l|}{\(4.55\)\quad\(8.45\)} \\
% \(10\) & 
% \multicolumn{2}{l|}{\(4.50\)\quad\(9.65\)} & 
% \multicolumn{2}{l|}{\(4.75\)\quad\(8.95\)} & 
% \multicolumn{2}{l|}{\(5.00\)\quad\(9.65\)} &  
% \multicolumn{2}{l|}{\(5.50\)\quad\(10.80\)} &  
% \multicolumn{2}{l|}{\(4.65\)\quad\(9.80\)} \\
% \(15\) & 
% \multicolumn{2}{l|}{\(4.35\)\quad\(9.10\)} & 
% \multicolumn{2}{l|}{\(4.10\)\quad\(9.45\)} & 
% \multicolumn{2}{l|}{\(4.30\)\quad\(8.30\)} &  
% \multicolumn{2}{l|}{\(4.95\)\quad\(10.55\)} &  
% \multicolumn{2}{l|}{\(5.05\)\quad\(10.30\)} \\
% \(20\) & 
% \multicolumn{2}{l|}{\(5.30\)\quad\(9.95\)} & 
% \multicolumn{2}{l|}{\(5.50\)\quad\(10.20\)} & 
% \multicolumn{2}{l|}{\(5.20\)\quad\(10.40\)} &  
% \multicolumn{2}{l|}{\(5.85\)\quad\(10.90\)} &  
% \multicolumn{2}{l|}{\(5.90\)\quad\(11.75\)} \\
% \(25\) & 
% \multicolumn{2}{l|}{\(4.45\)\quad\(9.45\)} & 
% \multicolumn{2}{l|}{\(5.10\)\quad\(9.90\)} & 
% \multicolumn{2}{l|}{\(5.50\)\quad\(10.70\)} &  
% \multicolumn{2}{l|}{\(6.60\)\quad\(12.40\)} &  
% \multicolumn{2}{l|}{\(5.70\)\quad\(11.05\)} \\
\hline
& \multicolumn{10}{c|}{$r = 0.8d$} \\
\hline
\(5\) &
\multicolumn{2}{l|}{\(5.05\)\quad\(9.45\)} & 
\multicolumn{2}{l|}{\(5.15\)\quad\(9.40\)} & 
\multicolumn{2}{l|}{\(5.05\)\quad\(10.10\)} &  
\multicolumn{2}{l|}{\(6.00\)\quad\(11.55\)} &  
\multicolumn{2}{l|}{\(4.50\)\quad\(8.75\)} \\
\(10\) & 
\multicolumn{2}{l|}{\(4.75\)\quad\(9.50\)} & 
\multicolumn{2}{l|}{\(4.80\)\quad\(8.75\)} & 
\multicolumn{2}{l|}{\(5.05\)\quad\(9.75\)} &  
\multicolumn{2}{l|}{\(5.20\)\quad\(10.70\)} &  
\multicolumn{2}{l|}{\(4.60\)\quad\(9.60\)} \\
\(15\) & 
\multicolumn{2}{l|}{\(4.15\)\quad\(9.00\)} & 
\multicolumn{2}{l|}{\(4.55\)\quad\(9.20\)} & 
\multicolumn{2}{l|}{\(4.35\)\quad\(8.85\)} &  
\multicolumn{2}{l|}{\(4.75\)\quad\(10.80\)} &  
\multicolumn{2}{l|}{\(4.95\)\quad\(10.50\)} \\
\(20\) & 
\multicolumn{2}{l|}{\(5.40\)\quad\(9.80\)} & 
\multicolumn{2}{l|}{\(5.40\)\quad\(10.20\)} & 
\multicolumn{2}{l|}{\(4.75\)\quad\(10.60\)} &  
\multicolumn{2}{l|}{\(5.50\)\quad\(11.00\)} &  
\multicolumn{2}{l|}{\(5.80\)\quad\(11.95\)} \\
\(25\) & 
\multicolumn{2}{l|}{\(4.55\)\quad\(9.25\)} & 
\multicolumn{2}{l|}{\(4.80\)\quad\(10.05\)} & 
\multicolumn{2}{l|}{\(5.60\)\quad\(10.55\)} &  
\multicolumn{2}{l|}{\(6.10\)\quad\(12.70\)} &  
\multicolumn{2}{l|}{\(5.65\)\quad\(11.35\)} \\
\hline
& \multicolumn{10}{c|}{$r = d$} \\
\hline
\(5\) &
\multicolumn{2}{l|}{\(5.05\)\quad\(9.45\)} & 
\multicolumn{2}{l|}{\(5.10\)\quad\(9.40\)} & 
\multicolumn{2}{l|}{\(5.00\)\quad\(10.20\)} &  
\multicolumn{2}{l|}{\(6.00\)\quad\(10.80\)} &  
\multicolumn{2}{l|}{\(4.25\)\quad\(9.05\)} \\
\(10\) & 
\multicolumn{2}{l|}{\(4.30\)\quad\(8.75\)} & 
\multicolumn{2}{l|}{\(4.45\)\quad\(8.80\)} & 
\multicolumn{2}{l|}{\(5.05\)\quad\(9.55\)} &  
\multicolumn{2}{l|}{\(5.15\)\quad\(10.65\)} &  
\multicolumn{2}{l|}{\(4.75\)\quad\(9.45\)} \\
\(15\) & 
\multicolumn{2}{l|}{\(4.10\)\quad\(8.45\)} & 
\multicolumn{2}{l|}{\(4.55\)\quad\(8.95\)} & 
\multicolumn{2}{l|}{\(4.00\)\quad\(8.85\)} &  
\multicolumn{2}{l|}{\(4.40\)\quad\(10.90\)} &  
\multicolumn{2}{l|}{\(5.10\)\quad\(10.00\)} \\
\(20\) & 
\multicolumn{2}{l|}{\(5.55\)\quad\(10.00\)} & 
\multicolumn{2}{l|}{\(5.30\)\quad\(10.60\)} & 
\multicolumn{2}{l|}{\(5.05\)\quad\(10.60\)} &  
\multicolumn{2}{l|}{\(5.35\)\quad\(12.05\)} &  
\multicolumn{2}{l|}{\(6.00\)\quad\(12.35\)} \\
\(25\) & 
\multicolumn{2}{l|}{\(4.65\)\quad\(9.45\)} & 
\multicolumn{2}{l|}{\(5.20\)\quad\(10.20\)} & 
\multicolumn{2}{l|}{\(5.60\)\quad\(10.40\)} &  
\multicolumn{2}{l|}{\(6.70\)\quad\(12.90\)} &  
\multicolumn{2}{l|}{\(5.70\)\quad\(11.60\)} \\
\hline
\end{tabular}
%\end{adjustbox}
%\caption{\label{tab:convex} \small Simulated Type I error for different dimension $d$ and model \ref{m1} - \ref{m5}.}
\end{table}

\cref{tab:thres} shows the simulated type I error for the hypothesis testing problem as described in \cref{section:r2}. The simulated Type I error is close to the theoretical level $\alpha$, which aligns with the consistency result offered by \cref{consistency_coupling}. The test appears to perform consistently across different levels of sparsity in the coordinates of $\beta$.

\subsubsection{Simulation Under the Alternative}
The power analysis is performed where $\beta = \beta_0 + [\delta, 0,0,\dots,0]$, while the null hypothesis is $H_0 : \beta = \beta_0$. We consider different sparsity levels with $r = 0, 0.2d, 0.4d, 0.6d, 0.8d$ and $d$, where $r$ represents the number of zero coordinates in $\beta_0$.
\begin{figure}[H]
\centering
\includegraphics[width = \linewidth, keepaspectratio]{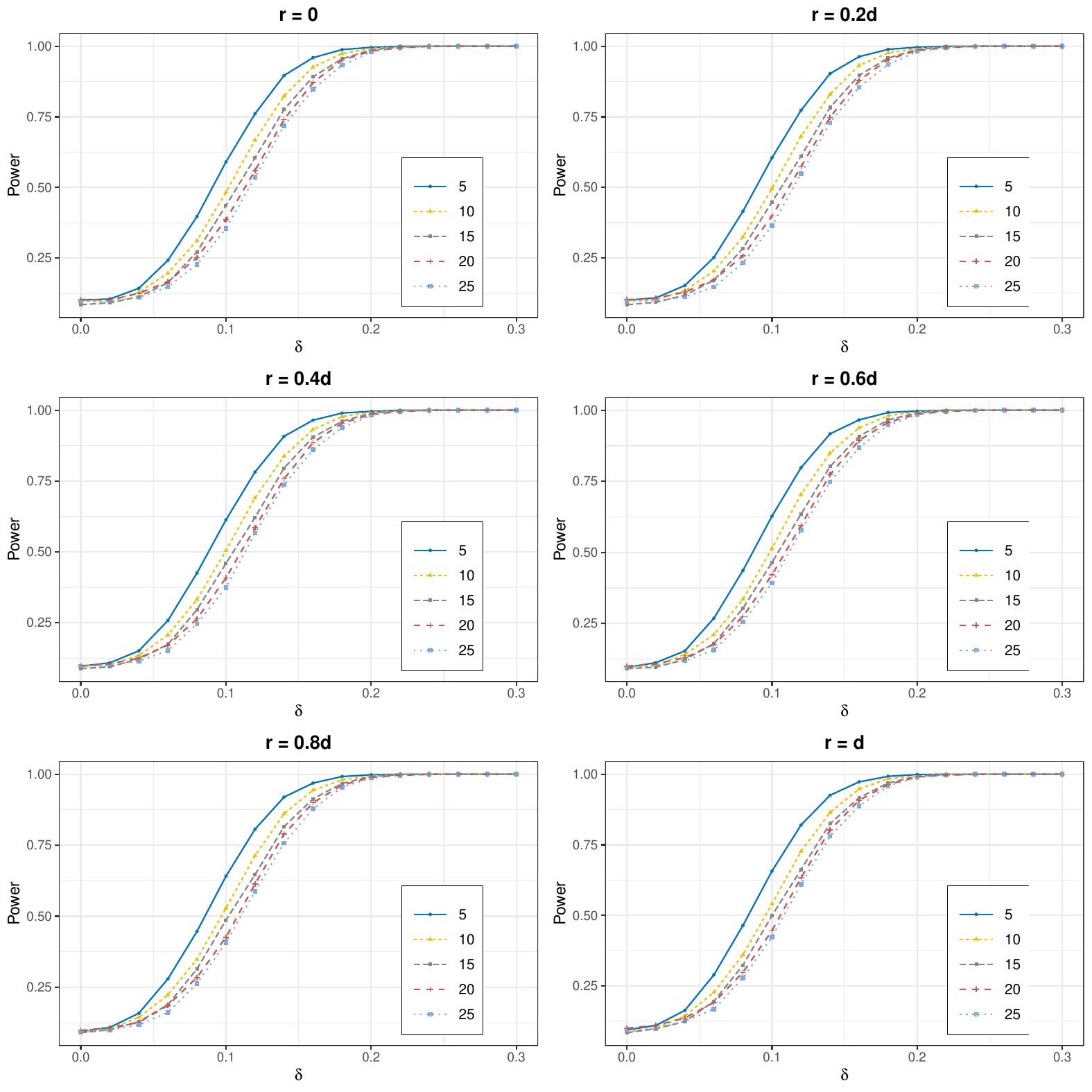}
\caption{Simulated power (y-axis) versus the size of deviation $\delta$ under specific level of sparsity.}
\label{fig:thres_power}
\end{figure}

\cref{fig:thres_power} demonstrates that the rejection rates coverage to $1$ as the size of deviation $\delta$ increases. Additionally, the rejection rates converge roughly at the same rate across all levels of sparsity. %\textcolor{OliveGreen}{Particularly for $r = d$, $\beta_0 = \boldsymbol{0}_d$, the threshold test is designed to detect which coordinates of $\beta$ are not statistically significant.} 

\begin{acks}[Acknowledgments]
The authors would like to thank Yang-Guan-Jian Guo for helpful discussions at the early stage of this project. Zhou's research is supported by NSERC of Canada.
\end{acks}

%% if your bibliography is in bibtex format, uncomment commands:
\bibliographystyle{imsart-nameyear} % Style BST file (imsart-number.bst or imsart-nameyear.bst)
\bibliography{ref}       % Bibliography file (usually '*.bib')

\begin{thebibliography}{39}
% BibTex style file: imsart-nameyear.bst, 2017-11-03
% Default style options (sort=1,type=nameyear).
% Used options (sort=1,type=nameyear).

\bibitem[\protect\citeauthoryear{Bentkus}{1986}]{Bentkus_1986}
\begin{barticle}[author]
\bauthor{\bsnm{Bentkus},~\bfnm{V.}\binits{V.}}
(\byear{1986}).
\btitle{Dependence of the {Berry-Esseen} estimate on the dimension}.
\bjournal{Lithuanian Mathematical Journal}
\bvolume{26}
\bpages{110–114}.
\end{barticle}
\endbibitem

\bibitem[\protect\citeauthoryear{Bentkus}{2003}]{bentkus2003dependence}
\begin{barticle}[author]
\bauthor{\bsnm{Bentkus},~\bfnm{Vidmantas}\binits{V.}}
(\byear{2003}).
\btitle{On the dependence of the {B}erry--{E}sseen bound on dimension}.
\bjournal{Journal of Statistical Planning and Inference}
\bvolume{113}
\bpages{385--402}.
\end{barticle}
\endbibitem

\bibitem[\protect\citeauthoryear{Berkes, Liu and Wu}{2014}]{berkes2014komlos}
\begin{barticle}[author]
\bauthor{\bsnm{Berkes},~\bfnm{Istv{\'a}n}\binits{I.}}, \bauthor{\bsnm{Liu},~\bfnm{Weidong}\binits{W.}} \AND \bauthor{\bsnm{Wu},~\bfnm{Wei~Biao}\binits{W.~B.}}
(\byear{2014}).
\btitle{Koml{\'o}s--Major--Tusn{\'a}dy approximation under dependence}.
\bjournal{The Annals of Probability}
\bvolume{42}
\bpages{794--817}.
\end{barticle}
\endbibitem

\bibitem[\protect\citeauthoryear{Bonnerjee, Karmakar and Wu}{2024}]{bonnerjee2024gaussian}
\begin{barticle}[author]
\bauthor{\bsnm{Bonnerjee},~\bfnm{Soham}\binits{S.}}, \bauthor{\bsnm{Karmakar},~\bfnm{Sayar}\binits{S.}} \AND \bauthor{\bsnm{Wu},~\bfnm{Wei~Biao}\binits{W.~B.}}
(\byear{2024}).
\btitle{Gaussian approximation for nonstationary time series with optimal rate and explicit construction}.
\bjournal{The Annals of Statistics}
\bvolume{52}
\bpages{2293--2317}.
\end{barticle}
\endbibitem

\bibitem[\protect\citeauthoryear{Chang, Chen and Wu}{2024}]{chang2021central}
\begin{barticle}[author]
\bauthor{\bsnm{Chang},~\bfnm{Jinyuan}\binits{J.}}, \bauthor{\bsnm{Chen},~\bfnm{Xiaohui}\binits{X.}} \AND \bauthor{\bsnm{Wu},~\bfnm{Mingcong}\binits{M.}}
(\byear{2024}).
\btitle{Central limit theorems for high dimensional dependent data}.
\bjournal{Bernoulli}
\bvolume{30}
\bpages{712–742}.
\end{barticle}
\endbibitem

\bibitem[\protect\citeauthoryear{Chernozhukov, Chetverikov and Kato}{2015}]{Chernozhukov2015}
\begin{barticle}[author]
\bauthor{\bsnm{Chernozhukov},~\bfnm{Victor}\binits{V.}}, \bauthor{\bsnm{Chetverikov},~\bfnm{Denis}\binits{D.}} \AND \bauthor{\bsnm{Kato},~\bfnm{Kengo}\binits{K.}}
(\byear{2015}).
\btitle{Comparison and anti-concentration bounds for maxima of {G}aussian random vectors}.
\bjournal{Probability Theory and Related Fields}
\bvolume{162}
\bpages{47--70}.
\end{barticle}
\endbibitem

\bibitem[\protect\citeauthoryear{Chernozhukov, Chetverikov and Kato}{2017}]{Chernozhukov_Chetverikov_Kato_2017}
\begin{barticle}[author]
\bauthor{\bsnm{Chernozhukov},~\bfnm{Victor}\binits{V.}}, \bauthor{\bsnm{Chetverikov},~\bfnm{Denis}\binits{D.}} \AND \bauthor{\bsnm{Kato},~\bfnm{Kengo}\binits{K.}}
(\byear{2017}).
\btitle{Central Limit Theorems and Bootstrap in High Dimensions}.
\bjournal{The Annals of Probability}
\bvolume{45}
\bpages{2309–2352}.
\end{barticle}
\endbibitem

\bibitem[\protect\citeauthoryear{Cui and Zhou}{2023}]{Cui_Zhou_2023}
\begin{barticle}[author]
\bauthor{\bsnm{Cui},~\bfnm{Yan}\binits{Y.}} \AND \bauthor{\bsnm{Zhou},~\bfnm{Zhou}\binits{Z.}}
(\byear{2023}).
\btitle{Simultaneous Inference for Time Series Functional Linear Regression}.
\bvolume{arXiv:2207.11392}.
\end{barticle}
\endbibitem

\bibitem[\protect\citeauthoryear{Devroye, Mehrabian and Reddad}{2018}]{devroye2018total}
\begin{barticle}[author]
\bauthor{\bsnm{Devroye},~\bfnm{Luc}\binits{L.}}, \bauthor{\bsnm{Mehrabian},~\bfnm{Abbas}\binits{A.}} \AND \bauthor{\bsnm{Reddad},~\bfnm{Tommy}\binits{T.}}
(\byear{2018}).
\btitle{The total variation distance between high-dimensional Gaussians with the same mean}.
\bjournal{arXiv preprint arXiv:1810.08693}.
\end{barticle}
\endbibitem

\bibitem[\protect\citeauthoryear{Donoho and Johnstone}{1995}]{Donoho_Johnstone_1995}
\begin{barticle}[author]
\bauthor{\bsnm{Donoho},~\bfnm{David~L.}\binits{D.~L.}} \AND \bauthor{\bsnm{Johnstone},~\bfnm{Iain~M.}\binits{I.~M.}}
(\byear{1995}).
\btitle{Adapting to Unknown Smoothness via Wavelet Shrinkage}.
\bjournal{Journal of the American Statistical Association}
\bvolume{90}
\bpages{1200–1224}.
\end{barticle}
\endbibitem

\bibitem[\protect\citeauthoryear{Einmahl}{1989}]{einmahl1989extensions}
\begin{barticle}[author]
\bauthor{\bsnm{Einmahl},~\bfnm{Uwe}\binits{U.}}
(\byear{1989}).
\btitle{Extensions of results of Koml{\'o}s, Major, and Tusn{\'a}dy to the multivariate case}.
\bjournal{Journal of multivariate analysis}
\bvolume{28}
\bpages{20--68}.
\end{barticle}
\endbibitem

\bibitem[\protect\citeauthoryear{Eldan, Mikulincer and Zhai}{2020}]{eldan2020clt}
\begin{barticle}[author]
\bauthor{\bsnm{Eldan},~\bfnm{Ronen}\binits{R.}}, \bauthor{\bsnm{Mikulincer},~\bfnm{Dan}\binits{D.}} \AND \bauthor{\bsnm{Zhai},~\bfnm{Alex}\binits{A.}}
(\byear{2020}).
\btitle{{The CLT in high dimensions: Quantitative bounds via martingale embedding}}.
\bjournal{The Annals of Probability}
\bvolume{48}
\bpages{2494 -- 2524}.
\end{barticle}
\endbibitem

\bibitem[\protect\citeauthoryear{Fan, Liao and Yao}{2015}]{fan2015power}
\begin{barticle}[author]
\bauthor{\bsnm{Fan},~\bfnm{Jianqing}\binits{J.}}, \bauthor{\bsnm{Liao},~\bfnm{Yuan}\binits{Y.}} \AND \bauthor{\bsnm{Yao},~\bfnm{Jiawei}\binits{J.}}
(\byear{2015}).
\btitle{Power enhancement in high-dimensional cross-sectional tests}.
\bjournal{Econometrica}
\bvolume{83}
\bpages{1497--1541}.
\end{barticle}
\endbibitem

\bibitem[\protect\citeauthoryear{Fang}{2016}]{Fang2016}
\begin{barticle}[author]
\bauthor{\bsnm{Fang},~\bfnm{Xiao}\binits{X.}}
(\byear{2016}).
\btitle{A Multivariate {CLT} for Bounded Decomposable Random Vectors with the Best Known Rate}.
\bjournal{Journal of Theoretical Probability}
\bvolume{29}
\bpages{1510--1523}.
\end{barticle}
\endbibitem

\bibitem[\protect\citeauthoryear{Fang and Koike}{2024}]{fang2024large}
\begin{barticle}[author]
\bauthor{\bsnm{Fang},~\bfnm{Xiao}\binits{X.}} \AND \bauthor{\bsnm{Koike},~\bfnm{Yuta}\binits{Y.}}
(\byear{2024}).
\btitle{Large-Dimensional Central Limit Theorem with Fourth-Moment Error Bounds on Convex Sets and Balls}.
\bjournal{The Annals of Applied Probability}
\bvolume{34}
\bpages{2065--2106}.
\end{barticle}
\endbibitem

\bibitem[\protect\citeauthoryear{Fang and R\"{o}llin}{2015}]{Fang2015}
\begin{barticle}[author]
\bauthor{\bsnm{Fang},~\bfnm{Xiao}\binits{X.}} \AND \bauthor{\bsnm{R\"{o}llin},~\bfnm{Adrian}\binits{A.}}
(\byear{2015}).
\btitle{Rates of convergence for multivariate normal approximation with applications to dense graphs and doubly indexed permutation statistics}.
\bjournal{Bernoulli}
\bvolume{21}
\bpages{2157--2189}.
\end{barticle}
\endbibitem

\bibitem[\protect\citeauthoryear{Feng et~al.}{2024}]{feng2022asymptotic}
\begin{barticle}[author]
\bauthor{\bsnm{Feng},~\bfnm{Long}\binits{L.}}, \bauthor{\bsnm{Jiang},~\bfnm{Tiefeng}\binits{T.}}, \bauthor{\bsnm{Li},~\bfnm{Xiaoyun}\binits{X.}} \AND \bauthor{\bsnm{Liu},~\bfnm{Binghui}\binits{B.}}
(\byear{2024}).
\btitle{Asymptotic independence of the sum and maximum of dependent random variables with applications to high-dimensional tests}.
\bjournal{Statistica Sinica}
\bpages{1745--1763}.
\end{barticle}
\endbibitem

\bibitem[\protect\citeauthoryear{G\"{o}tze}{1991}]{Gotze_1991}
\begin{barticle}[author]
\bauthor{\bsnm{G\"{o}tze},~\bfnm{F.}\binits{F.}}
(\byear{1991}).
\btitle{On the Rate of Convergence in the Multivariate {CLT}}.
\bjournal{The Annals of Probability}
\bvolume{19}
\bpages{724–739}.
\end{barticle}
\endbibitem

\bibitem[\protect\citeauthoryear{Karmakar and Wu}{2020}]{Karmakar2020}
\begin{barticle}[author]
\bauthor{\bsnm{Karmakar},~\bfnm{Sayar}\binits{S.}} \AND \bauthor{\bsnm{Wu},~\bfnm{Wei~Biao}\binits{W.~B.}}
(\byear{2020}).
\btitle{Optimal Gaussian approximation for multiple time series}.
\bjournal{Statistica Sinica}
\bvolume{30}
\bpages{1399--1417}.
\end{barticle}
\endbibitem

\bibitem[\protect\citeauthoryear{Koml{\'o}s, Major and Tusn{\'a}dy}{1975}]{komlos1975approximation}
\begin{barticle}[author]
\bauthor{\bsnm{Koml{\'o}s},~\bfnm{J{\'a}nos}\binits{J.}}, \bauthor{\bsnm{Major},~\bfnm{P{\'e}ter}\binits{P.}} \AND \bauthor{\bsnm{Tusn{\'a}dy},~\bfnm{G{\'a}bor}\binits{G.}}
(\byear{1975}).
\btitle{An approximation of partial sums of independent RV'-s, and the sample DF. I}.
\bjournal{Zeitschrift f{\"u}r Wahrscheinlichkeitstheorie und verwandte Gebiete}
\bvolume{32}
\bpages{111--131}.
\end{barticle}
\endbibitem

\bibitem[\protect\citeauthoryear{Koml{\'o}s, Major and Tusn{\'a}dy}{1976}]{Komlos_1976}
\begin{barticle}[author]
\bauthor{\bsnm{Koml{\'o}s},~\bfnm{J.}\binits{J.}}, \bauthor{\bsnm{Major},~\bfnm{P.}\binits{P.}} \AND \bauthor{\bsnm{Tusn{\'a}dy},~\bfnm{G.}\binits{G.}}
(\byear{1976}).
\btitle{An approximation of partial sums of independent {RV’s}, and the sample DF. {II}}.
\bjournal{Zeitschrift für Wahrscheinlichkeitstheorie und Verwandte Gebiete}
\bvolume{34}
\bpages{33–58}.
\end{barticle}
\endbibitem

\bibitem[\protect\citeauthoryear{Liu and Lin}{2009}]{liu2009strong}
\begin{barticle}[author]
\bauthor{\bsnm{Liu},~\bfnm{Weidong}\binits{W.}} \AND \bauthor{\bsnm{Lin},~\bfnm{Zhengyan}\binits{Z.}}
(\byear{2009}).
\btitle{Strong approximation for a class of stationary processes}.
\bjournal{Stochastic Processes and their Applications}
\bvolume{119}
\bpages{249--280}.
\end{barticle}
\endbibitem

\bibitem[\protect\citeauthoryear{Mies and Steland}{2023}]{mies2023sequential}
\begin{barticle}[author]
\bauthor{\bsnm{Mies},~\bfnm{Fabian}\binits{F.}} \AND \bauthor{\bsnm{Steland},~\bfnm{Ansgar}\binits{A.}}
(\byear{2023}).
\btitle{Sequential Gaussian approximation for nonstationary time series in high dimensions}.
\bjournal{Bernoulli}
\bvolume{29}
\bpages{3114--3140}.
\end{barticle}
\endbibitem

\bibitem[\protect\citeauthoryear{Nagaev}{1976}]{Nagaev_1976}
\begin{binbook}[author]
\bauthor{\bsnm{Nagaev},~\bfnm{S.~V.}\binits{S.~V.}}
(\byear{1976}).
\btitle{An estimate of the remainder term in the multidimensional central limit theorem}.
In \bbooktitle{Proceedings of the Third Japan — USSR Symposium on Probability Theory}.
\bseries{Lecture Notes in Mathematics}
\bvolume{550}
\bpages{419–438}.
\bpublisher{Springer Berlin Heidelberg}, \baddress{Berlin, Heidelberg}.
\end{binbook}
\endbibitem

\bibitem[\protect\citeauthoryear{Pengel, Yang and Zhou}{2024}]{pengel2024gaussian}
\begin{barticle}[author]
\bauthor{\bsnm{Pengel},~\bfnm{Ardjen}\binits{A.}}, \bauthor{\bsnm{Yang},~\bfnm{Jun}\binits{J.}} \AND \bauthor{\bsnm{Zhou},~\bfnm{Zhou}\binits{Z.}}
(\byear{2024}).
\btitle{Gaussian Approximation and Output Analysis for High-Dimensional {MCMC}}.
\bjournal{arXiv preprint arXiv:2407.05492}.
\end{barticle}
\endbibitem

\bibitem[\protect\citeauthoryear{Pollard}{2001}]{Pollard2001}
\begin{bbook}[author]
\bauthor{\bsnm{Pollard},~\bfnm{David}\binits{D.}}
(\byear{2001}).
\btitle{A User's Guide to Measure Theoretic Probability}.
\bseries{Cambridge Series in Statistical and Probabilistic Mathematics}.
\bpublisher{Cambridge University Press}.
\end{bbook}
\endbibitem

\bibitem[\protect\citeauthoryear{Rosenblatt}{1952}]{rosenblatt1952remarks}
\begin{barticle}[author]
\bauthor{\bsnm{Rosenblatt},~\bfnm{Murray}\binits{M.}}
(\byear{1952}).
\btitle{Remarks on a multivariate transformation}.
\bjournal{The Annals of Mathematical Statistics}
\bvolume{23}
\bpages{470--472}.
\end{barticle}
\endbibitem

\bibitem[\protect\citeauthoryear{St{\'e}phane}{2009}]{Stephane_2009}
\begin{binbook}[author]
\bauthor{\bsnm{St{\'e}phane},~\bfnm{Mallat}\binits{M.}}
(\byear{2009}).
\btitle{CHAPTER 11 - Denoising}
In \bbooktitle{A Wavelet Tour of Signal Processing (Third Edition)}
\bpages{535–610}.
\bpublisher{Academic Press}, \baddress{Boston}.
\end{binbook}
\endbibitem

\bibitem[\protect\citeauthoryear{van Hemmen and Ando}{1980}]{van1980inequality}
\begin{barticle}[author]
\bauthor{\bparticle{van} \bsnm{Hemmen},~\bfnm{J~Leo}\binits{J.~L.}} \AND \bauthor{\bsnm{Ando},~\bfnm{Tsuneya}\binits{T.}}
(\byear{1980}).
\btitle{An inequality for trace ideals}.
\bjournal{Communications in Mathematical Physics}
\bvolume{76}
\bpages{143--148}.
\end{barticle}
\endbibitem

\bibitem[\protect\citeauthoryear{Wang, Wu and Castleman}{2023}]{Wang_Wu_Castleman_2023}
\begin{binbook}[author]
\bauthor{\bsnm{Wang},~\bfnm{Yu-Ping}\binits{Y.-P.}}, \bauthor{\bsnm{Wu},~\bfnm{Qiang}\binits{Q.}} \AND \bauthor{\bsnm{Castleman},~\bfnm{Kenneth~R.}\binits{K.~R.}}
(\byear{2023}).
\btitle{Image Enhancement}
In \bbooktitle{Microscope Image Processing}
\bpages{55–74}.
\bpublisher{Elsevier}.
\end{binbook}
\endbibitem

\bibitem[\protect\citeauthoryear{Wu}{2005}]{wu2005nonlinear}
\begin{barticle}[author]
\bauthor{\bsnm{Wu},~\bfnm{W.~B.}\binits{W.~B.}}
(\byear{2005}).
\btitle{Nonlinear system theory: Another look at dependence}.
\bjournal{Proceedings of the National Academy of Sciences}
\bvolume{102}
\bpages{14150--14154}.
\end{barticle}
\endbibitem

\bibitem[\protect\citeauthoryear{Wu and Zhou}{2011}]{wu2011gaussian}
\begin{barticle}[author]
\bauthor{\bsnm{Wu},~\bfnm{Wei~Biao}\binits{W.~B.}} \AND \bauthor{\bsnm{Zhou},~\bfnm{Zhou}\binits{Z.}}
(\byear{2011}).
\btitle{Gaussian approximations for non-stationary multiple time series}.
\bjournal{Statistica Sinica}
\bpages{1397--1413}.
\end{barticle}
\endbibitem

\bibitem[\protect\citeauthoryear{Wu and Zhou}{2024}]{zhou2020frequency}
\begin{barticle}[author]
\bauthor{\bsnm{Wu},~\bfnm{Hau-Tieng}\binits{H.-T.}} \AND \bauthor{\bsnm{Zhou},~\bfnm{Zhou}\binits{Z.}}
(\byear{2024}).
\btitle{Frequency Detection and Change Point Estimation for Time Series of Complex Oscillation}.
\bjournal{Journal of the American Statistical Association}
\bvolume{119}
\bpages{1945--1956}.
\end{barticle}
\endbibitem

\bibitem[\protect\citeauthoryear{Zaitsev}{2007}]{zaitsev2007estimates}
\begin{barticle}[author]
\bauthor{\bsnm{Zaitsev},~\bfnm{A~Yu}\binits{A.~Y.}}
(\byear{2007}).
\btitle{Estimates for the rate of strong approximation in the multidimensional invariance principle}.
\bjournal{Journal of Mathematical Sciences}
\bvolume{145}
\bpages{4856--4865}.
\end{barticle}
\endbibitem

\bibitem[\protect\citeauthoryear{Zhai}{2018}]{zhai2018high}
\begin{barticle}[author]
\bauthor{\bsnm{Zhai},~\bfnm{Alex}\binits{A.}}
(\byear{2018}).
\btitle{A high-dimensional CLT in $W_2$ distance with near optimal convergence rate}.
\bjournal{Probability Theory and Related Fields}
\bvolume{170}
\bpages{821--845}.
\end{barticle}
\endbibitem

\bibitem[\protect\citeauthoryear{Zhang and Cheng}{2018}]{zhang2018}
\begin{barticle}[author]
\bauthor{\bsnm{Zhang},~\bfnm{Xianyang}\binits{X.}} \AND \bauthor{\bsnm{Cheng},~\bfnm{Guang}\binits{G.}}
(\byear{2018}).
\btitle{Gaussian approximation for high dimensional vector under physical dependence}.
\bjournal{Bernoulli}
\bvolume{24}
\bpages{2640--2675}.
\end{barticle}
\endbibitem

\bibitem[\protect\citeauthoryear{Zhang and Wu}{2017}]{zhang2017gaussian}
\begin{barticle}[author]
\bauthor{\bsnm{Zhang},~\bfnm{Danna}\binits{D.}} \AND \bauthor{\bsnm{Wu},~\bfnm{Wei~Biao}\binits{W.~B.}}
(\byear{2017}).
\btitle{Gaussian approximation for high dimensional time series}.
\bjournal{The Annals of Statistics}
\bvolume{45}
\bpages{1895--1919}.
\end{barticle}
\endbibitem

\bibitem[\protect\citeauthoryear{Zhou}{2013}]{zhou2013heteroscedasticity}
\begin{barticle}[author]
\bauthor{\bsnm{Zhou},~\bfnm{Zhou}\binits{Z.}}
(\byear{2013}).
\btitle{Heteroscedasticity and autocorrelation robust structural change detection}.
\bjournal{Journal of the American Statistical Association}
\bvolume{108}
\bpages{726--740}.
\end{barticle}
\endbibitem

\bibitem[\protect\citeauthoryear{Zhou}{2014}]{zhou2014inference}
\begin{barticle}[author]
\bauthor{\bsnm{Zhou},~\bfnm{Z.}\binits{Z.}}
(\byear{2014}).
\btitle{Inference of weighted {$V$-statistics} for nonstationary time series and its applications}.
\bjournal{The Annals of Statistics}
\bvolume{42}
\bpages{87--114}.
\end{barticle}
\endbibitem

\end{thebibliography}

%% or include bibliography directly:
% \begin{thebibliography}{}
% \bibitem{b1}
% \end{thebibliography}

\clearpage
\appendix

\section{Proof of \cref{main_borel}\label{proof_main_borel}}
%\subsection{Proof of \ref{main_borel}}\label{proof_main_borel}
	
%\RED{(JY: some arguments might be missing: e.g., need to get back to the original the covariance matrix. Right now the covariance matrix is the covariance of the truncated version)}

%\textcolor{OliveGreen}{Maybe need to add more explanation on how $PB$ and $QB$ is defined?} \RED{I think it's just $P(B)$ and $Q(B)$...I was using the notation from the book...}
Throughout the proof, we define $\|\cdot\|_{\gamma}:=(\E[|\cdot|^{\gamma}])^{1/{\gamma}}, \forall \gamma>2$.
		Recall the definition of L\'evy--Prohorov distance. Let $(\mathcal{X},|\cdot|)$ be a separable metric space equipped with a Borel sigma field $\mathcal{B}(\mathcal{X})$. The L\'evy--Prohorov distance between two probability measures $P$ and $Q$ is defined as
	\[
	\rho(P,Q):=\inf_{B\in \mathcal{B}(\mathcal{X})} \{\epsilon>0: P(B)\leq   Q(B^{\epsilon})+\epsilon \}.
	\]
	Recall the characteristic of proximity of probability distributions, which is closely connected with the L\'evy--Prohorov distance:
	\[
	\rho(P,Q; \epsilon):=\sup_{B\in\mathcal{B}(\mathcal{X})} \max\{P(B)-Q(B^{\epsilon}), Q(B)-P(B^{\epsilon}) \}
	\]
	We have the connection
	\[
	\rho(P,Q)=\inf\{\epsilon>0: \rho(P,Q; \epsilon)\leq   \epsilon \}.
	\]
		Note that $\rho(P,Q; \epsilon)\leq   \epsilon'$ implies $P(B)\leq   Q(B^{\epsilon})+\epsilon', \forall B\in \mathcal{B}(\mathcal{X})$,
	which is exactly what we need according to Strassen's theorem in \cref{thm_Strassen}.
	
		\begin{lemma}{\textit{Strassen's theorem} \citep[Ch.10, Theorem 8]{Pollard2001}}\label{thm_Strassen}
		Let $P$ and $Q$ be tight probability measures on the Borel sigma field $\mathcal{B}(\mathcal{X})$. Let $\epsilon$ and $\epsilon'$ be positive constants. There exists random elements $X$ and $Y$ of $\mathcal{X}$ with distributions $P$ and $Q$ such that $\pr( |X-Y|>\epsilon)\leq   \epsilon'$ if and only if $P(B)\leq   Q(B^{\epsilon})+\epsilon'$ for all Borel sets $B$.
	\end{lemma}

    \begin{comment}
	\PURPLE{Next lemma from \citep[Fact 2.2]{Einmahl1997} (see also \citep[Thm 1.1]{zaitsev1987estimates} and \citep[Thm 1.1]{Zaitsev1987}) provides an upper bound of L\'evy--Prohorov distance for sums of independent random vectors.   
	\begin{lemma}{\citep[Fact 2.2]{Einmahl1997}}\label{lem_zaitsev}
		Let $\tau>0$ and $x_1,\dots,x_n\in \R^d$ be independent random vectors such that $\E[x_i]=0$ and $|x_i|\le \tau$ for $i=1,\dots,n$. Let $X_n=x_1+\dots+x_n$ and $F=\mathcal{L}(X_n)$. Denote by $\Phi$ the Gaussian with zero mean and the same covariance matrix as $F$. Then
		$\rho(F,\Phi)\leq   C_1{d^{2}}\tau(|\log \tau|+1)$
		and for all $\epsilon\ge 0$
		\[
		\rho(F,\Phi;\epsilon)\leq   C_2{d^{2}}\exp\left(-\frac{\epsilon}{C_3{d^{2}}\tau}\right).
		\]
%    \RED{TODO: CHANGE IT TO $d^2$ AND UPDATE ALL ORDERS...}\RED{TODO: add Eldan (and Mies) result}
	\end{lemma}	
}
\end{comment}

By extending \citet[Theorem 1]{eldan2020clt}, we can prove the following key lemma for sums of independent random vectors, which will be later applied to prove our main theorem on sums of dependent random vectors.
\begin{lemma}\label{key_lemma1} 
Suppose $x_1,\dots,x_n\in {\R}^d$ are independent random vectors that $\E[x_i]=0$ and $x_i=(x_{i,1},\dots,x_{i,d})^T$, in which $|x_i|$ are bounded such that ${|x_{i}|\le d^{1/2}\tau}, i=1,\dots,n$. Let $X_n=\sum_{i=1}^n x_i$,  then let $Y_n$ be a Gaussian vector with the same mean and covariance matrix as $X_n$, we have
\[
    \mathcal{W}_2\left(\frac{1}{\sqrt{n}}X_n,\frac{1}{\sqrt{n}}Y_n\right)\le Cd\tau\frac{\sqrt{\log(n)}}{\sqrt{n}},
\]
for some constant $C$.
\end{lemma}
   % 
    %    \PURPLE{then one can construct $X_n^c$ and $Y_n^c$ in a richer probability space such that $X_n^c\stackrel{D}{=} X_n$ and for all $\epsilon\ge 0$,
	%	\[
	%	\Pr\left(|X_n^c-Y_n^c|>\epsilon\right)\le \min\left\{{C_1\frac{d^2\tau^2\log(n)}{\epsilon^2}}, C_2 {d^{2}}\exp\left(-\frac{\epsilon}{C_3{d^{5/2}}\tau}\right)\right\},
	%	\]
	%	 where $C_1$,$C_2$ and $C_3$ are some universal constants, $Y_n^c$ is a Gaussian vector with the same mean and covariance matrix as $X_n^c$.
     %    }
      %
\begin{proof}
This is an extension of \citet[Theorem 1]{eldan2020clt}. It follows mostly from the arguments in \citet[Proof of Theorem 2.1]{mies2023sequential} except a few steps, see \cref{remark_mies} for the differences. First of all, according to the construction by \cite{eldan2020clt}, there exists independent Brownian motions $(B_s^i)_{s\ge 0}$, stopping times $\tau_i$, and adapted processes $\Gamma_s^i\in\mathbb{R}^{d\times d}$, for $i=1,\dots,n$, such that $\Gamma_s^i=0$ for $s\ge \tau_i$, and $x_i^c:=\int_0^{\infty}\Gamma_s^i\dee B_s^i\stackrel{D}{=} x_i$. Moreover, each $\Gamma_s^i$ is a symmetric positive semi-definite projection matrix. Then, by \citet[Proposition 5.5]{mies2023sequential}, there is another Brownian motion $B_s$ such that $X_n^c:=\int_0^{\infty} \tilde{\Gamma}_s \dee B_i\stackrel{D}{=} X_n$, where $\tilde{\Gamma}_s:=\sqrt{\frac{1}{n}\sum_{i=1}^n(\Gamma_s^i)^2}$.

Denoting $Y_n^c:=\int_0^{\infty}\sqrt{\E[(\Gamma_s^i)^2]}\dee B_s\sim \mathcal{N}(0,\Cov(X_n))$, analogously to the proof of \citet[Theorem 1]{eldan2020clt}, we have
\[
\E\left[\frac{1}{n}\left|X_n^c-Y_n^c\right|^2\right]\le \int_0^{\infty}\min\left\{\frac{1}{n}\tr\left(\E[(\tilde{\Gamma}_s)^4](\E[(\tilde{\Gamma}_s)^2])^{\dagger}\right),4 \tr(\E[\tilde{\Gamma}_s^2])\right\} \dee s,
\]
where $(\cdot)^{\dagger}$ denotes the pseudo-inverse. Then using the fact that $\Gamma_s^i$s are projection matrices so that $\Gamma_s^i\Gamma_s^i=\Gamma_s^i$ and all traces of $\Gamma_s^i$ are bounded by $d$, substituting the definition of $\tilde{\Gamma}_s$, we have
 \[
 \frac{1}{n}\tr\left(\E[(\tilde{\Gamma}_s)^4](\E[(\tilde{\Gamma}_s)^2])^{\dagger}\right)\le \frac{d}{n}.
 \]
 Furthermore, using independence and the property of projection matrices, following the same arguments as in \citet[pp.2510]{eldan2020clt}, we have
 \[
 4 \tr(\E[\tilde{\Gamma}_s^2])=\frac{4}{n}\sum_{i=1}^n \E[\tr(\Gamma_s^i)]\le d \Pr(\tau_i>s),
 \]
 where the last inequality comes from the fact $\tr(\E[(\Gamma_s^i)^2])\le d \Pr(\tau_i>s)$.

Therefore, for any $\kappa>0$, we have
\[
\E\left[\frac{1}{n}\left|X_n^c-Y_n^c\right|^2\right]\le \int_0^{\infty} \left(\frac{d}{n}\wedge \frac{4d}{n}\sum_{i=1}^n \Pr(\tau_i>s)\right) \dee s \le \int_0^{\kappa}\frac{4d}{n} \dee s + \int_{\kappa}^{\infty} \frac{4d}{n}\sum_{i=1}^n \Pr(\tau_i>s) \dee s.
\]
Choosing $\kappa=4(\max_i |x_i|)\log(n)$ yields 
\[
\left\|\frac{1}{\sqrt{n}}\left(X_n^c-Y_n^c\right)\right\|_2^2\le C \frac{d}{n}\log(n) (\max_i |x_i|)^2=C \frac{d}{n}\log(n) (d^{1/2}\tau)^2,
\]
for some constant $C$. 
%\PURPLE{Finally, by Chebyshev's inequality
%\[
%\Pr\left(|X_n^c-Y_n^c|>\epsilon\right)\le \frac{\E\left|\frac{1}{\sqrt{n}}(X_n^c-Y_n^c)\right|^2}{(\epsilon/\sqrt{n})^2} \le C_1\frac{\left[d^{1/2}\tau \sqrt{\frac{d}{n}\log(n)}\right]^2}{(\epsilon/\sqrt{n})^2}=C_1\frac{d^2\tau^2\log(n)}{\epsilon^2}.
%\]
%}
%\PURPLE{ The second part of the result follows almost directly from \cref{thm_Strassen} and \cref{lem_zaitsev}: the Strasssen's Theorem (\cref{thm_Strassen}) suggests we only need to bound $\rho(P,Q,\epsilon)$, where $P$ and $Q$ are distributions of $X_n$ and $Y_n$, respectively. Therefore, we can apply directly \cref{lem_zaitsev} using $|x_i|\le d^{1/2}\tau$ to complete the proof. }
\end{proof}
	
\begin{remark}\label{remark_mies}
    In the proof of \cref{key_lemma1}, we noticed a minor technical point in the proof of \citet[Theorem 2.1]{mies2023sequential}, which we address here. Specifically, we established that 
$\E\left[\frac{1}{n}\left|X_n^c-Y_n^c\right|^2\right] \leq \sqrt{C_1} \frac{d}{n}\log(n) (\max_i |x_i|)^2$,
which is slightly different from the bound in \citet[Theorem 2.1]{mies2023sequential}, 
$\E\left[\frac{1}{n}\left|X_n^c-Y_n^c\right|^2\right] \leq \sqrt{C_1} \frac{d}{n}\log(n) \left(\frac{1}{n}\sum_i |x_i|^2\right)$.
To illustrate this distinction, consider the example where $|x_i|=0$ for all $i<n$ and $|x_n|=\sqrt{n}$. Despite this adjustment, the core results of \cite{mies2023sequential} remain unaffected, and we emphasize that this point does not detract from the overall contributions of their work.
\end{remark}
	
\emph{Proof of \cref{main_borel}:} We mainly focus on proving the most general case (which is also the most technical case): when $\{x_i\}$ satisfy \cref{asm:moment}(a) and \cref{asm:dependence}(a). At the end of the proof, we will discuss the differences of the proofs for the other cases, such as $\{x_i\}$ satisfying \cref{asm:dependence}(b) and/or having finite exponential moments. We first present a new truncation result in \cref{lem:l2truncation} with ${\cal{L}}^2$-norm of truncation errors. Then we follow similar technical steps as in \citet{Karmakar2020,bonnerjee2024gaussian}, which generalizes the proof techniques for univariate processes in \citet{berkes2014komlos} to vector-valued processes, and either \cref{ineq-vanHemmen-Ando} or \citet[Lemma 1]{eldan2020clt} to ``correct'' the covariance matrix mismatch due to approximation. In the ``preparation step'' in \cref{subsec_prepare}, we first truncate $x_i$ in \cref{subsubsec_truncate}. Then, we apply $M$-depedence approximation in \cref{subsubsec_m_dep}. Finally, we use block approximation in \cref{subsubsec_block_approx} to approximate $X_n$ by the sum of a sequence of block sums. The obtained blocks are weakly dependent because of the $e_i$'s in the shared border. Next, in \cref{subsec_cond_GA}, conditional on these border $e_i$'s, we apply the GA result in \cref{key_lemma1} to the conditional independent blocks. This step results in an unconditional approximation using a mixture of Gaussian vectors. Then, we consider unconditional approximation by one Gaussian in \cref{subsec_uncond_GA} to apply \cref{key_lemma1} again. Finally, we put together results from previous steps and summarize the final results in \cref{subsec_summary}.
	
\subsection{Preparation}\label{subsec_prepare}
In this subsection, we first present a new truncation result in \cref{lem:l2truncation}, which is the key step for getting ${\cal{L}}^2$-norm of truncation errors. Then, we follow similar techniques in \citet{Karmakar2020,bonnerjee2024gaussian}, which include $M$-dependence approximation and block approximation. The difference in our result is that we keep the dimension dependence explicit. This step aims to approximate $X_n=\sum_{i=1}^n x_i$ by a sum of $\bigO(n/m)$ conditional independent variables where $m$ is chosen carefully based on the dependence measure.
	
	\subsubsection{Truncation}\label{subsubsec_truncate}
	
	For $b>0$ and $v=(v_1,\dots,v_d)^T$, we define the truncation operator as follows
	\[
	T_b(v):=(T_b(v_1),\dots,T_b(v_d))^T,
	\]
	where $T_b(v_i)=\min(\max(v_i,-b),b)$. 

Before we start controlling the truncation effect for the purpose of Gaussian approximations, we shall first prove the following useful lemma which investigates ${\cal{L}}^2$ norm of truncation errors for non-stationary time series. The result may be of general interest. 

%\begin{lemma}\label{lem:l2truncation}
%Let $h_i={\cal H}_i({\cal F}_i)$, $i=1,2,\cdots,n$ be a univariate non-stationary time series. For a truncation level $M>0$, let $\tilde{h}_i=h_i-T_M(h_i)$ and $\tilde{H}_n=\sum_{i=1}^n\tilde{h}_i$. Assume that (i) $\sup_i\|h_i\|_{p}\le C_p<\infty$ for some $p> 2$;  (ii) $\inf_i\|h_i\|_{p/2-1}\ge c_p>0$;  (iii) The truncation level $M$ goes to infinity as $n\rightarrow\infty$. Denote by $\theta^*_{h,k,q}$ the dependence measure of the sequence $\{h_i\}$ at order $k$ with respect to the $L^q$ norm. Further assume that (iv) $\theta^*_{h,k,p}=O(k^{-\alpha})$ for some $\alpha>1$. Then we have
%\begin{align*}
%\|\tilde{H}_n\|^2_2\le C_{1p}n/M^{p-2} \mbox{ for sufficiently large } n,
%\end{align*}
%where $C_{1p}$ is a finite constant which does not depend on $M$ or $n$.
%\end{lemma}
\begin{lemma}\label{lem:l2truncation}
Let $h_i={\cal H}_i({\cal F}_i)$, $i=1,2,\cdots,n$ be a univariate non-stationary time series. For a truncation level $M>0$, let $\tilde{h}_i=h_i-T_M(h_i)$ and $\tilde{H}_n=\sum_{i=1}^n\tilde{h}_i$. Assume that (i) $\sup_i\|h_i\|_{p}\le C_p<\infty$ for some $p> 2$.  Denote by $\theta^*_{h,k,q}$ the dependence measure of the sequence $\{h_i\}$ at order $k$ with respect to the ${\cal{L}}^q$ norm. Further assume that (ii) $\theta^*_{h,k,p}=\bigO(k^{-\alpha})$ for some $\alpha>1$. Then we have
\begin{align*}
\|\tilde{H}_n\|_2\le C_{1p}(\sqrt{n/M^{p-2}}+n/M^{p-1}),
\end{align*}
where $C_{1p}$ is a finite constant which does not depend on $M$ or $n$.

Furthermore, if (i) is replaced by (i') $\sup_i \exp(|h_i|)\le C_{\infty}<\infty$ we have
\[
\|\tilde{H}_n\|_2\le C_{1\infty}n/(\exp(M)^{(p-2)/p}),
\]
where $p$ comes from (ii) and $C_{1\infty}$ is a finite constant which does not depend on $M$ or $n$.
\end{lemma}
\begin{proof}
Let $\bar{H}_n=\tilde{H}_n-\E[\tilde{H}_n]$.  Note that $\|\bar{H}_n\|^2_2=\sum_{i=1}^n\sum_{j=1}^n\mbox{Cov}(\tilde{h}_i,\tilde{h}_j)$. We will control $\mbox{Cov}(\tilde{h}_i,\tilde{h}_j)$ for all pairs $(i,j)$. We separate two cases: $i=j$ and $i\neq j$ as follows:
\begin{itemize}
    \item 
Case 1: $i=j$. Note that 
\begin{align}\label{eq:s241}
|\tilde{h}_i|= (|h_i|-M)\Ind_{|h_i|>M}.
\end{align}
Therefore, we have by Assumption (i) of this lemma that
\begin{align*}
\mbox{Cov}(\tilde{h}_i,\tilde{h}_i)&\le \E(\tilde{h}_i^2)=\E[(|h_i|-M)^2\Ind_{|h_i|>M}]\\
&\le\E[(|h_i|-M)^2|h_i|^{p-2}\Ind_{|h_i|>M}]/M^{p-2}\le C^p_p/M^{p-2}.    
\end{align*}
\item Case 2: $i\neq j$. In this case we shall first control the physical dependence measure of $\{\tilde{h}_i\}$. Let $h_{i,k}={\cal G}_i({\cal F}_{i,i-k})$ and $\tilde{h}_{i,k}=h_{i,k}-T_M(h_{i,k})$. Recall the definition of ${\cal F}_{i,i-k}$ at the beginning of Section \ref{section:dependence}. Observe that the truncation function is Lipchitz continuous with Lipchitz constant 1. Therefore it is easy to see that, for all $i$ and $k$,
\begin{align}\label{eq:242}
|\tilde{h}_i-\tilde{h}_{i,k}|\le |h_i-h_{i,k}|(\Ind_{|h_i|>M}+\Ind_{|h_{i,k}|>M})    
\end{align}
Let $\theta^*_{\tilde{h},k}$ be the $k$-th order physical dependence measure of the sequence $\{\tilde{h}_i\}$ with respect to the ${\cal{L}}^2$ norm. Then \cref{eq:242} implies that
\begin{align}\label{eq:s243}
(\theta^*_{\tilde{h},k})^2&=\E(|\tilde{h}_i-\tilde{h}_{i,k}|^2)\le 2\E[|h_i-h_{i,k}|^2(\Ind_{|h_i|>M}+\Ind_{|h_{i,k}|>M})] \\\nonumber  
&\le 2\E[|h_i-h_{i,k}|^2(|h_i|^{p-2}+|h_{i,k}|^{p-2})(\Ind_{|h_i|>M}+\Ind_{|h_{i,k}|>M})]/M^{p-2}\\\nonumber
&\le 2[\E(|h_i-h_{i,k}|^2|h_i|^{p-2})+\E(|h_i-h_{i,k}|^2|h_{i,k}|^{p-2})]/M^{p-2}\\\nonumber
&:=2(\textrm{Term I}+\textrm{ Term II})/M^{p-2}.
\end{align}
By H\"older's inequality $\E[|XY|]\le \left(\E[|X|^{p/2}]\right)^{\frac{2}{p}}\left(\E[|Y|^{p/(p-2)}]\right)^{\frac{p-2}{p}}$, we have
\begin{align*}
\textrm{Term I}\le \|h_i-h_{i,k}\|^2_p\|h_i\|_p^{p-2}\le C^{p-2}_{p}(\theta^*_{h,k,p})^2. \end{align*}
The same inequality holds for $\textrm{Term II}$. Therefore, \cref{eq:s243} implies that
\begin{align*}
\theta^*_{\tilde{h},k}\le C_{2p}\theta^*_{h,k,p}/M^{(p-2)/2}.
\end{align*}
By Assumption (ii) of this lemma  and  \citet[Lemma 6]{zhou2014inference}, we have that
\begin{align}\label{eq:s244}
\mbox{Cov}(\tilde{h}_i,\tilde{h}_j)\le C_{3p}|i-j|^{-\alpha}/M^{p-2}.
\end{align}
where $C_{3p}$ is a finite constant. 
\end{itemize}
Combining the results of Case 1 and Case 2, simple calculations yield that
\begin{align}\label{eq:s245}
\|\tilde{H}_n-\E\tilde{H}_n\|_2^2=\bigO(n/M^{p-2}).    
\end{align}
Our final task is to control $\E[\tilde{H}_n]$. Note that
\begin{align*}
|\E\tilde{H}_n|\le n\max_i\E|\tilde{h}_i|\le n\max_i\E(|h_i|^p \Ind_{|h_i|>M})/M^{p-1}\le C_{4p}n/M^{p-1}.
\end{align*}
This result together with \cref{eq:s245} implies that the lemma holds.\\
%\textcolor{OliveGreen}{Should it be (i) or (ii)? I beleive (i') is a typo} \RED{it is (i'} defined on the top of pp.4}\\

Finally, if (i') $\sup_i \exp(|h_i|)\le C_{\infty}<\infty$ holds instead of (I), using $x^2< \exp(|x|)$, for Case 1, we have
\[
\mbox{Cov}(\tilde{h}_i,\tilde{h}_i)\le \E[(|h_i|-M)^2\Ind_{|h_i|>M}]\le \E[\exp(|h_i|-M)]=C_{\infty}/\exp(M).
\]
Similarly, for Case 2, we have
\begin{align*}
    (\theta^*_{\tilde{h},k})^2&=\E(|\tilde{h}_i-\tilde{h}_{i,k}|^2)\le 2\E[|h_i-h_{i,k}|^2(\Ind_{|h_i|>M}+\Ind_{|h_{i,k}|>M})] \\\nonumber  
&\le 2\E[|h_i-h_{i,k}|^2(\exp(|h_i|-M)^{(p-2)/p}+\exp(|h_{i,k}|-M)^{(p-2)/p})\\\nonumber
&\le 4\|h_i-h_{i,k}\|^2_p\left(C_{\infty}/\exp(M)\right)^{(p-2)/p}\le 4C_{\infty}^{(p-2)/p}(\theta^*_{h,k,p})^2/(\exp(M)^{(p-2)/p}).
\end{align*}
Following the same arguments and combining two cases yield
\begin{align*}
\|\tilde{H}_n-\E\tilde{H}_n\|_2^2=\bigO(n/\exp(M)^{(p-2)/p}).    
\end{align*}
Finally, combing with $|\E\tilde{H}_n|$ finishes the proof:
\[
|\E\tilde{H}_n|\le n\max_i\E|\tilde{h}_i|\le n\max_iE[\exp(|h_i|-M)]=C_{\infty}n/\exp(M).
\]
\end{proof}
\begin{remark}
Selecting $M=n^{1/p}$ in Lemma \ref{lem:l2truncation}, we obtain that $\|\tilde{H}_n\|_2\le \sqrt{C_{1p}}n^{1/p}$. It is well-known that the order $\bigO(n^{1/p})$ is optimal under the finite $\mathcal{L}^p$ norm assumption of the time series. To our knowledge, Lemma \ref{lem:l2truncation} is the first result that yields optimal $\mathcal{L}^2$ truncation error rate for time series. We notice that typically previous results on optimal truncation error control for time series were $\bigO_\p(\cdot)$, $o_\p(\cdot)$, or $\mathcal{L}^1$-type results which were not quite suitable for problems related to the 2-Wasserstein distance. As a result, Lemma \ref{lem:l2truncation} is important for the establishment of the nearly optimal 2-Wasserstein Gaussian approximation result of this paper and could be of broad interest. %We also remark that it is difficult to get an accurate control of the physical dependence measures of $\tilde{h}_i$ with respect to the truncation level $M$. Consequently it is difficult to control the $L^2$ norm of the truncation error in $\tilde{H}_n$ by directly using Bernstein-type or Rosenthal-type inequalities for time series. Instead, in our proofs, we resort to the sequence $\{f_i(h_i)\}$ to circumvent the difficulty. Finally, we remark that the assumption (ii) of this lemma is mild in view of $\|h_i\|_{q/2-1}=0$ implies $h_i=0$ almost surely. Assumption (ii) excludes near degenerate cases of the sequence.     

\end{remark}
    
%    We follow similar arguments (but with explicit dimension dependence) as \cite[Proof of Proposition 8.1]{bonnerjee2024gaussian}.
Now we come back to control the truncation effect for our high dimensional Gaussian approximation.
To keep the dimension dependence explicit, using the assumption $\sup_{j=1,\dots,d} \E[|x_{i,j}|^p]<C_p<\infty$, for any $p>2$, we have
	\[
	\sup_i \E[|x_i|^p]=d^{p/2}\sup_i \E\left[ \left(\frac{1}{d}\sum_j x_{i,j}^2\right)^{p/2}\right]\le d^{p/2}\sup_i \E\left[\frac{1}{d}\sum_j |x_{i,j}|^p\right]\le d^{p/2}C_p.
	\]
     We define $x_i^{\oplus}:=T_{n^{1/p}}(x_i)$ and $X_n^{\oplus}$ as follows to approximate $X_n$:
    \[
X_n^{\oplus}:=\sum_{i=1}^n \left[x_i^{\oplus}-\E[x_i^{\oplus}]\right]=\sum_{i=1}^n\left[T_{n^{1/p}}(x_i)-\E\left[T_{n^{1/p}}(x_i)\right]\right].
	\]
      Applying \cref{lem:l2truncation} coordinate-wisely yields
\begin{equation}\label{summary_step1}
	    \|X_n-X_n^{\oplus}\|_2=\bigO(d^{1/2}n^{1/p}).
	\end{equation}

\subsubsection{$M$-dependence approximation}\label{subsubsec_m_dep}
	
Define the truncated $M$-dependence approximation of $x_i$ as
\[
\tilde{x}_i:=\E[T_{n^{1/p}}(x_i)\,|\, e_i,\dots,e_{i-m}]-\E[T_{n^{1/p}}(x_i)],
\]
where $m$ will be later chosen such that 
	\begin{equation}\label{temp_cond_1}
	\Theta_{m,p}=o(n^{1/p-1/2}).    
	\end{equation}
Defining the partial sum process $\tilde{R}_{c,l}:=\sum_{i=1+c}^{l+c}\tilde{x}_i$ and 
	$\tilde{X}_n:=\tilde{R}_{0,n}$, we will use $\tilde{X}_n$ to approximate $X_n^{\oplus}$. Denote the $j$-th elements of $X_{n}^{\oplus}$ and $\tilde{X}_{n}$ by $X_{n,j}^{\oplus}$ and $\tilde{X}_{n,j}$, respectively, then according to \citet[Eq.(6.8)]{Karmakar2020} and \citet[Lemma A1]{liu2009strong}
\begin{equation}\label{summary_step2}
\begin{split}
\|X_n^{\oplus}-\tilde{X}_n\|_p= (\E[|X_n^{\oplus}-\tilde{X}_n|^p])^{1/p}&\le \left(d^{p/2}\E\left[\frac{1}{d}\sum_{j=1}^d	\left|X_{n,j}^{\oplus}-\tilde{X}_{n,j}\right|^p\right]\right)^{1/p}\\
&=d^{1/2}\bigO(n^{1/2}\Theta_{m,p})=o(d^{1/2}n^{1/p}),    
\end{split}
\end{equation}
where the last equality follows from \cref{temp_cond_1}. %Therefore, we have $|X_n^{\oplus}-\tilde{X}_n|=o_{\pr}(d^{1/2}n^{1/p})$.
	
\subsubsection{Block approximation}\label{subsubsec_block_approx}
We further approximate $\tilde{X}_n$ using block sums $\{A_j\}$ which are defined by
\[
A_{j+1}:=\sum_{i=2jk_0m+1}^{(2k_0j+2k_0)m}\tilde{x}_i,\quad k_0:=\lfloor\Theta_{0,2}^2/\lambda_*\rfloor +2.
\]
We define $X_n^{\diamond}:=\sum_{j=1}^{q_n} A_j$ where $q_n:=\lfloor n/(2k_0m)\rfloor$.
In the following, we follow similar techniques of \citet[Proposition 6.2]{Karmakar2020} to show that under certain conditions of the dependence measure in $M$-dependence approximation, we can approximate $\tilde{X}_n$ by $X_n^{\diamond}$. More precisely, we show that if for some $\gamma>p$ and under the following conditions
\begin{equation}\label{assumption_temp}
\begin{split}
&\Theta_{m,p}=\bigO(m^{-\chi}(\log m)^{-A}),\quad A>\gamma/p,\\
%	&\RED{n^{-\gamma/p}m^{\gamma/2}\to 0,\textrm{JY; remove it}}\\ 
	&(\chi+1)p/\gamma>1,\\
	&1/p-1/\gamma <  ((\chi+1)p/\gamma-1)L,\\
	%&\RED{m=\lfloor n^L (t_n/d^{1/2})^k\rfloor=o(n^L\log\log(n))},\quad 0<k<(\gamma-p)/(\gamma/2-1),
	& m=o(n^L),\quad 0<k<(\gamma-p)/(\gamma/2-1),
	\end{split}
	\end{equation}
we have 
\begin{equation}\label{summary_step3}
\|\tilde{X}_n-X_n^{\diamond}\|_{\gamma}=\bigO(d^{1/2}m^{1/2}).
\end{equation}
	
%\RED{(JY: red color term can simply be replaced by $m=\lfloor n^L\rfloor$ without using $k$)}
The proof follows \citet[Proof of Proposition 6.2]{Karmakar2020} with slightly weaker conditions. The main difference with \citet[Proof of Proposition 6.2]{Karmakar2020} is that we have kept the dependence on the dimension explicit.

Note that using the definition of $\Theta_{m,p}$ and Holder's inequality, we have
\begin{align*}
\max_{1\le j\le d}\sum_{l=m+1}^{\infty}\theta_{l,j,p}^{p/\gamma}&\leq   \max_{1\le j\le d}\sum_{i=\lfloor \log_2 m\rfloor}^{\infty}\sum_{l=2^i}^{2^{i+1}-1}\theta_{l,j,p}^{p/\gamma}\\
&\leq \max_{1\le j\le d}\sum_{i=\lfloor \log_2 m\rfloor}^{\infty}\left(\sum_{l=2^i}^{2^{i+1}-1}\theta_{l,j,p}\right)^{p/\gamma}\left(\sum_{l=2^i}^{2^{i+1}-1}1\right)^{1-p/\gamma}\\
&\leq   \sum_{i=\lfloor \log_2 m\rfloor}^{\infty} 2^{i(1-p/\gamma)}\Theta_{2^i,p}^{p/\gamma}\\
&=\sum_{i=\lfloor \log_2 m\rfloor}^{\infty} 2^{i(1-p/\gamma)}\bigO(2^{-\chi ip/\gamma}i^{-Ap/\gamma})\\
&=\bigO(m^{1-p/\gamma-\chi p/\gamma}(\log m)^{-Ap/\gamma}).
\end{align*}
Then using the assumptions in \cref{assumption_temp} one can easily verify that
\[
	n^{1/p-1/\gamma}\max_{1\le j\le d} \sum_{l=m+1}^{\infty}\theta_{l,j,p}^{p/\gamma}={o(n^{(1/p-1/\gamma)+(1-(\chi+1)p/\gamma)L})=o(1)}.
\]
Recall that the partial sum process has been defined by $\tilde{R}_{t,l}=\sum_{i=1+t}^{l+t}\tilde{x}_i$. We denote the $j$-th element of $\tilde{R}_{t,l}$ as $\tilde{R}_{t,l,j}$. Then, by the Rosenthal-type moment bound in \citet[Lemma 7.3]{Karmakar2020} for the $j$-th element of $\tilde{R}_{t,l}$, we have that for a constant $c_{\gamma}$ depending only on $\gamma$
	\begin{equation}\label{eq_Rosenthal-type}
	    \begin{split}
	        	\left\|\max_{1\le l\le m} |\tilde{R}_{t,l,j}|\right\|_{\gamma} &\le c_{\gamma}m^{1/2}\left[\max_{1\le j\le d}\sum_{k=1}^m \theta_{k,j,2}+\max_{1\le j\le d}\sum_{k=m+1}^{\infty} \theta_{k,j,\gamma}+\sup_{i,j} \|\tilde{x}_{i,j}\|_2\right]\\
	&\quad+c_{\gamma}m^{1/\gamma}\left[\max_{1\le j\le d}\sum_{k=1}^m k^{1/2-1/\gamma}\theta_{k,j,\gamma}+\sup_{i,j}\|\tilde{x}_{i,j}\|_{\gamma}\right]=\bigO(m^{1/2}),
	    \end{split}
	\end{equation}
	where the r.h.s.\ holds uniformly over $1\le j\le d$ and $t$.
Then, we have 
\[
\E\left[\max_{1\leq   l \leq   m}|\tilde{R}_{t,l}|^{\gamma}\right] =\E\left[\max_{1\leq   l \leq   m}\left(\sum_j\tilde{R}_{t,l,j}^2\right)^{\gamma/2}\right]\le d^{\gamma/2-1} \E\left[\max_{1\leq   l \leq   m}\sum_j|\tilde{R}_{t,l,j}|^{\gamma}\right]=\bigO(d^{\gamma/2}m^{\gamma/2}).
\]
Therefore, we have
	\begin{align*}
	    	\|\tilde{X}_n-X_n^{\diamond}\|_{\gamma}= \left\|\tilde{R}_{0,n}-\sum_{i=1}^{q_n}A_i\right\|_{\gamma}
	    	&\le \left(\max_t\E[\max_{1\leq   l \leq   m}|\tilde{R}_{t,l}|^{\gamma}]\right)^{1/\gamma}=\bigO(d^{1/2}m^{1/2}).
	\end{align*}
	This finishes the proof of \cref{summary_step3}.
	%\begin{align*}
	%\pr\left(|\tilde{R}_{0,n}-\sum_{i=1}^{q_n}A_i|\ge  d^{1/2}n^{1/p}\right)&\leq   \frac{\max_t\E[\max_{1\leq   l\leq   2k_0m}|\tilde{R}_{t,l}|^{\gamma}]}{d^{\gamma/2}n^{\gamma/p}}\\
	%&=\bigO(d^{1-\gamma/2}n^{-\gamma/p}m^{\gamma/2})=o(1),
	%\end{align*}
	
	\subsection{Conditional Gaussian Approximation}\label{subsec_cond_GA}

We first follow similar arguments as in \citet[Section 6.2]{Karmakar2020} then apply \cref{key_lemma1}. Note that
	the blocks $\{A_j\}$ are weakly independent because the shared border $e_i$'s between two adjacent blocks.
	
Note that we can write $\tilde{x}_i$ as $\tilde{x}_i=\tilde{\mathcal{G}}_i(e_i,\dots,e_{i-m})$ where $\tilde{\mathcal{G}}_i$ is a measurable function. As in \citet[Section 6.2]{Karmakar2020}, we consider Gaussian approximation for the conditional distribution condition on the border $e_i's$ shared between two adjacent blocks. We will denote the fixed $e_i$ by $a_i$. We first introduce some notations, the border $a_i$'s between two adjacent blocks are denoted by  $\bar{a}_{2k_0j}=\{a_{(2k_0j-1)m+1},\dots, a_{2k_0jm}\}$ which has length of $m$. Then we denote by $a$ all the border $a_i$'s, i.e., $a=\{\dots, \bar{a}_0,\bar{a}_{2k_0},\bar{a}_{4k_0},\dots\}$. Next, we define the version of $\tilde{x}_i$ where the border $a_i$'s are fixed. For $2k_0jm+1\leq   i\leq   (2k_0j+1)m$, we let $\tilde{x}_i(\bar{a}_{2k_0j})=\tilde{\mathcal{G}}_i(e_i,e_{i-1},\dots,e_{2k_0jm+1}, a_{2k_0jm},\dots, a_{i-m})$; 	for $(2k_0j+2k_0-1)m+1\leq   i\leq   (2k_0j+2k_0)m$, we let $		\tilde{x}_i(\bar{a}_{2k_0j+2k_0})=\tilde{\mathcal{G}}_i(a_i,a_{i-1},\dots,a_{(2k_0j+2k_0-1)m+1},e_{(2k_0j+2k_0-1)m},\dots,e_{i-m})$.
	
	Recall $A_{j+1}=\sum_{i=2k_0jm+1}^{(2k_0j+2k_0)m}\tilde{x}_j$. We divide the sum into four sub-block sums with lengths $(m,(k_0-1)m,(k_0-1)m,m)$, respectively.
	\begin{align*}
	&F_{4j+1}(\bar{a}_{2k_0j})=\sum_{i=2k_0jm+1}^{(2k_0j+1)m}\tilde{x}_i(\bar{a}_{2k_0j}),\quad
	F_{4j+2}=\sum_{i=(2k_0j+1)m+1}^{(2k_0j+k_0)m}\tilde{x}_i,\\
	&F_{4j+3}=\sum_{i=(2k_0j+k_0)m+1}^{(2k_0j+2k_0-1)m}\tilde{x}_i,\quad
	F_{4j+4}(\bar{a}_{2k_0j+2k_0})=\sum_{i=(2k_0j+2k_0-1)m+1}^{(2k_0j+2k_0)m}\tilde{x}_i(\bar{a}_{2k_0j+2k_0}).
	\end{align*}
	Note that the first and the last sub-blocks contains the borderline $a_i$'s. By replacing non-random $a_i$'s by random $e_i$'s and non-random $\bar{a}_{2k_0j}$ by random $\bar{e}_{2k_0j}$, we can recover $A_{j+1}$ as the sum of the four sub-block sums:
	\[
	A_{j+1}=F_{4j+1}(\bar{e}_{2k_0j})+F_{4j+2}+F_{4j+3}+F_{4j+4}(\bar{e}_{2k_0j+2k_0}).
	\]
	When conditional on border $e_i=a_i$, we use $F_{4j+1}(\bar{a}_{2k_0j})$ and $F_{4j+4}(\bar{a}_{2k_0j+2k_0})$ for the first and last sub-blocks. However, the first and last sub-block sums may not have zero means anymore. Therefore, we need to center $A_{j+1}$ by removing the conditional mean functions of $F_{4j+1}$ and $F_{4j+4}$. Define $\Lambda_{4j+1}:=\E[F_{4j+1}\,|\,a]$ and $\Lambda_{4j+4}:=\E[F_{4j+4}\,|\, a]$. Then we center $A_{j+1}$ conditional on $a$ by removing the mean functions:
	\begin{align*}
	y_j(\bar{a}_{2k_0j},\bar{a}_{2k_0j+2k_0})&:=
	[F_{4j+1}(\bar{a}_{2k_0j})-\Lambda_{4j+1}(\bar{a}_{2k_0j})]+F_{4j+2}\\
	&+F_{4j+3}+[F_{4j+4}(\bar{a}_{2k_0j+2k_0})-\Lambda_{4j+4}(\bar{a}_{2k_0j+2k_0})].
	\end{align*}
	Therefore, we can rewrite $A_{j+1}$ as $A_{j+1}=y_j(\bar{e}_{2k_0j},\bar{e}_{2k_0j+2k_0})+\Lambda_{4j+1}(\bar{e}_{2k_0j})+\Lambda_{4j+4}(\bar{e}_{2k_0j+2k_0})$. We will first condition on $a$ and ignore the mean functions $\Lambda_{4j+1}$ and $\Lambda_{4j+4}$ in this subsection. The mean functions will be added back in \cref{subsec_uncond_GA}.
	
	We apply the same truncation arguments in \cref{subsubsec_truncate} to $y_j(\bar{a}_{2k_0j},\bar{a}_{2k_0j+2k_0})$. Define
	\[
	y_j^a:=T_{n^{\frac{1}{\gamma}}m^{\frac{1}{2}-\frac{1}{\gamma}}(\log m)}\left(y_j(\bar{a}_{2k_0j},\bar{a}_{2k_0j+2k_0})\right)-\E\left[T_{n^{\frac{1}{\gamma}}m^{\frac{1}{2}-\frac{1}{\gamma}}(\log m)}\left(y_j(\bar{a}_{2k_0j},\bar{a}_{2k_0j+2k_0})\right)\right].
	\]
%	Recall that $X_n^{\diamond}=\sum_{i=1}^{q_n}A_i$. We denote $X_n^{\diamond}(a)=\sum_{j=1}^{q_n}A_j(a)$, where $A_{j+1}(a):=y_j(\bar{a}_{2k_0j},\bar{a}_{2k_0j+2k_0})+\Lambda_{4j+1}(\bar{a}_{2k_0j})+\Lambda_{4j+4}(\bar{a}_{2k_0j+2k_0})$. 
We approximate $\sum_{j=0}^{q_n-1} y_j(\bar{a}_{2k_0j},\bar{a}_{2k_0j+2k_0})$ by 
	$X_n(a):=\sum_{i=0}^{q_n-1}y_i^a$ using \cref{eq_Rosenthal-type} following the same arguments as in \cref{subsubsec_truncate}. %Then, we have
	%\[\Pr(|y_j(\bar{a}_{2k_0j},\bar{a}_{2k_0j+2k_0})|\ge  d^{1/2}n^{\frac{1}{\gamma}}m^{\frac{1}{2}-\frac{1}{\gamma}}(\log m)) =o(m^{\gamma/2}/(nm^{\gamma/2-1}))=o(m/n).\]
	Using $q_n=\bigO(n/m)$ and applying \cref{lem:l2truncation} coordinate-wisely, we have
	\begin{equation}\label{eq_adjust_cov}
	    \begin{split}
	        &\left\|\sum_{i=1}^{q_n} y_i(\bar{a}_{2k_0j},\bar{a}_{2k_0i+2k_0})-X_n(a)\right\|_2
            %\\
	      %  &\le \sum_{i=1}^{q_n}\| y_i(\bar{a}_{2k_0j},\bar{a}_{2k_0i+2k_0})-y_i^a\|_2\\
	    %&=q_n \bigO((m/n) d^{1/2}n^{1/\gamma}m^{1/2-1/\gamma}(\log m))\\
	    =\bigO(d^{1/2}n^{1/\gamma}m^{1/2-1/\gamma}(\log m)).
	    \end{split}
	\end{equation}
	Note that the key observation is that $y_i^a$'s are independent conditional on $a$ (if non-random $a_i$'s are replaced by $e_i$'s then the corresponding $y_i^{e}$'s are $1$-dependent). Therefore, we can apply \cref{key_lemma1}. Since there are $q_n=\lfloor n/(2k_0m)\rfloor$ (conditional) independent random vectors $\{y_i^a\}$. 
    
   Note that each $|y_i^a|$ has been truncated by $d^{1/2}n^{\frac{1}{\gamma}}m^{\frac{1}{2}-\frac{1}{\gamma}}(\log m)$.  By \cref{key_lemma1}, we can construct $Y_n(a)$ which is (conditional) Gaussian with the same covariance matrix as $X_n(a)$, such that
   \[
   \|X_n(a)-Y_n(a)\|_2^2\le C \left(d^2(\log n)\left(n^{1/\gamma}m^{1/2-1/\gamma}(\log m)\right)^2\right),
   \]
   for some constant $C$.
   
  %  \PURPLE{we can construct $X_n^a$ and $Y_n'(a)$ in the same probability space such that 
	%\begin{equation*}
 %\begin{split}
  %  &\pr\left(|X_n^a-Y_n'(a)|>\epsilon\right)\\
   % &\leq C\min \left\{ \frac{d^2({n^{\frac{1}{\gamma}}m^{\frac{1}{2}-\frac{1}{\gamma}}(\log m)})^2\log(n)}{\epsilon^2}, d^{2}\exp\left(-\frac{\epsilon}{C'd^{\frac{5}{2}}({n^{\frac{1}{\gamma}}m^{\frac{1}{2}-\frac{1}{\gamma}}(\log m)})}\right)\right\} ,
 %\end{split}
%	\end{equation*}
%	where $C$ and $C'$ are some constants, $X_n^a\stackrel{D}{=} X_n(a)$ and $Y_n'(a)$ is (conditional) Gaussian with the same covariance matrix as $X_n(a)$.} 
    
    Note that the above upper bound doesn't depend on $a$. However, $\Cov(Y_n(a))=\Cov\left(\sum_{j=0}^{q_n-1}y_j^a\right)\neq \Cov\left(\sum_{j=0}^{q_n-1} y_j(\bar{a}_{2k_0j},\bar{a}_{2k_0j+2k_0})\right)$ Next, we construct $Y_n^a$ based on $Y_n(a)$ such that 
    \[
    Y_n^a:=\Cov\left(\sum_{j=0}^{q_n-1} y_j(\bar{a}_{2k_0j},\bar{a}_{2k_0j+2k_0})\right)^{1/2}\Cov\left(\sum_{j=0}^{q_n-1}y_j^a\right)^{-1/2}Y_n(a).
    \] Then because of \cref{eq_adjust_cov}, either by \citet[Eq.(3.18)]{wu2011gaussian} or by similar arguments as in \cref{eq_correct_covariance}, we have
	$\|Y_n^a-Y_n(a)\|_2=\bigO(d^{1/2}{n^{1/\gamma}m^{1/2-1/\gamma}(\log m)})$.
	Finally, note that if we replace non-random $a_i$'s by random $e_i$'s and denote the randomized $Y_n^a$ by $Y_n^{e}$, then $Y_n^{e}$ follows from a mixture Gaussian distribution where each Gaussian component has mean $0$. %Furthermore, the above upper bound still holds since $\pr\left(|X_n^e-Y_n^e|>\epsilon\right)\le \sup_a\pr\left(|X_n^a-Y_n^a|>\epsilon\right)$.
	
	\subsection{Unconditional Gaussian Approximation}\label{subsec_uncond_GA}

	We follow similar arguments as in \citet[Section 6.3]{Karmakar2020} then apply \cref{key_lemma1} again. Note that we have constructed $Y_n^a$ in the conditional Gaussian approximation in \cref{subsec_cond_GA} with
	\[
	\Cov(Y_n^a)=\Cov\left(\sum_{j=0}^{q_n-1} y_j(\bar{a}_{2k_0j},\bar{a}_{2k_0j+2k_0})\right)=\sum_{j=0}^{q_n-1}\Cov(y_j(\bar{a}_{2k_0j},\bar{a}_{2k_0j+2k_0})).
	\]
	We define $V_j(\bar{a}_{2k_0j},\bar{a}_{2k_0j+2k_0}):=\Cov(y_j(\bar{a}_{2k_0j},\bar{a}_{2k_0j+2k_0}))$. Note that, unconditionally, the sequence of $V_j(\bar{e}_{2k_0j},\bar{e}_{2k_0j+2k_0})$ is $1$-dependent. Dimension-dependent GA for $1$-dependent time series has been recently studied in \citet{pengel2024gaussian}. However, we can extensively use the structure of Bernoulli shifts to get a better rate (in terms of $n$) than \citet{pengel2024gaussian}. We regroup the sum of $V_j$ to separate the dependence on ``$\bar{a}_{2k_0j}$'' and ``$\bar{a}_{2k_0j+2k_0}$''.
	\[
		\Cov(Y_n^a)=\sum_{j=0}^{q_n-1} V_j(\bar{a}_{2k_0j},\bar{a}_{2k_0j+2k_0})=L(\bar{a}_0)+\sum_{j=1}^{q_n-2}V_{j0}(\bar{a}_{2k_0j})+U_{q_n-1}(\bar{a}_{2k_0q_n}),
	\] 
	where
	\begin{align*}
	V_{j0}(\bar{a}_{2k_0j}):&=\E[F_{4j-2}F^T_{4j-1}+F_{4j-1}F^T_{4j-2}]\\
	&\quad+\Cov(F_{4j-1}+F_{4j}(\bar{a}_{2k_0j})-\Lambda_{4j}(\bar{a}_{2k_0j}))\\
	&\quad+\Cov(F_{4j+1}(\bar{a}_{2k_0j})-\Lambda_{4j+1}(\bar{a}_{2k_0j})+F_{4j+2}),\\
	L(\bar{a}_0):&=\Cov(F_1(\bar{a}_0)+F_2),\\
	U_{q_n-1}(\bar{a}_{2k_0q_n}):&=\E[F_{4q_n-2}F^T_{4q_n-1}+F_{4q_n-1}F^T_{4q_n-2}]\\
	&\quad + \Cov(F_{4q_n-1}+F_{4q_n}(\bar{a}_{2k_0q_n})-\Lambda_{4q_n}(\bar{a}_{2k_0q_n})).
	\end{align*}
	By \citet[Lemma 7.3]{Karmakar2020} we have
	%\[
	%\pr\left(|U_{q_n-1}(\bar{e}_{2k_0q_n})|\ge d^{1/2}m^{1/2}\right)=\bigO(m^{-\gamma/2}m^{\gamma/2})=\bigO(1).
	%\]
	$\|U_{q_n-1}(\bar{e}_{2k_0q_n})\|_{\gamma/2}=\bigO(d^{1/2}m^{1/2})$, and similarly  $\|L(\bar{e}_0)\|_{\gamma/2}=\bigO(d^{1/2}m^{1/2})$. This implies that we can approximate  $\sum_{j=0}^{q_n-1} V_j(\bar{e}_{2k_0j},\bar{e}_{2k_0j+2k_0})$ by $\sum_{j=1}^{q_n-2}V_{j0}(\bar{e}_{2k_0j})$. Note that the later term is a sum of independent matrices $\{V_{j0}(\bar{e}_{2k_0j})\}$.
	
	Now we have two issues left. First, the matrices $\{V_{j0}(\bar{a}_{2k_0j})\}$ may not be positive semi-definite (p.s.d.), which means they may not be valid covariance matrices. Second, we haven't added the mean functions back to the approximation. Next, we follow similar arguments as in \citet[Proposition 6.9]{Karmakar2020} with slightly weaker conditions to approximate $V_{j0}(\bar{a}_{2k_0j})$ by another matrix which is p.s.d.
	
	For a constant $0<\delta_*<\lambda_*$, define
	\[
	V_{j1}(\bar{a}_{2k_0j}):=
	\begin{cases}
	V_{j0}(\bar{a}_{2k_0j}),&\quad\textrm{if $V_{j0}^a-\delta_*m$ is p.s.d.} \\
	(\delta_*m)I_d,&\quad \textrm{otherwise}.
	\end{cases}
	\]
	Then by following the same proof of \citet[Proposition 6.9]{Karmakar2020}, we have that by choosing some $0<\delta_*<(k_0-1)\lambda_*$, we have
	\[
	\E\left[\left|\sum_{j=1}^{q_n-2}V_{j0}(\bar{a}_{2k_0j})-\sum_{j=1}^{q_n-2}V_{j1}(\bar{a}_{2k_0j})\right|\right]=o(d^{1/2}n^{2/p}),
	\]
	provided that
		\begin{equation}\label{temp_cond_2}
		1-\gamma L/2<2/p,\quad 1-\chi\gamma L < 2/p.
		\end{equation}
	Therefore, using \cref{key_lemma3}, we can approximate the (conditional) Gaussian vector $Y_n^a$ by another (conditional) Gaussian vector $\tilde{Y}_n^a$ whose covariance matrix equals to $\sum_{j=1}^{q_n-2}V_{j1}(\bar{a}_{2k_0j})$, such that
	\[
	\|Y_n^{a}-\tilde{Y}_n^a\|_2\leq   \frac{1}{(\delta_* m(q_n-2))^{1/2}}\E\left[\left|\sum_{j=1}^{q_n-2}V_{j0}(\bar{a}_{2k_0j})-\sum_{j=1}^{q_n-2}V_{j1}(\bar{a}_{2k_0j})\right|\right]=o(d^{1/2}n^{\frac{2}{p}-\frac{1}{2}}).
	\]
	 Note that if we replace the non-random $a_i$'s by the random $e_i$'s then
	\[
	\Cov(\tilde{Y}_n^{e})=\E\left[\sum_{j=1}^{q_n-2}V_{j1}(\bar{e}_{2k_0j})\right]=\sum_{j=1}^{q_n-2}\E[V_{j1}(\bar{e}_{2k_0j})],
	\]
	where $\E[V_{j1}(\bar{e}_{2k_0j})]$ is p.s.d.~by construction. Hence, the covariance matrix of the unconditional $\tilde{Y}_n^{e}$ can be written as a sum of covariance matrices $\{\E[V_{j1}(\bar{e}_{2k_0j})]\}_{j=1}^{q_n-2}$. Note that the unconditional $\tilde{Y}_n^{e}$ follows a mixture Gaussian distribution in which each Gaussian component has mean $0$. Without loss of generality, we can write
	\[
	\tilde{Y}_n^{e}=\sum_{j=1}^{q_n-2}\tilde{y}_j^{e},
	\]
	where $\tilde{y}_j^{e}$ is a mixture Gaussian random vector such that the conditional version $\tilde{y}_j^{a}\sim \mathcal{N}_d(0,V_{j1}(\bar{a}_{2k_0j})))$ and $\Cov(\tilde{y}_j^{e})=\E[V_{j1}(\bar{e}_{2k_0j})]$.

	Next, we add back the mean functions $\Lambda_{4j+1}$ and $\Lambda_{4j+4}$. The unconditional version of the sum of mean functions is $\sum_{j=0}^{q_n-1}\left[ \Lambda_{4j+1}(\bar{e}_{2k_0j}) +\Lambda_{4j+4}(\bar{e}_{2k_0j+2k_0}) \right]$.
	We need to regroup the terms to separate $\bar{e}_{2k_0j}$ and $\bar{e}_{2k_0(j+1)}$:
	\begin{align*}
	&\sum_{j=0}^{q_n-1}\left[ \Lambda_{4j+1}(\bar{e}_{2k_0j}) +\Lambda_{4j+4}(\bar{e}_{2k_0j+2k_0}) \right]\\
	&=\Lambda_1(\bar{e}_0) +\sum_{j=0}^{q_n-2} \left[\Lambda_{4j+1}(\bar{e}_{2k_0j+2k_0})+\Lambda_{4j+4}(\bar{e}_{2k_0j+2k_0}) \right] +\Lambda_{4q_n}(\bar{e}_{2k_0q_n}).
	\end{align*}
	Since $\|\Lambda_{4q_n}(\bar{e}_{2k_0q_n})+\Lambda_1(\bar{e}_0)\|_{\gamma/2}=\bigO( d^{1/2}m^{1/2})$, we can approximate the sum of unconditional mean functions by
	\[
	\sum_{j=0}^{q_n-2} \left[\Lambda_{4j+1}(\bar{e}_{2k_0j+2k_0})+\Lambda_{4j+4}(\bar{e}_{2k_0j+2k_0}) \right]
	\] 
	which is a sum of independent random vectors.
	
	Overall, by adding the mean functions back, letting $\tilde{y}_0^e=0$, we can approximate $X_n^{\diamond}$ by
	\begin{align*}
	&\tilde{Y}_n^{e}+\sum_{j=0}^{q_n-2} \left[\Lambda_{4j+1}(\bar{e}_{2k_0j+2k_0})+\Lambda_{4j+4}(\bar{e}_{2k_0j+2k_0}) \right]\\
	&=\sum_{j=1}^{q_n-2}\tilde{y}_j^{e}+\sum_{j=0}^{q_n-2} \left[\Lambda_{4j+1}(\bar{e}_{2k_0j+2k_0})+\Lambda_{4j+4}(\bar{e}_{2k_0j+2k_0}) \right]\\
	&=\sum_{j=0}^{q_n-2}\tilde{y}_j^{e}+ \left[\Lambda_{4j+1}(\bar{e}_{2k_0j+2k_0})+\Lambda_{4j+4}(\bar{e}_{2k_0j+2k_0}) \right],
	\end{align*}
	which is a sum of $\bigO(q_n)$ {independent} random vectors. 
    
    Next, we use the truncation arguments in \cref{subsubsec_truncate} and \cref{subsec_cond_GA} based on \cref{lem:l2truncation} again to $\tilde{y}_j^{e}+ \left[\Lambda_{4j+1}(\bar{e}_{2k_0j+2k_0})+\Lambda_{4j+4}(\bar{e}_{2k_0j+2k_0}) \right]$ then apply \cref{key_lemma1} again. Same as in \cref{subsubsec_truncate} and \cref{subsec_cond_GA}, we use truncation level $d^{1/2}n^{1/\gamma}m^{1/2-1/\gamma}(\log m)$ to $\tilde{y}_j^{e}+ \left[\Lambda_{4j+1}(\bar{e}_{2k_0j+2k_0})+\Lambda_{4j+4}(\bar{e}_{2k_0j+2k_0}) \right]$, it follows the same arguments as in \cref{subsubsec_truncate} and \cref{subsec_cond_GA} to verify the overall truncation error is $\bigO(d^{1/2}n^{1/\gamma}m^{1/2-1/\gamma}(\log m))$. Therefore, we can approximate $X_n^{\diamond}$ by $\hat{X}_n$ such that
	\begin{equation}\label{eq:approx_X_hat_diamond}
	    \|\hat{X}_n-X^{\diamond}_n\|_2=\bigO(d^{1/2}n^{1/\gamma}m^{1/2-1/\gamma}(\log m))+\bigO(d^{1/2}m^{1/2}),
	\end{equation}
	where
	\begin{align*}
    \hat{X}_n:&=\sum_{j=0}^{q_n-2}\left[T_{n^{\frac{1}{\gamma}}m^{\frac{1}{2}-\frac{1}{\gamma}}(\log m)}(\hat{y}_j^e)-\E\left[T_{n^{\frac{1}{\gamma}}m^{\frac{1}{2}-\frac{1}{\gamma}}(\log m)}(\hat{y}_j^e)\right]\right],
	\end{align*}
	where $\hat{y}_j^e:=\tilde{y}_j^{e}+ \left[\Lambda_{4j+1}(\bar{e}_{2k_0j+2k_0})+\Lambda_{4j+4}(\bar{e}_{2k_0j+2k_0}) \right]$.
Then by \cref{key_lemma1}, one can construct a Gaussian random vector $\hat{Y}_n$ such that $\Cov(\hat{Y}_n)=\Cov(\hat{X}_n)$ and
	\begin{equation}\label{summary_step4}
    \E\left[|\hat{X}_n-\hat{Y}_n|^2\right]=\bigO(d^2(\log n)(n^{\frac{1}{\gamma}}m^{\frac{1}{2}-\frac{1}{\gamma}}(\log m))^2).
 %\begin{split}
%	&\pr(|\tilde{X}_n^c-\tilde{Y}_n^c|>\epsilon)\\
 %&\leq C \min\left\{\frac{d^2({n^{\frac{1}{\gamma}}m^{\frac{1}{2}-\frac{1}{\gamma}}(\log m)})^2\log(n)}{\epsilon^2},   d^{2}\exp\left(-\frac{\epsilon}{C'd^{5/2}({n^{\frac{1}{\gamma}}m^{\frac{1}{2}-\frac{1}{\gamma}}(\log m)})}\right)\right\},
 %\end{split}
	\end{equation}

	\subsection{Summary}\label{subsec_summary}

	We have shown in \cref{summary_step1,summary_step2} that
	\[
	  \|X_n-X_n^{\oplus}\|_2=\bigO(d^{1/2}n^{1/p}),\quad
\|X_n^{\oplus}-\tilde{X}_n\|_p=o(d^{1/2}n^{1/p}). 
	\]
	Furthermore, we have in \cref{summary_step3} and \cref{eq:approx_X_hat_diamond} that
	\[
	\|\tilde{X}_n-X_n^{\diamond}\|_{\gamma}=\bigO(d^{1/2}m^{1/2}), \quad \|X_n^{\diamond}-\hat{X}_n\|_2=\bigO(d^{1/2}{n^{1/\gamma}m^{1/2-1/\gamma}(\log m)})+\bigO(d^{1/2}m^{1/2}).
	\]
	Moreover, in \cref{summary_step4}, we have the Gaussian approximation by $\hat{Y}_n$
	\[
	\|\hat{X}_n-\hat{Y}_n\|_2=\bigO(d\sqrt{\log(n)}{n^{1/\gamma}m^{1/2-1/\gamma}(\log m)}).
	\]
 
	Finally, note that the covariance matrix of $\hat{Y}_n$ doesn't equal the covariance matrix of $X_n$. We must construct another Gaussian random vector $Y_n$ based on $\hat{Y}_n$ to fix the covariance matrix mismatch. Writing $\hat{Y}_n=\Cov(\hat{X}_n)^{1/2}\tilde{Z}$ where $\tilde{Z}\sim\mathcal{N}(0,I_d)$, we define $Y_n:=\Cov(X_n)^{1/2}\tilde{Z}$. Then by \cref{key_lemma3},  we have
 \begin{align*}
 \E[|\hat{Y}_n-Y_n|^2]\le \left|\Cov(\hat{X}_n)-\Cov(X_n)\right|^2\frac{1}{n\lambda_*}.
\end{align*}
Next, using properties of Frobenius norm and Cauchy--Schwartz inequality
 \begin{equation}\label{eq_correct_covariance}
	    \begin{split}
	        	&\|Y_n-\hat{Y}_n\|_2= \sqrt{\E[|Y_n-\hat{Y}_n|^2]}\le  \frac{1}{(\lambda_*n)^{1/2}}   \left|\E[X_nX_n^T]-\E[\hat{X}_n\hat{X}_n^T]\right|\\
 &{=\frac{1}{(\lambda_*n)^{1/2}}   \left|\E[\hat{X}_n(X_n-\hat{X}_n)^T]+\E[(X_n-\hat{X}_n)\hat{X}_n^T]+\E[(X_n-\hat{X}_n)(X_n-\hat{X}_n)^T]\right|}\\
 &{\le\frac{1}{(\lambda_*n)^{1/2}}   \left|\E[\hat{X}_n(X_n-\hat{X}_n)^T]\right|+\left|\E[(X_n-\hat{X}_n)\hat{X}_n^T]\right|+\left|\E[(X_n-\hat{X}_n)(X_n-\hat{X}_n)^T]\right|}\\
 &{\le\frac{1}{(\lambda_*n)^{1/2}}  \left[2\sqrt{\E[|\hat{X}_n|^2]}\sqrt{\E[|X_n-\hat{X}_n|^2]}+\left|\E[(X_n-\hat{X}_n)(X_n-\hat{X}_n)^T]\right|\right]}\\	
  &{\le\frac{1}{(\lambda_*n)^{1/2}}  \left[2\sqrt{\E[|\hat{X}_n|^2]}\sqrt{\E[|X_n-\hat{X}_n|^2]}+\E[|X_n-\hat{X}_n|^2]\right]}\\	&=\bigO(n^{-1/2})\left[\bigO(d^{1/2}n^{1/2})(d^{1/2}[o(n^{{1/r}})+\bigO(m^{1/2})])+d[o(n^{{2/r}})+\bigO(m)]\right]\\
&=\bigO(n^{-1/2})\left[\bigO(d^{1/2}n^{1/2})(d^{1/2}[o(n^{{1/r}})+\bigO(m^{1/2})])+d[o(n^{{2/r}})+\bigO(m)]\right]\\
%	&=\bigO(d^{2-1/p}[n^{1/p}+m^{1/2}]+d^{5/2-2/p}[n^{2/p-1/2}+mn^{-1/2}])\\
	&=o(d(n^{{1/r}}+n^{2/r-1/2}))+\bigO(dm^{1/2})+\bigO(dn^{-1/2}m),
	    \end{split}
	\end{equation}
 {where we choose $r$ later to satisfy $n^{1/r}\ge \max\{n^{1/p}, n^{1/\gamma}m^{1/2-1/\gamma}(\log m)^2\}$.}
	Putting everything together, we have a Gaussian vector $Y_n$ with the same covariance matrix as $X_n$ such that
	\begin{align*}
	    \|X_n-Y_n\|_2&=\bigO(d\log(n)n^{{1/r}})+\bigO(dm^{1/2})+\bigO(dn^{-1/2}m).
	\end{align*}
	To complete the proof, we need to determine $m$ by summarizing all the required conditions for $m$. 
	First, $M$-dependence step requires \cref{temp_cond_1} which is
		\[
		\Theta_{m,p}=o(n^{1/p-1/2}).
		\]
	Then, the block approximation step requires for some $\gamma>p>2$ that \cref{assumption_temp} holds. That is,
		\begin{align*}
		&\Theta_{m,p}=\bigO(m^{-\chi}(\log m)^{-A}),\quad A>\gamma/p,\\
		%&\RED{n^{-\gamma/p}m^{\gamma/2}\to 0, \textrm{remove it}}\\ 
		&(\chi+1)p/\gamma>1,\\
		&1/p-1/\gamma <  ((\chi+1)p/\gamma-1)L,\\
		%&\RED{m=\lfloor n^L (t_n/d^{1/2})^k\rfloor=o(n^L\log\log(n)),\quad 0<k<(\gamma-p)/(\gamma/2-1),}
		& m=\lfloor n^Lt_n^k\rfloor=o(n^L),\quad 0<k<(\gamma-p)/(\gamma/2-1).
		\end{align*}
	%	\RED{(JY: check the red color might not be true anymore)}
	Finally, the unconditional Gaussian approximation step requires \cref{temp_cond_2}, which is
		\[
		1-\gamma L/2<2/p,\quad 1-\chi\gamma L < 2/p.
		\]
	Overall, the following conditions are sufficient:
		\begin{align*}
		&p<\gamma<(\chi+1)p,\quad p>2,\quad \frac{1}{\chi}\left(\frac{1}{2}-\frac{1}{p}\right)-\log\log(n^L) \le L,\\
		&\max\left\{\frac{\max\{2,1/\chi\}}{\gamma}\left(1-\frac{2}{p}\right), \frac{\gamma/p-1}{(\chi+1)p-\gamma}\right\} < L.
		\end{align*}
		Therefore, choosing $\gamma=\sqrt{\chi+1}p$, we have $A>\gamma/p$ by the assumption $A>\sqrt{\chi+1}$. Overall, we can choose $L$ small enough such that
		\[
		L<\max\left\{\frac{2}{\sqrt{\chi+1}p}, \frac{1}{\chi}\left(\frac{1}{2}-\frac{1}{p}\right)\right\}=:\kappa.
		\]
		Then we have $m(\log m)=o(n^{\kappa})$. Using $1/p>1/(\sqrt{\chi+1}p)$ we have
	\begin{equation*}
 \begin{split}
	\|X_n-Y_n\|_2&=\bigO(d\log(n)n^{{1/r}})+o(dn^{\kappa/2})
    % &=o_{\pr}(\RED{\min\{d^{5/2}\log(d),d\log(n)\}}n^{\max\{1/p, {(1/\gamma+(1/2-1/\gamma)\kappa}\}}+d^{3/2}n^{\kappa/2})\\
     =\bigO(d\log(n) n^{1/r})
	\end{split}
	\end{equation*}
	where {$1/r= \max\left\{1/p, \frac{1}{\sqrt{\chi+1}p}+\left(\frac{1}{2}-\frac{1}{\sqrt{\chi+1}p}\right)\max\left\{\frac{2}{\sqrt{\chi+1}p}, \frac{1}{\chi}\left(\frac{1}{2}-\frac{1}{p}\right)\right\}\right\}$}.
	This completes the proof of the most general case in \cref{main_borel}.

\subsection{The other cases in \cref{main_borel}}
In this subsection, we briefly discuss the differences of the proofs for the other cases. The high-level picture is: if $\{x_i\}$ satisfy a finite exponential moment uniformly, we only need to (i) adjust all the truncation levels in the proof, replacing level $\bigO(n^{1/p})$ by level $\bigO((\log n)^2)$; and (ii) choose $m$ such that $\Theta_{m,p}=o(n^{-1/2})$ instead of $\Theta_{m,p}=o(n^{1/p-1/2})$ in the $M$-dependence approximation in \cref{subsubsec_m_dep}. If the dependence measure of $\{x_i\}$ has a faster decay, satisfying \cref{asm:dependence}(b) instead of \cref{asm:dependence}(a), the proof actually becomes simpler: we only need to choose $m=\bigO(\log n)$ to satisfy either $\Theta_{m,p}=o(n^{1/p-1/2})$ or $\Theta_{m,p}=o(n^{-1/2})$. We give more details in the following.

\subsubsection{Under a finite exponential moment uniformly}
When $\{x_i\}$  have a uniformly finite exponential moment, the main differences are the truncation arguments:
\begin{itemize}
    \item In \cref{subsec_prepare}, the preparation step, we first adjust the truncation level from $\bigO(d^{1/2}n^{1/p})$ to $\bigO(d^{1/2}(\log n)^2)$. 
    We define
    \[
    X_n^{\oplus}:=\sum_{i=1}^n\left[T_{C_{\infty}(\log n)^2}(x_i)-\E\left[T_{C_{\infty} (\log n)^2}(x_i)\right]\right].
    \]
    Applying \cref{lem:l2truncation} coordinate-wisely with $M\sim (\log n)^2$, we can get
    \[
    \|X_n-X_n^{\oplus}\|_2=\bigO(d^{1/2}(\log n)^2).
    \]
    
    Next, in the $M$-dependence approximation, we follow the same steps as in \cref{subsubsec_m_dep}, except we choose $m$ later to satisfy
    \[
    \Theta_{m,p}=o(n^{-1/2}),
    \]
    for any finite $p$. Then, defining the truncated $M$-dependence approximation of $x_i$ as
  \[
	\tilde{x}_i:=\E[T_{C_{\infty}(\log n)^2}(x_i)\,|\, e_i,\dots,e_{i-m}]-\E[T_{C_{\infty}(\log n)^2}(x_i)],
	\]
    and the same definition for $\tilde{X}_n$ as 
	$\tilde{X}_n:=\tilde{R}_{0,n}$ where $\tilde{R}_{c,l}:=\sum_{i=1+c}^{l+c}\tilde{x}_i$, we can directly get
    \[
    \|X_n^{\oplus}-\tilde{X}_n\|_p=o(d^{1/2}(\log n)^2),
    \]
    for any finite $p$. The the block approximation in \cref{subsubsec_block_approx} stays the same, by the same definition of $X_n^{\diamond}$ for any given finite $p$ we can choose $\gamma=p$ to get
    \[
    \|\tilde{X}_n-X_n^{\diamond}\|_p=\bigO(d^{1/2}m^{1/2}).
    \]
\item In \cref{subsec_cond_GA}, the conditional Gaussian approximation step, the only thing we need to adjust is again the truncation level. We replace the truncation level of $y_j$ to $y_j^a$ from $\bigO(d^{1/2}(n/m)^{1/\gamma}m^{1/2}(\log m))$ to $\bigO(d^{1/2}(\log(n/m)^2m^{1/2}(\log m))$. Then, using the uniform finite exponential moment and following the same arguments as in \cref{subsec_cond_GA}, we can construct $X_n^a$ and $Y_n'(a)$ the same way as in \cref{subsec_cond_GA} satisfying
\[
   \|X_n(a)- Y_n^a\|_2^2\le Cd^2(\log n)\left({(\log (n/m))^2m^{\frac{1}{2}}(\log m)}\right)^2,
   \]
for some constant $C$ and $Y_n^a$ is a (conditional) Gaussian vector with the same mean and covariance matrix as $X_n(a)$. 
\item In the unconditional Gaussian approximation in \cref{subsec_uncond_GA}, using uniform finite exponential moment and the new truncations level, by following the same proof of \citet[Proposition 6.9]{Karmakar2020}, we have that for any $p$ and some $0<\delta_*<(k_0-1)\lambda_*$
	\[
	\E\left[\left|\sum_{j=1}^{q_n-2}V_{j0}(\bar{a}_{2k_0j})-\sum_{j=1}^{q_n-2}V_{j1}(\bar{a}_{2k_0j})\right|\right]=o(d^{1/2}(\log n)^2).
	\]
Then, we adjust the truncation level in the same way again to $\hat{y}_j^e$  from $\bigO(d^{1/2}(n/m)^{1/\gamma}m^{1/2}(\log m))$ to $\bigO(d^{1/2}(\log(n/m)^2m^{1/2}(\log m))$ to get  
\[
  \E\left[|\hat{X}_n-\hat{Y}_n|^2\right]=\bigO(d^2(\log n){(\log (n/m))^2m^{\frac{1}{2}}(\log m)})^2).
\]
\item In summary, for any finite $p>2$, we get
\begin{align*}
&\|X_n-X_n^{\oplus}\|_2=\bigO(d^{1/2}(\log n)^2),\quad \|X_n^{\oplus}-\tilde{X}_n\|_p=o(d^{1/2}(\log n)^2),\\
&\|\tilde{X}_n-X_n^{\diamond}\|_p=\bigO(d^{1/2}m^{1/2}), \|X_n^{\diamond}-\hat{X}_n\|_2=\bigO(d^{1/2}(\log(n/m))^2m^{1/2}(\log m))+\bigO(d^{1/2}m^{1/2})\\
&\|\hat{X}_n-\hat{Y}_n\|_2=\bigO(d \sqrt{\log n}(\log(n/m))^2m^{1/2}(\log m)).
\end{align*}
We can put everything together as in \cref{subsec_summary} under uniform finite exponential moment applying \cref{key_lemma3} to get
\[
\|Y_n-\hat{Y}_n\|_2=\bigO(d(\log n)^2 + (\log n)^4n^{-1/2}) + \bigO(dm^{1/2})+\bigO(dn^{-1/2}m),
\]
which gives the final result
\[
 \|X_n-Y_n\|_2=\bigO(d\log(n)(\log (n/m))^2m^{1/2}(\log m)+dm^{1/2}+dn^{-1/2}m).
\]
Finally, we choose the key term $m$. As we need to choose $m$ such that
\[
\Theta_{m,p}=\bigO(m^{-\chi}(\log m)^{-A})=o(n^{-1/2}),
\]
we therefore choose $m\sim n^{\frac{1}{2\chi}}$ to get
\[
\|X_n-Y_n\|_2=\bigO(dm^{1/2}(\log n)^4)=\bigO(dn^{\frac{1}{4\chi}}(\log n)^4),
\]
which finishes the proof.
\end{itemize}

\subsubsection{Under \cref{asm:dependence}(b) }
Finally, for the other two cases such that $\{x_i\}$ satisfy \cref{asm:dependence}(b), Under a uniform finite $p$-th moment, we can choose $\chi$ large enough in the $M$-dependence approximation in \cref{subsubsec_m_dep}. Then we can directly follow the same proof to get $1/r=1/p$. Under a uniform finite exponential moment, we choose $\Theta_{m,p}=\bigO(\exp(-Cm))=o(n^{-1/2})$, which gives $m\sim \log n$. Then we have $\|X_n-Y_n\|_2=\bigO(d(\log n)^5)$.

\section{Proof of \cref{main_convex}\label{proof_main_convex}}
Since the distance $d_c(X,Y)$ only depends on the law of $X$ and the law of $Y$, we write $d_c(X,Y)$ as $d_c(\mathcal{L}(X), \mathcal{L}(Y))$ in this proof for the sake of simplicity. 
The following lemma provides two approaches for bounding $d_c(\mathcal{L}(X), \mathcal{L}(Y))$.

\begin{lemma}\label{lem_1st_step}
Let $\mathcal{A}$ denote the collection of convex sets in $\mathbb{R}^d$. Suppose $\Cov(X)=I_d$ and $Y$ is a standard $d$-dimensional Gaussian vector. For any $\epsilon>0$, we have
\[
d_c(\mathcal{L}(X),\mathcal{L}(Y))\le 4d^{1/4}\epsilon + \min\left\{\sup_{A\in \mathcal{A}}\left|\E[h_{A,\epsilon}(X)-h_{A,\epsilon}(Y)]\right|, \Pr(|X-Y|>\epsilon)\right\},
\]
where the function $h_{A,\epsilon}(x)$ satisfies that it equals $1$ if $x\in A$, and $0$ if $x\in \mathbb{R}^d\setminus A^{\epsilon}$. Furthermore $|\nabla h_{A,\epsilon(x)}|\le 2/\epsilon$.
\end{lemma}
\begin{proof}
We can separate the two terms in the $\min\left\{\dots\right\}$ on the r.h.s. to get two upper bounds of $d_c(\mathcal{L}(X),\mathcal{L}(Y))$. Then the first upper bound comes directly from \citet[Lemma 4.2]{Fang2015}. For the second upper bound, note that
\begin{equation}
        \begin{split}
            &\Pr(X\in A)-\Pr(Y\in A)\\
            &=\Pr(X\in A, Y\in A^{\epsilon})-\Pr(Y\in A^{\epsilon}) + \Pr(X\in A, Y\notin A^{\epsilon}) + \Pr(Y\in A^{\epsilon}\setminus A)\\
            &\le \Pr(|X-Y|>\epsilon)+\Pr(Y\in A^\epsilon\setminus A),
        \end{split}
    \end{equation}
    and similarly,
     \begin{equation}
        \begin{split}
            &\Pr(Y\in A)-\Pr(X\in A)\\
            &=\Pr(Y\in A^{-\epsilon}, X\notin A) +\Pr(Y\in A^{-\epsilon}, X\in A) + \Pr(Y\in A\setminus A^{-\epsilon}) -\Pr(X\in A)\\
            &\le \Pr(|X-Y|>\epsilon)+ \Pr(Y\in A\setminus A^{-\epsilon}),
        \end{split}
    \end{equation}
    which imply $d_c(X,Y)\le \Pr(|X-Y|>\epsilon) + \sup_{A\in\mathcal{A}}\left\{\Pr(Y\in A^{\epsilon}\setminus A),\Pr(Y\in A\setminus A^{-\epsilon})\right\}$. Finally, since $A$ is a convex set and $Y$ is Gaussian then according to \cite{bentkus2003dependence},
    \[
    \sup_{A\in\mathcal{A}}\left\{\Pr(Y\in A^{\epsilon}\setminus A),\Pr(Y\in A\setminus A^{-\epsilon})\right\}\le 4d^{1/4}\epsilon,
    \]
    where the dimension dependence $d^{1/4}$ is optimal.
\end{proof}

\emph{Proof of \cref{main_convex}:} According to \cref{lem_1st_step}, we will take two approaches to bound $d_c(\mathcal{L}(X), \mathcal{L}(Y))$: in the first approach, we start from bounding $\sup_{A\in \mathcal{A}}\left|\E[h_{A,\epsilon}(X)-h_{A,\epsilon}(Y)]\right|$; in the second approach, we leverage our Wasserstein GA results to bound $\Pr(|X-Y|>\epsilon)$. 

\subsection{First approach}

Let $X_n^*=\frac{1}{\sqrt{n}}X_n$ and $x_i^*=\frac{1}{\sqrt{n}}x_i$.
	We first assume $\E[x_i]=0$ and $\Cov(X_n^*)=I_d$, where $I_d$ denotes the identity matrix with dimension $d\times d$. We will consider the general case $\Cov(X_n^*)\neq I_d$ later. First approach starts from the first result of \cref{lem_1st_step}, for any $\epsilon_1>0$, we have
	\[
	d_c(\mathcal{L}(X_n^*),\mathcal{L}(Y))\le 4 d^{1/4}\epsilon_1+ \sup_{A\in\mathcal{A}} \left|\E[h_{A,\epsilon_1}(X_n^*)-h_{A,\epsilon_1}(Y)] \right|,
	\]
	where the function $h_{A,\epsilon}(x)$ satisfies that it equals $1$ if $x\in A$; $0$ if $x\in \R^d\setminus A^{\epsilon}$. Furthermore, $\left|\triangledown h_{A,\epsilon}(x)\right|\le 2/\epsilon$. Moreover, $Y$ is a standard $d$-dimensinal Gaussian vector with $\Cov(Y)=I_d$. The properties of the function $h_{A,\epsilon}(x)$ imply that if we have $X_n'$ such that $\E\left[\left| X_n^*-X_n'\right|\right]= \bigO(\epsilon_2)$, then
	\[
	\left|\E[h_{A,\epsilon_1}(X_n^*)-h_{A,\epsilon_1}(X_n')] \right|\le \left|\E[\left|\triangledown h_{A,\epsilon_1}(x)\right|\cdot\left| X_n^*-X_n'\right|] \right|=\bigO(\epsilon_2/\epsilon_1).
	\]
	We will choose $\epsilon_1$ and $\epsilon_2$ later to balance the complexity terms. 
 
 In the following, we first apply truncation to $X_n^*$ to get $X_n'$ then $M$-dependence to $X_n'$ to get $\tilde{X}_n$. Noting that, by triangle inequality
	\begin{align*}
	    d_c(\mathcal{L}(X_n^*),\mathcal{L}(Y))\le& 4d^{1/4}\epsilon_1+
	\sup_{A\in\mathcal{A}}\left|\E[h_{A,\epsilon_1}(X_n^*)-h_{A,\epsilon_1}(X_n')] \right|\\
	&+	\sup_{A\in\mathcal{A}}\left|\E[h_{A,\epsilon_1}(X_n')-h_{A,\epsilon_1}(\tilde{X}_n)] \right|\\
	&+	\sup_{A\in\mathcal{A}}\left|\E[h_{A,\epsilon_1}(\tilde{X}_n)-h_{A,\epsilon_1}(\tilde{Y})] \right|\\
	&+	\sup_{A\in\mathcal{A}}\left|\E[h_{A,\epsilon_1}(\tilde{Y})-h_{A,\epsilon_1}(Y)] \right|,
	\end{align*}
	where $\tilde{Y}$ is a Gaussian matches the covariance of $\tilde{X}_n$. We define $\tilde{\Sigma}_n:=\Cov(\tilde{X}_n)=\Cov(\tilde{Y})$.
	
	Later, we show for desired $\epsilon_2$, both the truncation and $M$-dependence satisfy
	\begin{align*}
	&\sup_{A\in\mathcal{A}}\left|\E[h_{A,\epsilon_1}(X_n^*)-h_{A,\epsilon_1}(X_n')] \right|=\bigO(\epsilon_2/\epsilon_1),\\
	&\sup_{A\in\mathcal{A}}\left|\E[h_{A,\epsilon_1}(X_n')-h_{A,\epsilon_1}(\tilde{X}_n)] \right|=\bigO(\epsilon_2/\epsilon_1).
	\end{align*}
		Furthermore, by the Gaussian approximation in \citet[Eq.(4.24)]{Fang2016}, denoting by $\beta$ the truncation level at the truncation step, we have
		\begin{align*}
		&\sup_{A\in\mathcal{A}}\left|\E[h_{A,\epsilon_1}(\tilde{X}_n)-h_{A,\epsilon_1}(\tilde{Y})] \right|\\
		&\le C\|\tilde{\Sigma}_n^{-1/2}\|_2^3n\beta^3 m^2\frac{1}{\epsilon_1}\left[d^{1/4}(\epsilon_1+3m\beta)+d_c(\mathcal{L}(\tilde{X}_n),\mathcal{L}(\tilde{Y}))\right].
		\end{align*}
		By triangle inequality and \citet[Lemma 4.2]{Fang2015},
		\begin{align*}
		&d_c(\mathcal{L}(\tilde{X}_n),\mathcal{L}(\tilde{Y}))\\
		&\le d_c(\mathcal{L}(X_n^*),\mathcal{L}(Y))+ d_c(\mathcal{L}(X_n^*),\mathcal{L}(\tilde{X}_n)) + d_c(\mathcal{L}(Y),\mathcal{L}(\tilde{Y}))\\
		&\le d_c(\mathcal{L}(X_n^*),\mathcal{L}(Y)) + 8d^{1/4}\epsilon_1 + \bigO(\epsilon_2/\epsilon_1)\\
		&\quad+\sup_{A\in\mathcal{A}}\left|\E[h_{A,\epsilon_1}(\tilde{Y})-h_{A,\epsilon_1}(Y)] \right|.
		\end{align*}
		Using \cref{ineq-vanHemmen-Ando}, we have
		\begin{align*}
		\sup_{A\in\mathcal{A}}\left|\E[h_{A,\epsilon_1}(\tilde{Y})-h_{A,\epsilon_1}(Y)] \right|	&\le\frac{2}{\epsilon_1}\frac{1}{\sqrt{\lambda_*}}|\Cov(Y)-\Cov(\tilde{Y})|\\
			&=\frac{2}{\epsilon_1}\frac{1}{\sqrt{\lambda_*}}|\Cov(X_n^*)-\Cov(\tilde{X}_n)|\\
			&\le\frac{2}{\epsilon_1}\frac{1}{\sqrt{\lambda_*}}  \left[2 |(X_n^*-\tilde{X}_n)^TX_n^*| +\bigO (\epsilon_2^2)\right]\\
			&=\bigO(d^{1/2}\epsilon_2/\epsilon_1+ \epsilon_2^2/\epsilon_1).
		\end{align*}
		Overall, replacing $d_c(\mathcal{L}(\tilde{X}_n),\mathcal{L}(\tilde{Y}))$ by  $d_c(\mathcal{L}(X_n^*),\mathcal{L}(Y))$, we have
		\begin{align*}
		&d_c(\mathcal{L}(X_n^*),\mathcal{L}(Y))\le 4d^{1/4}\epsilon_1+ \bigO(\epsilon_2/\epsilon_1)\\
		&\quad+ Cn\beta^3 m^2\frac{1}{\epsilon_1}\left[d^{1/4}(9\epsilon_1+3m\beta)+d_c(\mathcal{L}(X_n^*),\mathcal{L}(Y))+(d^{1/2}+\epsilon_2)\bigO(\epsilon_2/\epsilon_1)\right].
		\end{align*}
		Therefore, we have
		\begin{align*}
			&d_c(\mathcal{L}(X_n^*),\mathcal{L}(Y))\le\\ &\frac{4d^{1/4}\epsilon_1+(1+Cn\beta^3m^2/\epsilon_1\cdot(d^{1/2}+\epsilon_2))\bigO(\epsilon_2/\epsilon_1)+Cn\beta^3m^2d^{1/4}(9+3m\beta/\epsilon_1)}{1-Cn\beta^3m^2/\epsilon_1}.
		\end{align*}
		Note that as long as $\epsilon_1$ satisfies both
		$\epsilon_1\ge 2Cn\beta^3m^2$ 
		where the constant $C$ is the same constant as the above formula and
		$m\beta=o(\epsilon_1)$, we can have
		\[
		d_c(\mathcal{L}(X_n^*),\mathcal{L}(Y))=\bigO(d^{1/4}\epsilon_1+ (d^{1/2}+\epsilon_2)\epsilon_2/\epsilon_1+ d^{1/4}n\beta^3m^2).
		\]
		
		If the $p$-th moment of $x_i$ is finite,
		later we will prove in \cref{step_truncate_m_dep}, choosing $\beta=d^{1/2}n^{\frac{3}{2p}-\frac{1}{2}}$, we have $\epsilon_2=o(d^{1/2}n^{\frac{3}{2p}-1})$. Then we have
		\begin{align*}
		d^{1/4}n\beta^3 m^2
		&\le Cd^{7/4}n^{\frac{9}{2p}-\frac{1}{2}}m^2.
		\end{align*}
		Therefore, we can let
		$\epsilon_1=2Cd^{3/2}n^{\frac{9}{2p}-\frac{1}{2}}m^2$,
		as it satisfies both conditions for $\epsilon_1$, which are
		\begin{align*}
m\beta&=md^{1/2}n^{3/(2p)-1/2}=o(\epsilon_1),\\
		\epsilon_1&\ge 2Cn\beta^3m^2=2Cd^{3/2}n^{9/(2p)-1/2}m^2.
		\end{align*}
		Then, we have $4d^{1/4}\epsilon_1=\bigO(d^{7/4}n^{\frac{9}{2p}-\frac{1}{2}}m^2)$. Therefore, we have
		\[
		\epsilon_2/\epsilon_1=\frac{o(d^{1/2} n^{3/(2p)-1})}{d^{3/2}n^{9/(2p)-1/2}m^2}=o(d^{-1}n^{-3p-1/2}m^{-2}),
		\]
		which implies the term $(d^{1/2}+\epsilon_2)\epsilon_2/\epsilon_1$ can be ignored, since
		\[
		(d^{1/2}+\epsilon_2)\epsilon_2/\epsilon_1=d^{1/2}(1+n^{3/(2p)-1})\cdot o(d^{-1}n^{-3p-1/2}m^{-2}).
		\]
		Overall, we have
		\[
		d_c(\mathcal{L}(X_n^*),\mathcal{L}(Y))=\bigO(d^{7/4}n^{\frac{9}{2p}-\frac{1}{2}}m^2).
		\]
		For general cases, if $\Cov(X_n^*)=\Sigma_n$ then
		\[
		d_c(\mathcal{L}(X_n^*),\mathcal{L}(\Sigma_n^{1/2}Y))=\bigO(\|\Sigma_n^{-1/2}\|_2^3\cdot d^{7/4}n^{\frac{9}{2p}-\frac{1}{2}}m^2)=\bigO( d^{7/4}n^{\frac{9}{2p}-\frac{1}{2}}m^2),
		\]
		where the last equality holds by the assumption that the smallest eigenvalue of $\Sigma_n$ is bounded below by $\lambda_*>0$.
		
		If the exponential moment of $x_i$ is finite,  later we will prove in \cref{step_truncate_m_dep}, by choosing $\beta=\frac{1}{\sqrt{n}}d^{1/2}\left(\frac{3}{2}\log n\right)$, we have
		$\epsilon_2=o\left(\frac{1}{\sqrt{n}} d^{1/2}n^{-1/2}\right)$. Then
\[
d^{1/4}n\beta^3m^2=d^{7/4}n^{-1/2}\left(\frac{3}{2}\log n\right)^3m^2.
\]
Therefore, we can choose
\[
\epsilon_1=\frac{27}{4}Cd^{3/2}n^{-1/2}(\log n)^3m^2,
\]
since both conditions for $\epsilon_1$ hold:
\begin{align*}
m\beta&=\bigO(n^{-1/2}d^{1/2}(\log n)m)=o(\epsilon_1),\\ \epsilon_1&\ge 2Cn\beta^3m^2=2Cd^{3/2}n^{-1/2}(\log n)^3\frac{27}{8}m^2.
\end{align*}
Then, we have
\[
\epsilon_2/\epsilon_1= o\left(\frac{n^{-1/2}d^{1/2}n^{-1/2}}{d^{3/2}n^{-1/2}(\log n)^3m^2}\right)=o(d^{-1}n^{-1/2}(\log n)^{-3}m^{-2}).
\]
Therefore, overall we have
\begin{align*}
&\bigO(d^{7/4}n^{-1/2}(\log n)^3m^2) +(d^{1/2}+\epsilon_2)\cdot o(d^{-1}n^{-1/2}(\log n)^{-3}m^{-2})\\
&=\bigO(d^{7/4}n^{-1/2}(\log n)^3m^2),
\end{align*}
which implies
\[
d_c(\mathcal{L}(X_n^*),\mathcal{L}(Y))=\bigO(d^{7/4}n^{-1/2}(\log n)^3m^2).
\]
For general cases, if $\Cov(X_n^*)=\Sigma_n$ then by the assumption that the smallest eigenvalue of $\Sigma_n$ is bounded below by $\lambda_*>0$,
\begin{align*}
d_c(\mathcal{L}(X_n^*),\mathcal{L}(\Sigma_n^{1/2}Y))&=\bigO(\|\Sigma_n^{-1/2}\|_2^3\cdot d^{7/4}n^{-1/2}(\log n)^3m^2)\\
&=\bigO( d^{7/4}n^{-1/2}(\log n)^3m^2).
\end{align*}
Finally, we choose $m$ using the assumptions on $\Theta_{m,p}$ in \cref{asm:dependence} in \cref{step_truncate_m_dep} and plug in $m$, which completes the proof.
	
\subsubsection{Truncation and $M$-dependence for  \cref{proof_main_convex}\label{step_truncate_m_dep}}
We determine the truncation level $\beta$ and the order of $\epsilon_2$ used in the proof of the first result in \cref{main_convex} in \cref{proof_main_convex}. 
First, we consider the case that each element of $x_i$ uniformly has up to $p$-th finite moments. 
Noting that $\E[x_i]=0$, denoting the truncation operator $T(X):=X\Ind_{\{|X|>d^{1/2}n^{\frac{3}{2p}}\}}$ for any $d$ dimensional random vector $X$, we define $X_n':=\frac{1}{\sqrt{n}} \{\sum_{i=1}^n T(x_i)-\E[T(x_i)]\}$. Then, we have
\begin{align*}
    \sqrt{n}\left|X_n^*-X_n'\right|&=\left|\sum_{i=1}^n \left\{ x_i-\E[x_i] -T(x_i)+ \E[T(x_i)] \right\}\right|\\
	&\le
	\left|\sum_{i=1}^n \{ x_i-T(x_i)\}\right|+\left|\sum_{i=1}^n  \E[x_i-T(x_i)]\right| 
	\end{align*}
	For the first term, since 
	\[
	\Pr(|x_i|>d^{1/2}n^{\frac{3}{2p}})=\E\left[\Ind_{\{|x_i|>d^{1/2}n^{\frac{3}{2p}}\}}\right]
	\le d^{-p/2}n^{-3/2}\E[|x_i|^p]\le C_p n^{-3/2}.
	\] 
	which means with probability $\bigO(n^{-1/2})$ that $x_i\neq T(x_i)$ for any $i=1,\dots,n$. Therefore 
	\[
	\Pr\left( \left|X_n^*-\frac{1}{\sqrt{n}}\sum_{j=1}^n T(x_j)\right|=0 \right)\to 1 
	\]
	which implies that for any order, say $n^{-2}$, we have
	\[
	\Pr\left(\left| \sum_{i=1}^n x_i -\sum_{i=1}^n T(x_i)\right| > n^{-2}\right)\to 0.
	\]
	For the second term,
	noting that
	\[
	x_i-T(x_i)= x_i\Ind_{\{|x_i|> d^{1/2}n^{\frac{3}{2p}}\}}
	\]
	which implies
	\begin{align*}
	\E[x_i-T(x_i)]&= \E\left[x_i\Ind_{\{|x_i|> d^{1/2}n^{\frac{3}{2p}}\}}\right]\\
	&\le d^{1/2}n^{\frac{3}{2p}}\E\left[\frac{|x_i|}{d^{1/2}n^{\frac{3}{2p}}}\Ind_{\{|x_i|> d^{1/2}n^{\frac{3}{2p}}\}}\right]\\
	&\le d^{1/2}n^{\frac{3}{2p}}\E\left[\left(\frac{|x_i|}{d^{1/2}n^{\frac{3}{2p}}}\right)^p\Ind_{\{|x_i|> d^{1/2}n^{\frac{3}{2p}}\}}\right]\\
	&\le d^{1/2}n^{\frac{3}{2p}}\frac{C_pd^{p/2}}{d^{p/2}n^{3/2}}=C_p d^{1/2}n^{\frac{3}{2p}-\frac{3}{2}},
	\end{align*}
	which implies $\left|\sum_{i=1}^n \E[x_i-T(x_i)]\right|=\bigO(d^{1/2}n^{\frac{3}{2p}-\frac{1}{2}})$.
    Therefore, by choosing $\beta=\frac{1}{\sqrt{n}}d^{1/2}n^{\frac{3}{2p}}$ for the truncation level of $x_i^*$, we have $\epsilon_2=\frac{1}{\sqrt{n}}\cdot o(d^{1/2}n^{\frac{3}{2p}-\frac{1}{2}})=o(d^{1/2}n^{\frac{3}{2p}-1})$.
	
	Finally, for $M$-dependence, we define $\tilde{X}_{n}$ similarly to  $X_{n}'$ except $T(x_i)-\E[T(x_i)]$ is replaced by its $M$-dependent version $\E[T(x_i)\,|\, e_i,\dots,e_{i-m}]-\E[T(x_i)]$.
	Denoting $X_{n,j}'$ and $\tilde{X}_{n,j}$ as the $j$-th elements of $X_{n}'$ and $\tilde{X}_{n}$, respectively, 
	we have 
	\[
	\left| X_n'- \tilde{X}_n\right|\le d^{1/2}\max_{1\le j\le d} \left|X_{n,j}'-\tilde{X}_{n,j}\right|\le  d^{1/2}\cdot c_p\max_{1\le j\le d}\sum_{l=m+1}^{\infty}\theta_{l,j,p} =c_p d^{1/2}\Theta_{1+m,p},
	\]
	for some constant $c_p$.
	Therefore, we only need to choose $m$ satisfying
	\[
	\Theta_{m,p}=o(n^{\frac{3}{2p}-1}).
	\]
 Next, we consider that every element of $x_i$ uniformly has a finite exponential moment. We re-define the truncation operator to be $T(X):=X \Ind_{\{|X|>3d^{1/2}(\log n)/2\}}$ for any $d$-dimensional random vector $X$.
	Then
	\begin{align*}
	\Pr(|x_i|>3d^{1/2}(\log n) /2)&=\Pr(\exp(|x_i|)> \exp(d^{1/2})\exp(3(\log n)/2))\\
	&\le \exp(-d^{1/2})\E[\exp(|x_i|)]/\exp(3(\log n)/2)= \bigO(n^{-3/2})=o(1/n).
	\end{align*}
		Therefore, by a union bound $\Pr(\left|\sum_{i=1}^n x_i- \sum_{i=1}^n T(x_i) \right|=0)\to 1$.
		Next, using $x\le \exp(x)-1$, we have
	\begin{align*}
	\E[x_i-T(x_i)]&=\E[(x_i-3d^{1/2}(\log n)/2)\Ind_{\{|x_i|>3d^{1/2}(\log n)/2\}}]\\
	&\le \E[(|x_i|-3d^{1/2}(\log n)/2)\Ind_{\{|x_i|>3d^{1/2}(\log n)/2\}}]\\
	&\le \E[[\exp(|x_i|-3d^{1/2}(\log n)/2)-1]\Ind_{\{|x_i|>3d^{1/2}(\log n)/2\}}]\\
	&\le \exp(-3d^{1/2}(\log n)/2)\E[\exp(|x_i|)\Ind_{\{|x_i|>3d^{1/2}(\log n)/2\}}]\\
	&=\bigO(n^{-3/2}),
	\end{align*}
	so $\sum_{i=1}^n \E(x_i-T(x_i))=\bigO(n^{-1/2})$.
	Now we can put everything together to get
	\[
	\epsilon_2=o\left(\frac{1}{\sqrt{n}} d^{1/2}n^{-1/2}\right),\quad \beta=\frac{1}{\sqrt{n}}d^{1/2}\left(\frac{3}{2}\log n\right).
	\]
	For $M$-dependence, similarly to the case when $x_i$ has finite $p$-th moment uniformly, we only need to choose $m$ such that $\Theta_{1+m,p}=o(n^{-1})$ for some $p<\infty$. 
	
	In conclusion, to finish the proof, according to \cref{asm:dependence}, we choose $m$ such that
 \[
 m^{-\chi}(\log m)^{-A}=o(n^{3/2p}-1), \quad \exp(-Cm)=o(n^{-1}),
 \]
 for the cases of finite $p$-th moment and finite exponential moment, respectively, and plug-in $m$ into $\bigO( d^{7/4}n^{\frac{9}{2p}-\frac{1}{2}}m^2)$ and $\bigO( d^{7/4}n^{-1/2}(\log n)^3m^2)$, respectively.

\subsection{Second approach}
In this approach, we start from the second part in \cref{lem_1st_step}
\[
d_c\left(\mathcal{L}\left(\frac{1}{\sqrt{n}}X_n\right),\mathcal{L}\left(\frac{1}{\sqrt{n}}Y_n\right)\right)\le 4d^{1/4}\epsilon +  \Pr\left(\left|\frac{1}{\sqrt{n}}X_n-\frac{1}{\sqrt{n}}Y_n\right|>\epsilon\right),
\]
where $Y_n$ is a Gaussian vector,
and leverage the proof of \cref{main_borel} in \cref{proof_main_borel} to bound $\Pr\left(\left|\frac{1}{\sqrt{n}}X_n-\frac{1}{\sqrt{n}}Y_n\right|>\epsilon\right)$. Note that from the proof of \cref{lem_1st_step} it is clear that the above inequality holds for any covariance matrices of $X_n$ and $Y_n$. 

Next, instead of directly ``plug-in'' \cref{main_borel} for $\Pr\left(\left|\frac{1}{\sqrt{n}}X_n-\frac{1}{\sqrt{n}}Y_n\right|>\epsilon\right)$, we follow a slightly different argument by separating the proof of \cref{main_borel} into two parts. We assume \cref{asm:dependence} (a) with $A>\sqrt{\chi+1}$ as the other cases directly follow using the same arguments. In the first part, we follow exactly the proof of \cref{main_borel} up to \cref{summary_step4}. Then we have constructed an intermediate Gaussian $\tilde{Y}_n$ (as in \cref{summary_step4}) such that (in a different probability space)
\[
\Pr\left(\left|\frac{1}{\sqrt{n}}X_n-\frac{1}{\sqrt{n}}\tilde{Y}_n\right|>\epsilon\right)\le C\frac{d^2n^{\frac{2}{r}}\log(n)}{\epsilon^2n},
\]
where $r$ is defined in \cref{main_borel}. Note that $\Cov(\tilde{Y}_n)\neq \Cov(X_n)$ so the second part of the proof of \cref{main_borel} is to correct the covariance matrix. We can follow the same arguments in \cref{eq_correct_covariance} to get $\sqrt{\E[|Y_n-\tilde{Y}_n|^2]}=\bigO(dn^{1/r})$ where the Gaussian vector $Y_n$ satisfies $\Cov(Y_n)=\Cov(X_n)$.

Next, using 
$\Pr(|X_n-Y_n|>\sqrt{n}\epsilon)\le \Pr(|X_n-\tilde{Y_n}|+|\tilde{Y_n}-Y_n|>\sqrt{n}\epsilon)\le \Pr(|X_n-\tilde{Y}_n|>\sqrt{n}\epsilon/2) + \Pr(|\tilde{Y}_n-Y_n|>\sqrt{n}\epsilon/2)$
and Chebyshev's inequality, we have
\begin{align*}
    &d_c\left(\mathcal{L}\left(\frac{1}{\sqrt{n}}X_n\right),\mathcal{L}\left(\frac{1}{\sqrt{n}}Y_n\right)\right)\\
    &\le 4d^{1/4}\epsilon + \Pr\left(\left|\frac{1}{\sqrt{n}}X_n-\frac{1}{\sqrt{n}}\tilde{Y}_n\right|>\epsilon/2\right)+\Pr\left(\left|\frac{1}{\sqrt{n}}Y_n-\frac{1}{\sqrt{n}}\tilde{Y}_n\right|>\epsilon/2\right)\\
    &\le 4d^{1/4}\epsilon + 4C\frac{d^2n^{\frac{2}{r}}\log(n)}{\epsilon^2n}  + \frac{\E[|Y_n-\tilde{Y}_n|^2]}{n\epsilon^2/4}\\
    &=  4d^{1/4}\epsilon + \bigO\left(\frac{d^2n^{\frac{2}{r}}\log(n)}{\epsilon^2n}\right)+\bigO\left( \frac{d^2n^{2/r-1}}{\epsilon^2}\right)\\
    &= 4d^{1/4}\epsilon + \bigO\left(\frac{d^2n^{\frac{2}{r}}\log(n)}{\epsilon^2n}\right).
\end{align*}
Now we choose $\epsilon=d^{\frac{7}{12}}n^{\frac{2}{3r}-\frac{1}{3}}(\log n)^{1/3}$ to get
\[
 d_c\left(\mathcal{L}\left(\frac{1}{\sqrt{n}}X_n\right),\mathcal{L}\left(\frac{1}{\sqrt{n}}Y_n\right)\right)=\bigO(d^{\frac{5}{6}}n^{\frac{2}{3r}-\frac{1}{3}}(\log n)^{\frac{1}{3}}),
\]
which finishes the proof. 

Finally, for the other cases, we summarize the differences in the following:
\begin{enumerate}
    \item Under \cref{asm:moment}(a) and \cref{asm:dependence}(b), we replace $r$ by $p$;
    \item Under \cref{asm:moment}(b) and \cref{asm:dependence}(a), we have
\[
d_c\left(\mathcal{L}\left(\frac{1}{\sqrt{n}}X_n\right),\mathcal{L}\left(\frac{1}{\sqrt{n}}Y_n\right)\right)\le 4d^{1/4}\epsilon +\bigO\left(\frac{d^2n^{\frac{1}{2\chi}}(\log n)^8}{\epsilon^2n}\right)=\bigO(d^{5/6}n^{\frac{1}{6\chi}-\frac{1}{3}}(\log n)^{8/3}).
\]
where we chose $\epsilon=d^{7/12}n^{\frac{1}{6\chi}-1/3}(\log n)^{8/3}$;
\item If both \cref{asm:moment}(b) and \cref{asm:dependence}(b) hold, we have
\[
d_c\left(\mathcal{L}\left(\frac{1}{\sqrt{n}}X_n\right),\mathcal{L}\left(\frac{1}{\sqrt{n}}Y_n\right)\right)\le 4d^{1/4}\epsilon +\bigO\left(\frac{d^2(\log n)^{10}}{\epsilon^2n}\right)=\bigO(d^{5/6}n^{-\frac{1}{3}}(\log n)^{10/3}),
\]
where we chose $\epsilon=d^{7/12}n^{-1/3}(\log n)^{10/3}$.
\end{enumerate}

\section{Proof of Results in Section \ref{section:bootstrap} } \label{proof:boot_convex}
%The following lemma presents a general result on $\Delta_n^{(1)}$ and $\Delta_n^{(2)}$, which quantifies the magnitude of the deviation of the covariance matrix in different norms.
\subsection{Proof of \cref{lemma_bootstrap_delta}}\label{proof_lemma_bootstrap_delta}
%\begin{proof}
To evaluate the entry-wise difference of the covariance matrix, we quantify the $(j,k)$-th entry as 
\begin{equation}\label{eq:entrywise_delta}
\begin{aligned}
|\Delta_{j,k}| & := \left| \Cov \left(\sum_{i =1 }^n x_{ij}/\sqrt{n},\sum_{i =1 }^n x_{ik}/\sqrt{n}  \right) - \Cov_{B \mid X}\left(\frac{\tau_{n,j}}{\sqrt{ n - L + 1}},\frac{\tau_{n,k}}{\sqrt{ n - L + 1}}\right )\right|\\ 
& \leq \left|\Cov \left(\sum_{i =1 }^n x_{ij}/\sqrt{n},\sum_{i =1 }^n x_{ik}/\sqrt{n}  \right) - \E_X \Cov_{B|X}\left(\frac{\tau_{n,j}}{\sqrt{ n - L + 1}},\frac{\tau_{n,k}}{\sqrt{ n - L + 1}}\right) \right| \\ 
    &+  \left|\Cov_{B|X}(\frac{\tau_{n,j}}{\sqrt{ n - L + 1}},\frac{\tau_{n,k}}{\sqrt{ n - L + 1}}) - \E_X \Cov_{B|X}\left(\frac{\tau_{n,j}}{\sqrt{ n - L + 1}},\frac{\tau_{n,k}}{\sqrt{ n - L + 1}}\right) \right|\\
    & := a_{j,k} + b_{j,k}.
\end{aligned}
\end{equation}
% Recall $\Delta_n^{(2)}$ defined in \cref{lemma_bootstrap_delta}, we have
% \begin{align*}
%     &\Delta_n^{(2)} =    \max_{1 \leq j,k \leq d } \left| \Cov \left(\sum_{i =1 }^n x_{ij}/\sqrt{n},\sum_{i =1 }^n x_{ik}/\sqrt{n}  \right) - \Cov_{B \mid X}\left(\frac{\tau_{n,j}}{\sqrt{ n - L + 1}},\frac{\tau_{n,k}}{\sqrt{ n - L + 1}}\right )\right|\\
%     & \leq \max_{1 \leq j,k \leq d } \left|\Cov \left(\sum_{i =1 }^n x_{ij}/\sqrt{n},\sum_{i =1 }^n x_{ik}/\sqrt{n}  \right) - \E_X \Cov_{B|X}\left(\frac{\tau_{n,j}}{\sqrt{ n - L + 1}},\frac{\tau_{n,k}}{\sqrt{ n - L + 1}}\right) \right| \\ 
%     &+  \max_{1 \leq j,k \leq d } \left|\Cov_{B|X}(\frac{\tau_{n,j}}{\sqrt{ n - L + 1}},\frac{\tau_{n,k}}{\sqrt{ n - L + 1}}) - \E_X \Cov_{B|X}\left(\frac{\tau_{n,j}}{\sqrt{ n - L + 1}},\frac{\tau_{n,k}}{\sqrt{ n - L + 1}}\right) \right|\\
%     & := I_1 + I_2.
% \end{align*}

By \citet[Lemma 1]{zhou2013heteroscedasticity}, under the constraint \cref{asm:moment}(a) and \cref{asm:dependence}, we have that for $1 \le j,k \le d$, %\textcolor{OliveGreen}{(In that paper we require the dependence measure $\delta_4(k) = \bigO(\chi^{-k})$) and assume finite 4-th order moments, and we bound its $\mathcal{L}_2$ norm, now we are trying to bound the $\mathcal{L}_r$ norm, $r = p/2$)} 
\begin{equation*}
    \begin{aligned}
         \left\| \Cov_{B|X}(\tau_{n,j},\tau_{n,k} ) - \E_X \Cov_{B|X}(\tau_{n,j},\tau_{n,k} ) \right\|_r \lesssim \sqrt{L(n-L+1)} \lesssim \sqrt{nL}.
         %& \lesssim \sqrt{n-L+1}\cdot\sqrt{L} \cdot \sum_{l=0}^\infty \sum_{m = l - (L-1)}^l (\theta_{m,j,2r} + \theta_{m,k,2r})/L \\
    \end{aligned}
\end{equation*}
This implies that for $1\le j,k\le d$,
\[||b_{j,k}||_r \lesssim \sqrt{L/n} \quad \text{and} \quad b_{j,k} = \bigO_\pr(\sqrt{L/n}).\]
%\[\left|\Cov_{B|X}\left(\frac{\tau_{n,j}}{\sqrt{n - L + 1}},\frac{\tau_{n,k}}{\sqrt{n - L + 1}}\right) - \E_X \Cov_{B|X}\left(\frac{\tau_{n,j}}{\sqrt{n - L + 1}},\frac{\tau_{n,k}}{\sqrt{n - L + 1}}\right)\right|= \bigO_{\Pr}(\sqrt{L/n})\]
%and
%\[ I_2 = \bigO_{\Pr}(d^{2/r}\sqrt{L/n}). \]
%This is obtained by considering $|a_k| = O_\Pr(\sqrt{L/n})$, for $k = 1, 2, \cdots, d^2$? If we assume $||a_k||_r = O(\sqrt{L/n})$, then
%\[\left\|\max_{1 \le k \le d^2}|a_k|\right\|_r = \left(\E\left(\max_{1 \le k \le d^2}|a_k|\right)^r\right)^{1/r} \le \left(\E\left(\sum_{k = 1}^{d^2}|a_k|^r\right)\right)^{1/r} \le d^{2/r}O(\sqrt{L/n})? \]
Observe that
% \begin{align*}
%     a_{j,k} = & \max_{1 \leq j,k \leq d } \left|\Cov \left(\sum_{i =1 }^n x_{ij}/\sqrt{n},\sum_{i =1 }^n x_{ik}/\sqrt{n}  \right) - \E_X \Cov_{B|X}\left(\frac{\tau_{n,j}}{\sqrt{n - L + 1}},\frac{\tau_{n,k}}{\sqrt{n - L + 1}}\right) \right|  \\
%     & = \max_{1 \leq j,k \leq d } \left|\sum_{1 \leq i,l\leq n} \frac{\E (x_{ij} x_{lk})}{n} - \sum_{i = 1}^{n-L+1} \frac{\E \left([\psi_{i,L}]_j[\psi_{i,L}]_k\right)}{n-L+1}  \right|\\
%     & \lesssim \max_{1 \leq j,k \leq d } \left|\sum_{ |i - l| < L} \left(\frac{L - |i-l|}{L(n-L+1)} - \frac{1}{n}\right) \E x_{ij} x_{lk}  \right|  + \max_{1 \leq j,k \leq d } \left|\sum_{ |i - l| \ge 
%     L} \frac{\E x_{ij} x_{lk}}{n}  \right|\\
%    & := I_{11} + I_{12}.
% \end{align*}
\begin{align*}
    a_{j,k} & =  \left|\Cov \left(\sum_{i =1 }^n x_{ij}/\sqrt{n},\sum_{i =1 }^n x_{ik}/\sqrt{n}  \right) - \E_X \Cov_{B|X}\left(\frac{\tau_{n,j}}{\sqrt{n - L + 1}},\frac{\tau_{n,k}}{\sqrt{n - L + 1}}\right) \right|  \\
    & =  \left|\sum_{1 \leq i,l\leq n} \frac{\E (x_{ij} x_{lk})}{n} - \sum_{i = 1}^{n-L+1} \frac{\E \left([\psi_{i,L}]_j[\psi_{i,L}]_k\right)}{n-L+1}  \right|\\
    & \lesssim  \left|\sum_{ |i - l| < L} \left(\frac{L - |i-l|}{L(n-L+1)} - \frac{1}{n}\right) \E x_{ij} x_{lk}  \right|  +  \left|\sum_{ |i - l| \ge 
    L} \frac{\E x_{ij} x_{lk}}{n}  \right|\\
   & := a_{j,k}^{(1)} + a_{j,k}^{(2)}.
\end{align*}
{According to \citet[Lemma 6]{zhou2014inference}, 
\[|\E(x_{ij}x_{lk})| = |\Cov(x_{ij}, x_{lk})| \le \sum_{s = -\infty}^\infty \theta_{i-s, j, 2}\cdot \theta_{l-s, k,2} = \sum_{s = -\infty}^{\min(i,l)} \theta_{i-s, j, 2}\cdot \theta_{l-s, k,2}.\] }
{Furthermore, by \cref{asm:dependence}(a), \[\sum_{s = -\infty}^{\min(i,l)} \theta_{i-s, j, 2}\cdot \theta_{l-s, k,2} \lesssim \max(\theta_{|i-l|,j,2}, \theta_{|i-l|,k,2})\cdot \Theta_{0,2}.\] }
Then for $ 1 \le j \le d$, %and approximate $\E x_{ij} x_{lk}$ with $\delta_2(|i-j|)$:
\begin{equation}\label{eq:split_cov}
    \begin{split}
         a_{j,k}^{(1)} %& \lesssim \max_{1 \leq j,k \leq d } \sum_{|i - l|< L}  \left(\left|\frac{1}{\um - L +1} - \frac{1}{\um}\right| + \left|\frac{i-l}{\um}\right|\right)  {\max(\theta_{|i-l|,j,2}, \theta_{|i-l|,k,2})}  \\
    & \lesssim \sum_{s = 0}^{ L-1 }\sum_{ |i - l| = s} \left(\frac{L}{n(n - L + 1)} + \frac{|i-l|}{L(n - L + 1)}\right) {\max(\theta_{|i-l|,j,2}, \theta_{|i-l|,k,2})}   \\
    & \lesssim  \sum_{s = 0}^{ L-1 } \frac{L(n - L + 1 -s)}{n(n - L + 1)}{\max(\theta_{s,j,2}, \theta_{s,k,2})}  +  \frac{1}{L}\sum_{s = 0}^{ L-1 }s\cdot {\max(\theta_{s,j,2}, \theta_{s,k,2})} \\
    & = \bigO \left( \frac{L}{n} \right) + \bigO \left( \frac{1}{L} \right). %\text{\textcolor{OliveGreen}{\; since $\sum_{s = 0}^{ L-1 }s\cdot {\max(\theta_{s,j,2}, \theta_{s,k,2})} = O(1)$ holds for $\chi > 1$}}
    \end{split}
\end{equation}
Similarly, under \cref{asm:dependence}(a), 
\begin{align*}
    a_{j,k}^{(2)} & \lesssim  \sum_{ |i - l| \ge 
    L} |\E x_{ij} x_{lk}|/n \\
    & \lesssim \max_{1 \leq j,k \leq d }\sum_{s = L}^{n-L} (n-L+1-s) \max(\theta_{|i-l|,j,2}, \theta_{|i-l|,k,2}) /n \\
    & \lesssim \Theta_{L,2} \lesssim \bigO\left(1/L \right). %\quad \text{\textcolor{OliveGreen}{this will be dominated by $\bigO \left( \frac{1}{L} \right)$ if $\chi > 1$} }
\end{align*}
Combine previous results we get $a_{j,k} = \bigO \left( \frac{L}{n} + \frac{1}{L}\right)$.\\
Recall that \[\Delta_n^{(1)} =  \left|\Cov\left[ X_n /\sqrt{n} \right]- \Cov\left[ \tau_{n,L}/\sqrt{n-L+1} \mid X \right] \right|_F = \sqrt{\sum_{j}\sum_{k}\Delta_{j,k}^2},\]
we have
\begin{align*}
    \|\Delta_n^{(1)}\|_2 & =  \left(\E\sum_{j}\sum_{k} \Delta_{j,k}^2\right)^{1/2}  \lesssim d\left(\sqrt{\frac{L}{n}} + \frac{L}{n} + \frac{1}{L}\right) \lesssim d\left(\sqrt{\frac{L}{n}} + \frac{1}{L}\right),
\end{align*}
which implies that $\Delta_n^{(1)} = \bigO_{\pr}\left(d\left(\sqrt{\frac{L}{n}} + \frac{1}{L}\right)\right)$.
As for $\Delta_n^{(2)}$, note that
\begin{align*}
    \Delta_n^{(2)} & = \max_{1\le j, k \le d}|\Delta_{j,k}| \le \max_{1 \le j,k \le d} a_{j,k} + \max_{1 \le j,k \le d}b_{j,k},
\end{align*}
we have 
\[
\|\max b_{j,k}\|_r = \left(\E(\max b_{j,k})^r\right)^{1/r} \le \left(\sum_{j}\sum_k \E|b_{j,k}|^r\right)^{1/r} =\bigO(d^{2/r}\sqrt{L/n}).
\]
Therefore, combining all results, we get
$\Delta_n^{(2)}  = \bigO_{\pr}\left(\sqrt{\frac{L}{n}}d^{4/p} + \frac{L}{n} + \frac{1}{L}\right)$.
%\end{proof}

\subsection{Proof of \cref{key_lemma3}}\label{proof_key_lemma3}
Note that for (zero-mean) Gaussian random vectors the $2$-Wasserstein distance is known
\[
\mathcal{W}_2(X,Y)=\tr\left(\Sigma+\hat{\Sigma}-2(\Sigma^{1/2}\hat{\Sigma}\Sigma^{1/2})^{1/2}\right).
\]
However, it is not easy to directly control the above form. Instead, we consider the following construction of coupling:
\[
X=(X_1,\dots,X_d)^T=\Sigma^{1/2}Z,\quad Y=(Y_1,\dots,Y_d)^T=\hat{\Sigma}^{1/2}Z,\]
where $Z=(Z_1,\dots,Z_d)^T\sim \mathcal{N}(0,I_d)$.
  Then
\begin{align*}
		\mathcal{W}_2(X,Y)\le \sqrt{\E \sum_{i=1}^d (X_i-Y_i)^2}
		&=\sqrt{\tr(\Sigma^{1/2}-\hat{\Sigma}^{1/2})^2}=|\Sigma^{1/2}-\hat{\Sigma}^{1/2}|_F,
\end{align*}
where the first inequality is tight when $\Sigma$ and $\hat{\Sigma}$ are commutative.

Using \cref{ineq-vanHemmen-Ando}, we have
$	|\Sigma^{1/2}-\hat{\Sigma}^{1/2}|_F\le  \lambda_*^{-1/2} |\Sigma-\hat{\Sigma}|_F$,
which finishes the proof.
An alternative proof is to use \citet[Lemma 1]{eldan2020clt} and cyclic invariance of the trace:
\begin{align*}
    \E[|X-Y|^2]&=\tr\left[(\Sigma^{1/2}-\hat{\Sigma}^{1/2})^2\right]
    \le \tr\left[(\Sigma-\hat{\Sigma})^2\Sigma^{-1}\right]
    =\tr\left[\Sigma^{-1/2}(\Sigma-\hat{\Sigma})^2\Sigma^{-1/2}\right]\\
    &=\tr\left[(\hat{\Sigma}-\Sigma)^2\Sigma^{-1}\right]=\left|(\hat{\Sigma}-\Sigma)\Sigma^{-1/2}\right|_F^2\le \left|\hat{\Sigma}-\Sigma\right|_F^2\frac{1}{\lambda_*}.
\end{align*}

\subsection{Proof of \cref{key_lemma2}}\label{proof_key_lemma2}

The following lemma (\cref{stein_identity}) is the famous Stein's identity which we will use in the proof.
\begin{lemma}{(Stein's identity)}\label{stein_identity}
	Let $W=(W_1,\dots,W_d)^T$ be a centered Gaussian random vector in $\R^d$. Let $f: \R^d\to \R$ be a $C^1$-function such that $\E[|\partial_j f(W)|]<\infty$ for all $1\leq   j\leq   d$. Then for every $1\leq   j\leq   d$,
	\[
	\E[W_jf(W)]=\sum_{k=1}^d \E[W_jW_k]\E[\partial_k f(W)].
	\]
\end{lemma}

\emph{Proof of \cref{key_lemma2}:} For a given Borel set $A\in \mathcal{B}(\mathcal{X})$, we first use similar technique as the Yurinskii coupling \citet[Ch.10, Lemma 18]{Pollard2001} to construct a function $f_A: \mathbb{R}^d\to \mathbb{R}$ by a convolution smoothing using two parameters $\sigma$ and $\delta$: 
\[
f_A(x):=\E[g_A(x+\sigma Z)]=\int g_A(w)\phi_{\sigma}(w-x)d  w
\]
where $Z\sim \mathcal{N}(0,I_d)$, $\phi_{\sigma}$ is the density of $\sigma Z$, and 
\[
g_A(x):=\left(1-\frac{\textrm{dist}(x,A^{\delta})}{\delta}\right)^+.\]
Note that the function $g_A$ is $1/\delta$-Lipschitz by the definition.
	
Furthermore, according to \citet[Ch.10, Lemma 18]{Pollard2001}, this constructed function $f_A$ satisfies
\[\left| f_A(x+y)-f_A(x)-y^Tf_A'(x)-\frac{1}{2}y^Tf_A''(x)y\right|\le 15(\sigma^2\delta)^{-1}|y|^3,\quad \forall x,y\in \mathbb{R}^d.\]
Furthermore, if $\delta>\sigma d^{1/2}$, the function $f_A$ satisfies
\begin{equation}\label{eq_construct_f_A}
	    (1-\epsilon)\Ind_{\{x\in A\}}\leq   f_A(x)\leq   \epsilon+(1-\epsilon)\Ind_{\{x\in A^{3\delta}\}}
\end{equation}
where
$\epsilon:=\left({\frac{\delta^2}{d \sigma^2}/\exp(\frac{\delta^2}{d \sigma^2}-1)}\right)^{d/2}$. 
	
Using \cref{eq_construct_f_A}, for any Borel set $A\in \mathcal{B}(\mathcal{X})$, we have
\begin{align*}
\pr(X\in A)&\leq   (1-\epsilon)^{-1}\E[f_A(X)]\\
	&\leq   (1-\epsilon)^{-1}\left(\E[f_A(Y)]+\sup_{A\in \mathcal{B}(\mathcal{X})}\left|\E[f_A(X)]-\E[f_A(Y)]\right|\right) \\
	&\leq (1-\epsilon)^{-1}\left( \epsilon + (1-\epsilon)  \pr(Y\in A^{3\delta}) +\sup_{A\in \mathcal{B}(\mathcal{X})}\left|\E[f_A(X)]-\E[f_A(Y)]\right|\right)\\
	&=  \pr(Y\in A^{3\delta})+\frac{\epsilon+\sup_{A\in \mathcal{B}(\mathcal{X})}\left|\E[f_A(X)]-\E[f_A(Y)]\right|}{1-\epsilon},
\end{align*}
where the supremum is taken over all Borel sets.
	
Therefore, by Strassen's Theorem (\cref{thm_Strassen}), this directly implies that one can construct $X^c$ and $Y^c$ which have the same distributions as $X$ and $Y$, respectively, in a richer probability space, such that 
\begin{equation}\label{eq_final_temp}
	\pr(|X^c-Y^c|>3\delta)\leq   \frac{\epsilon+\sup_{A\in \mathcal{B}(\mathcal{X})}|\E f_A(X)- \E f_A(Y)|}{1-\epsilon},
\end{equation}
where the term $\sup_{A\in \mathcal{B}(\mathcal{X})}|\E f_A(X)- \E f_A(Y)|$ only depends on two parameters $\delta$ and $\sigma$ satisfying $\delta>\sigma d^{1/2}$ and the term  $\epsilon=\left({\frac{\delta^2}{d \sigma^2}/\exp(\frac{\delta^2}{d \sigma^2}-1)}\right)^{d/2}$. 
Then it suffices to bound the term  $\sup_{A\in \mathcal{B}(\mathcal{X})}|\E f_A(X)- \E f_A(Y)|$. Note that a trivial upper bound can be obtained by the Lindeberg's method as \citet[Ch.7.2, Eq.(18)]{Pollard2001}. However, we need a tight bound in terms of $\Delta$ and the two parameters $\delta$ and $\sigma$.
	
%Then without loss of generality, we write $X=\sum_{i=1}^n x_i$ where $x_i\sim \mathcal{N}(0, \Cov(X)/n)$ are independent random variables (we actually only require bounded third moment of $x_i$). From the \RED{multivaraite form of Lindeberg's method (JY: TBA more detail)}, under the third moment bound for a difference in expectations, we have
	%\[
	%	|\E f(X) - \E f(Y)|\leq   C \left(\E|x_1|^3+\dots+ \E|x_n|^3 \right)=C\beta,
	%\]
	%for some constant $C$ and $\beta:=\E|x_1|^3+\dots+ \E|x_n|^3$. 
	%Meanwhile, the following is satisfied:
	%\begin{align*}
	%    \pr(X\in A)&\leq   (1-\epsilon)^{-1}\E[f(X)]\\
	%&\leq   (1-\epsilon)^{-1}\left(\E[f(Y)]+C\beta \right) \\
	%&\leq   \pr(Y\in A^{3\delta})+\epsilon',
	%\end{align*}
	%where 
	%$\epsilon':=\frac{\epsilon+C\beta}{1-\epsilon}$.
	%Note that by choosing $\alpha>0$ one can adjust $\epsilon$. Using \cref{thm_Strassen}, this directly implies that for the particular constructed function $f$ using parameters $\sigma$ and $\delta$, then one can construct $X^c$ and $Y^c$ such that \RED{(JY: RHS shouldn't be sup over $A$?)}
	%		\[
	%		\pr(|X^c-Y^c|>3\delta)\leq   \frac{\epsilon+\sup_{A}|\E f(X)- \E f(Y)|}{1-\epsilon},
	%		\]
	%		where the supremum is over all Borel sets $A$ and $\epsilon \in (0,1)$ and $\delta>\sigma d^{1/2}$.

Next, we consider the fact that both $X$ and $Y$ are Gaussian vectors and use techniques in \citet{Chernozhukov2015} to bound $\sup_{A\in \mathcal{B}(\mathcal{X})}|\E f_A(X)- \E f_A(Y)|$. Consider the following Slepian interpolation between $X$ and $Y$:
\[
Z(t):=\sqrt{t}X+\sqrt{1-t}Y, \quad t\in [0,1].
\]
Define $\Psi_A(t):=\E[f_A(Z(t))]$, then our focus is
\[
\sup_{A\in \mathcal{B}(\mathcal{X})}|\E f_A(X)-\E f_A(Y)|=\sup_{A\in \mathcal{B}(\mathcal{X})}|\Psi_A(1)-\Psi_A(0)|=\sup_{A\in \mathcal{B}(\mathcal{X})}\left| \int_0^1 \Psi_A'(t)\dee   t \right|.
\]
Note that
\[
\frac{d  Z(t)}{d  t}=\frac{1}{2}(t^{-1/2}X-(1-t)^{-1/2}Y)
\]
so let $X_j$ and $Y_j$ be the $j$-th elements of $X$ and $Y$, respectively, we have
\[
\Psi_A'(t)=\frac{1}{2}\sum_{j=1}^d \E\left[\partial_j f_A(Z(t))\cdot(t^{-1/2}X_j-(1-t)^{-1/2}Y_j)\right].
\]
Next, for fixed $j$, we construct $d+1$ vector and function $\tilde{f}_A$ as follows:
\begin{align*}
W:&=(Z(t)^T, (t^{-1/2}X_j-(1-t)^{-1/2}Y_j))^T\in \R^{d+1},\\
\tilde{f}_A(W):&=\partial_j f_A(Z(t))\in \R.
\end{align*}
Next, by Stein's identity (\cref{stein_identity})
\begin{equation}\label{eq_stein_temp}
\E[W_{d+1}\tilde{f}_A(W)]=\sum_{k=1}^{d+1}\E[W_{d+1}W_k]\E[\partial_k \tilde{f}_A(W)].
\end{equation}
Note that the left hand side of \cref{eq_stein_temp} is
\[
\E[W_{d+1}\tilde{f}_A(W)]=\E\left[\partial_j f_A(Z(t))\cdot(t^{-1/2}X_j-(1-t)^{-1/2}Y_j)\right].
\]
Furthermore, since $\E[\partial_{d+1}\tilde{f}_A(W)]=0$, we have
\begin{align*}
&\sum_{k=1}^{d+1}\E[W_{d+1}W_k]\E[\partial_k \tilde{f}_A(W)]=\sum_{k=1}^{d}\E[W_{d+1}W_k]\E[\partial_k \tilde{f}_A(W)]\\
&=\sum_{k=1}^d\E\left[\left(\frac{X_j}{\sqrt{t}}-\frac{Y_j}{\sqrt{1-t}}\right)\left( \sqrt{t}X_k+\sqrt{1-t}Y_k \right) \right]\E [\partial_k\partial_j f_A(Z(t))]\\
&=\sum_{k=1}^d\left[\Cov(X)_{jk}-\Cov(Y)_{jk}\right]\E [\partial_k\partial_j f_A(Z(t))]    
\end{align*}
where $\Cov(\cdot)_{jk}$ denotes the $(j,k)$-th element of the covariance matrix.

Overall, we have
\begin{align*}
\left|\int_0^1 \Psi_A'(t)d  t\right|&\leq   \frac{1}{2} \sum_{j,k=1}^d \left|\Cov(X)_{jk}-\Cov(Y)_{jk}\right|\cdot\left|\int_0^1 \E [\partial_k\partial_j f_A(Z(t))]\dee  t\right|\\
&\le\frac{1}{2} \max_{j,k} \left|\Cov(X)_{jk}-\Cov(Y)_{jk}\right|\cdot \int_0^1\sum_{j,k=1}^d \left|\E [\partial_k\partial_j f_A(Z(t))]\right| \dee  t.
\end{align*}
Next, we show
\[
\sup_{A\in \mathcal{B}(\mathcal{X})}\int_0^1\sum_{j,k=1}^d \left|\E [\partial_k\partial_j f_A(Z(t))]\right| d  t\leq   3d^2\sigma^{-1}\delta^{-1}.
\]
By the construction of $f_A$, we have
\[
\frac{\partial^2 f_A(x)}{\partial x_k \partial x_j}=\int g_A(w)\left(\frac{\partial^2 \phi_{\sigma}(w-x)}{\partial x_k \partial x_j}\right) d w.
\]
Using the fact
\[
\left[\frac{\partial^2 \phi_{\sigma}(x)}{\partial x^2}\right]_{d\times d}=\left(\frac{xx^T-\sigma^2I_d}{\sigma^4}\right)_{d\times d}\phi_{\sigma}(x)
\]
and $\phi_{\sigma}$ is the density of $\sigma Z$, we have
\begin{align*}
\frac{\partial^2 f_A(x)}{\partial x_k \partial x_j}&=\int g_A(w)\left(\frac{\partial^2 \phi_{\sigma}(w-x)}{\partial x_k \partial x_j}\right) \dee  w\\
&=\int g_A(w) \left(\frac{(w-x)(w-x)^T-\sigma^2I_d}{\sigma^4}\right)_{kj}\phi_{\sigma}(w-x) \dee  w\\
&=\E\left[\left(\frac{(\sigma Z)(\sigma Z)^T-\sigma^2I_d}{\sigma^4}\right)_{kj}g_A(x+\sigma Z)\right]\\
&=\E\left[\left(\frac{ZZ^T-I_d}{\sigma^2}\right)_{kj}g_A(x+\sigma Z)\right]\\
&=\sigma^{-2}\E \left[(Z_kZ_j-\Ind_{\{k=j\}})g_A(x+\sigma Z)\right]
\end{align*}
Therefore, we have
\[
\sum_{j,k=1}^d \left|\partial_k\partial_j f_A(x)\right|=\sigma^{-2}\sum_{j,k=1}^d\left|\E \left[(Z_kZ_j-\Ind_{\{k=j\}})g_A(x+\sigma Z)\right]\right|.
\]
Note that, if we construct $Z^*$ as $(Z_1,\dots,Z_k^*,\dots,Z_j^*,\dots,Z_d)$ where $Z_k^*$ and $Z_j^*$ are identically distributed and independent with $Z_k$ and $Z_j$, then
\[
\E \left[(Z_kZ_j-\Ind_{\{k=j\}})g_A(x+\sigma Z^*)\right]
=	\E \left[Z_kZ_j-\Ind_{\{k=j\}}\right]\E[g_A(x+\sigma Z^*)]=0.
\]
Using the fact that $g_A(x)$ is $1/\delta$-Lipschitz, we have
\begin{align*}
&\E \left[(Z_kZ_j-\Ind_{\{k=j\}})g_A(x+\sigma Z)\right]\\
&=	\E \left[(Z_kZ_j-\Ind_{\{k=j\}})\left(g_A(x+\sigma Z)-g_A(x+\sigma Z^*)\right)\right]\\
&\leq   \E \left[(Z_kZ_j-\Ind_{\{k=j\}})\left(\frac{\sigma}{\delta}\sqrt{(Z_k-Z_k^*)^2+(Z_j-Z_j^*)^2}\right)\right]\\
&=\frac{\sigma}{\delta}\E \left[(Z_kZ_j-\Ind_{\{k=j\}})\sqrt{(Z_k-Z_k^*)^2+(Z_j-Z_j^*)^2}\right] \leq   3\frac{\sigma}{\delta},
\end{align*}
where the right hand side holds for any $A\in \mathcal{B}(\mathcal{X})$. Combing all the above arguments together, we have shown
\begin{align*}
&\sup_{A\in \mathcal{B}(\mathcal{X})}\int_0^1\sum_{j,k=1}^d \left|\E [\partial_k\partial_j f_A(Z(t))]\right| \dee  t\\
&\leq   \sup_{A\in \mathcal{B}(\mathcal{X})}\int_0^1\sum_{j,k=1}^d \E_{x\sim Z(t)} \left[\left|\partial_k\partial_j f_A(x)\right|\right] \dee  t\\
&=\sup_{A\in \mathcal{B}(\mathcal{X})}\int_0^1\E_{x\sim Z(t)}\left[\sum_{j,k=1}^d  \left|\partial_k\partial_j f_A(x)\right|\right] \dee  t\\
&\leq   3d^2\sigma^{-1}\delta^{-1}.
\end{align*}
Therefore, if $|\Cov(X)-\Cov(Y)|_{\max}=\Delta$
		then 
		\[
		\sup_{A\in \mathcal{B}(\mathcal{X})} |\E f_A(X)- \E f_A(Y)|\leq   \frac{3}{2}d^2\frac{1}{\sigma\delta}\Delta.
		\]
Substituting to \cref{eq_final_temp} yields
\[
\pr(|X^c-Y^c|>3\delta)\leq   \frac{\epsilon+\frac{3}{2}d^2\Delta \frac{1}{\sigma\delta}}{1-\epsilon},
\]
where $\delta>\sigma d^{1/2}$ and  $\epsilon=\left({\frac{\delta^2}{d \sigma^2}/\exp(\frac{\delta^2}{d \sigma^2}-1)}\right)^{d/2}$. 

Finally, if we choose $\delta=2\sigma d^{1/2}>\sigma d^{1/2}$, it can be easily verified that $\epsilon=\left(\frac{4}{\exp(3)}\right)^{d/2}\to 0$ faster than $\frac{3}{2}d^2\Delta\frac{1}{2\sigma^2}d^{-1/2}$, then there exists a constant $C$ such that
\[
\pr(|X^c-Y^c|>6d^{1/2}\sigma)\leq   Cd^{3/2}\sigma^{-2}\Delta.
\]
Finally, by changing notations $\tilde{\epsilon}=6d^{1/2}\sigma$, we have
	\[
	\pr(|X^c-Y^c|>\tilde{\epsilon})\leq   Cd^{5/2}\tilde{\epsilon}^{-2}\Delta,
	\]
	for some constant $C$. This completes the proof of \cref{key_lemma2}.

\subsection{Proof of \cref{boot:borel}}
\label{proof:boot_borel}
%Let $Y_n$ be a Gaussian random vector with the same mean and covariance structure as $X_n$. Under the assumptions of \cref{boot:borel}, the time series $\{x_i\}_{i = 1}^n$ satisfies $\Delta_n^{(2)} \leq a_n$ and $\tau_{n,L}$ is a Gaussian random vector conditional on $\{x_i\}_{i=1}^n$. Denote the distribution of $Y_n/\sqrt{n}$ by $Q'$, then using triangle inequality, we have
We shall only prove part (a) of this theorem since part (b) follows by essentially the same arguments by considering $\Delta_n^{(2)}$ in Lemma \ref{key_lemma2} instead of $\Delta_n^{(1)}$ in Lemma \ref{key_lemma3} and the L\'evy-Prohorov metric instead of the ${\cal W}_2$ distance. Under the assumptions of \cref{boot:borel}, $\tau_{n,L}$ is a Gaussian random vector conditional on $\{x_i\}_{i=1}^n$. Let $Y_n$ be a Gaussian random vector with the same mean and covariance structure as $X_n$, by Lemma \ref{key_lemma3}, for any $\omega$ in the sample space of $\{x_i\}$, one can construct $\boldsymbol{w}_n^{(1)} \stackrel{D}{=} Y_n/\sqrt{n}$ and $\boldsymbol{v}_n^{(1)} \stackrel{D}{=} \tau_{n,L}/\sqrt{n-L+1} \mid \{x_i\}$, where $\boldsymbol{v}_n^{(1)}$, $\boldsymbol{w}_n^{(1)}$ are in the same probability space with $\E|\boldsymbol{v}_n^{(1)} - \boldsymbol{w}_n^{(1)}|^2 \le C(\Delta_n^{(1)})^2$. Additionally, by \cref{main_borel}, one could construct $\boldsymbol{u}_n^{(1)} \stackrel{D}{=} X_n/\sqrt{n}$, such that $\boldsymbol{u}_n^{(1)}$ belongs to the same probability space as $\boldsymbol{w}_n^{(1)}$ such that $\{\E|\boldsymbol{u}_n^{(1)} - \boldsymbol{w}_n^{(1)}|^2\}^{1/2} \le Cd{n^{\frac{1}{r}-\frac{1}{2}}}(\log n)$
     under \cref{asm:dependence}(a),  and $\{\E|\boldsymbol{u}_n^{(1)} - \boldsymbol{w}_n^{(1)}|^2\}^{1/2} \le Cd{n^{\frac{1}{p}-\frac{1}{2}}}(\log n)$
     under \cref{asm:dependence}(b).
     Then by the triangle inequality and \cref{key_lemma3}, we have under \cref{asm:dependence}(a) that 
\begin{equation}\label{eq:41}
    \begin{split}
        \sqrt{\E\left|\boldsymbol{u}_n^{(1)} - \boldsymbol{v}_n^{(1)}\right|^2} &\le \sqrt{\E\left|\boldsymbol{u}_n^{(1)} - \boldsymbol{w}_n^{(1)}\right|^2} + \sqrt{\E\left|\boldsymbol{w}_n^{(1)} - \boldsymbol{v}_n^{(1)}\right|^2}\\ 
        &=\bigO\left(d{n^{\frac{1}{r}-\frac{1}{2}}}(\log n)\right)+\bigO(a_n)
\end{split}
\end{equation}
for any $\omega\in A_n$. Similar result holds under \cref{asm:dependence}(b) with $r$ in \eqref{eq:41} replaced by $p$. Part (a) of the theorem follows. %by plugging in the bound of $\mathcal{W}_2(X_n/\sqrt{n}, Y_n/\sqrt{n})$ in \cref{main_borel} and by considering the event $\{\Delta_n^{(1)} \leq a_n\}$.

%Following the similar procedure, one can also construct $\boldsymbol{w}_n^{(2)} \stackrel{D}{=} Y_n/\sqrt{n}$ and $\boldsymbol{v}_n^{(2)} \stackrel{D}{=} \tau_{n,L}/\sqrt{n-L+1} \mid \{x_i\}$, where $\boldsymbol{v}_n^{(2)}$, $\boldsymbol{w}_n^{(2)}$ are in the same probability space. Additionally, by \cref{main_borel}, one could construct $\boldsymbol{u}_n^{(2)} \stackrel{D}{=} X_n/\sqrt{n}$, such that $\boldsymbol{u}_n^{(2)}$ belongs to the same probability space as $\boldsymbol{w}_n^{(2)}$, and $|\boldsymbol{u}_n^{(2)} - \boldsymbol{w}_n^{(2)}| \le \mathcal{W}_2(X_n/\sqrt{n}, Y_n/\sqrt{n})$. \\Then by triangle inequality, 
%\begin{equation}
%    \begin{split}
        %|\boldsymbol{u}_n^{(2)} - \boldsymbol{v}_n^{(2)}| &\le |\boldsymbol{u}_n^{(2)} - \boldsymbol{w}_n^{(2)}| + |\boldsymbol{w}_n^{(2)} - \boldsymbol{v}_n^{(2)}|\\ 
        %&=\bigO_\Pr\left(\mathcal{W}_2(X_n/\sqrt{n}, Y_n/\sqrt{n})\right)+\bigO_\Pr\left(d^{\frac{5}{4}}\sqrt{\Delta_n^{(2)}}\right),
%\end{split}
%\end{equation}
%where the last equality comes from the following proof of \cref{key_lemma2} and
%\cref{main_borel}. The theorem follows by plugging in the bound of $\mathcal{W}_2(X_n/\sqrt{n}, Y_n/\sqrt{n})$ in \cref{main_borel} and by considering the event $\{\Delta_n^{(2)} \leq a_n\}$.

\subsection{Proof of \cref{lemma_convex_bootstrap}}\label{proof_lemma_convex_bootstrap}
First of all, as the collection of all convex sets in $\mathbb{R}^d$ is a subset of the collection of all measurable sets, by the definition of $d_c$, we have
\begin{equation*}
d_c(X,Y)\le d_{\text{TV}}(X,Y)
\end{equation*}
where $d_{\text{TV}}$ denotes the total variation distance. Without loss of generality, suppose $\lambda_*$ is the lower bound of the smallest eigenvalue of $\Sigma$. Then, by \citet[Theorem 1.1]{devroye2018total}, we have
\[
d_{\text{TV}}(X,Y)\le \frac{3}{2}\min\left\{1, \sqrt{\sum_{i=1}^d\rho_i^2}\right\},
\]
where $\rho_1,\dots,\rho_d$ are the eigenvalues of $\Sigma^{-1}\hat{\Sigma}-I$. Since both $\Sigma$ and $\hat{\Sigma}$ are positive definite, we further have
\[
\sum_{i=1}^d\rho_i^2=\tr\left((\Sigma^{-1/2}\hat{\Sigma}\Sigma^{-1/2}-I)^2\right)=|\Sigma^{-1/2}\hat{\Sigma}\Sigma^{-1/2}-I|^2_F=|\Sigma^{-1/2}(\hat{\Sigma}-\Sigma)\Sigma^{-1/2}|^2_F.
\]
Putting everything together, we have
\[
d_c(X,Y)\le \frac{3}{2}\min\left\{1, |\Sigma^{-1/2}(\hat{\Sigma}-\Sigma)\Sigma^{-1/2}|_F\right\}\le\frac{3}{2} \min\left\{1, \lambda_*^{-1}|\Sigma-\hat{\Sigma}|_F\right\}. 
\]

\begin{comment}
We consider the following construction 
\[
X=(X_1,\dots,X_d)^T=\Sigma^{1/2}Z,\quad Y=(Y_1,\dots,Y_d)^T=\hat{\Sigma}^{1/2}Z,\]
where $Z=(Z_1,\dots,Z_d)^T\sim \mathcal{N}(0,I_d)$.
		
First, by \citep[Lemma 4.2]{Fang2015}, for any $\epsilon_1>0$,
\[\label{temptemp1}
	d_c(X,Y)\le 4 d^{1/4}\epsilon_1+ \sup_{A\in\mathcal{A}} \left|\E[h_{A,\epsilon_1}(X)-h_{A,\epsilon_1}(Y)] \right|\]
where $h_{A,\epsilon}(x)$ equals $1$ if $x\in A$, and $0$ if $x\in \R^d\setminus A^{\epsilon}$ and $\left|\triangledown h_{A,\epsilon}(x)\right|\le 2/\epsilon$.
Therefore 
\begin{equation}\label{temptemp2}
\left|\E[h_{A,\epsilon_1}(X)-h_{A,\epsilon_1}(Y)] \right|\le \left|\E[\left|\triangledown h_{A,\epsilon_1}(x)\right|\cdot\left| X-Y\right|] \right|\le \frac{2}{\epsilon_1} \E\left|X-Y\right|
\end{equation}
  then
\begin{align*}
		\E\left|X-Y\right|&=\E\sqrt{\sum_{i=1}^d (X_i-Y_i)^2}\le \sqrt{\E \sum_{i=1}^d (X_i-Y_i)^2}\\
		&=\sqrt{\text{Tr}(\Sigma^{1/2}-\hat{\Sigma}^{1/2})^2}=|\Sigma^{1/2}-\hat{\Sigma}^{1/2}|_{\textcolor{OliveGreen}{F}}
\end{align*}	
Next, using \cref{ineq-vanHemmen-Ando}, we have
\[
	|\Sigma^{1/2}-\hat{\Sigma}^{1/2}|_F\le  \lambda_*^{-1/2} |\Sigma-\hat{\Sigma}|_F.\]
The proof is completed by optimizing $\epsilon_1$.
\end{comment}

\subsection{Proof \cref{boot:convex} }\label{proof-thm42}
%This theorem follows from Lemma \ref{lemma_convex_bootstrap}, Theorem \ref{main_convex}, and similar arguments as those of the proof of Theorem \ref{boot:borel}. Details are omitted.
Let $Y_n$ be a Gaussian random vector with the same mean and covariance structure as $X_n$. Then we can write the following triangle inequality
\begin{align*}
   & d_c(X_n/\sqrt{n}, \tau_{n,L}/\sqrt{n-L+1} \mid \{x_i\}_{i=1}^n )\\ \le &  d_c(Y_n/\sqrt{n}, \tau_{n,L}/\sqrt{n-L+1} \mid \{x_i\}_{i=1}^n) +  d_c(X_n/\sqrt{n},Y_n/\sqrt{n}).
\end{align*}

Under the assumptions of \cref{boot:convex}, the time series $\{x_i\}_{i = 1}^n$ satisfies $\Delta_n^{(1)} \leq a_n$ and $\tau_{n,L}$ is a Gaussian random vector conditional on $\{x_i\}_{i=1}^n$. So we can apply \cref{lemma_convex_bootstrap} to control $d_c(Y_n/\sqrt{n}, \tau_{n,L}/\sqrt{n} \mid \{x_i\}_{i=1}^n)$ and
\cref{main_convex} to control $d_c(X_n/\sqrt{n},Y_n/\sqrt{n} )$, which finishes the proof.

\section{Proof \cref{convex:consistency} }\label{proof-convex:consistency}

This section proves the consistency result given by \cref{convex:consistency}. 
Define \[\tilde{z}_i:= \left(\frac{1}{n} \E(X^\top X)\right)^{-1}x_i{\epsilon}_i, \quad \tilde{z}^\prime_i:= \left(\frac{1}{n} \E(X^\top X)\right)^{-1}x_i\hat{{\epsilon}}_i,\] we use $\tilde{z}_i$ and $\tilde{z}^\prime_i$ as an intermediate when approximating $z_i$ and $\hat{z}_i$. Since both $x_k$ and $\epsilon_k$ are locally stationary, it is easy to see that $Y_k := x_k \epsilon_k$ is locally stationary as well. 

Recall the intermediate version of $\Gamma_n:=\frac{1}{\sqrt{n}}\sum_{i=1}^n z_i$ based on $\tilde{z}_i$ as \[  \tilde{\Gamma}_n:=\frac{1}{\sqrt{n}}\sum_{i=1}^n \tilde{z}_i,\] 
we similarly define intermediates for $\hat{G}_n$ as
\[\tilde{G}_n :=\frac{1}{\sqrt{n-L+1}} \sum_{i=1}^{n-L+1}\left(\frac{1}{\sqrt{L}}\sum_{j=i}^{i+L} \tilde{z}_j\right)  B_i \quad \text{and} \quad   \tilde{G}^\prime_n :=\frac{1}{\sqrt{n-L+1}} \sum_{i=1}^{n-L+1}\left(\frac{1}{\sqrt{L}}\sum_{j=i}^{i+L} \tilde{z}^\prime_j\right)  B_i,\] 
where $B_i$'s are i.i.d.~standard normal random variables.

Since \[\sup_{t\in \R}|\Pr(T \le t) - \Pr(T^B \le t \mid \{(x_i, y_i)\}_{i=1}^n)| \le d_c(\Gamma_n, \hat{G}_n \mid \{(x_i, y_i)\}_{i=1}^n), \] we may derive an upper bound using the triangle inequality 
\begin{align*}
& \quad \; d_c(\Gamma_n, \hat{G}_n \mid \{(x_i, y_i)\}_{i=1}^n) \\ &\le d_c(\Gamma_n, \tilde{\Gamma}_n) + d_c(\tilde{\Gamma}_n, \tilde{G}_n \mid \{(x_i, y_i)\}_{i=1}^n) + d_c(\tilde{G}_n \mid \{(x_i, y_i)\}_{i=1}^n, \hat{G}_n \mid \{(x_i, y_i)\}_{i=1}^n). 
\end{align*}
The task remains to control the three terms on the right-hand side one by one.

The next couple of lemmas will be useful to quantify the difference between $z_i$, $\tilde{z}_i$, and $\tilde{z}^\prime_i$. Note that the difference can be described by replacing $\left(\frac{1}{n}  X^TX\right)^{-1} $ with theoretical parameters $ \left(\frac{1}{n} \E  X^TX\right)^{-1}$ and/or by replacing sample observations statistics $\hat{\epsilon}$ with $\epsilon$, respectively.
	
%The next Lemma follows similar arguments of \citet[Lemma 6]{Cui_Zhou_2023}, it utilizes Bernstein-type concentration inequality \citep{Tropp_2012} and Burkholder inequality. \textcolor{OliveGreen}{We need to change the following lemma as the bound is problematic. The tentative bound we have for now is the one from Yan's FLM 2nd version on Arxiv Lemma 6 and Lemma 7, $\bigO_\pr\left(\frac{d\log(n)}{\sqrt{n}}\right)$, but this will limit the scope of $d$ to be very small, we may want to aim for a larger bound of this lemma, as a trade-off for better range of $d$ instead. Also, in the later proof, we want an upper bound of the second-order moment so that we can control the Kolmogorov distance, so the $\bigO_{\pr}$ result does not seem sufficient.}

%\textcolor{OliveGreen}{Note that the following bound only holds for spectral norm (operator norm). So when we use its Frobenius norm afterward, we will have to use $\|A\|_F \le \sqrt{d}\|A\|_2$.}
\begin{lemma}\label{consistency 1}
Assume $x_i$ satisfies \cref{asm:moment}(a), \cref{asm:dependence}(a) for $p \ge 4$ and $d^2=o(n)$, then
% \begin{equation}
% \left| \left(\frac{1}{n}X^{\top}X\right)^{-1}- \left(\frac{1}{n}\E X^{\top}X  \right)^{-1} \right |_2 = \bigO_\pr\left( \frac{d \log(n)}{\sqrt{n}} \right).
% \end{equation}
\begin{equation}
\left| \left(\frac{1}{n}X^{\top}X\right)^{-1}- \left(\frac{1}{n}\E X^{\top}X  \right)^{-1} \right |^2_F = \bigO_\pr\left( \frac{d^2}{n} \right).
\end{equation}
\end{lemma}
\begin{proof}
Observe that $\frac{1}{n}X^{\top}X = \frac{1}{n}\sum_{i=1}^n x_ix_i^{\top}$, we can write $\frac{1}{n}X^{\top}X - \E\left(\frac{1}{n}X^{\top}X \right)$ as $\frac{1}{n}\sum_{i=1}^n \left(x_ix_i^{\top} - \E(x_ix_i^{\top})\right)$. Denote $\bphi_i$ as the Kronecker product of $x_i$, that is, 
\begin{eqnarray}
    \bphi_i = (\underbrace{x_{i,1}x_{i,1}, \cdots, x_{i,1}x_{i,d}}_{d},\underbrace{x_{i,2}x_{i,1}, \cdots, x_{i,2}x_{i,d}}_{d}, \cdots, \underbrace{x_{i,d}x_{i,1}, \cdots, x_{i,d}x_{i,d}}_{d})^\top,
\end{eqnarray}
then the dependence measure regarding $\{\bphi_i\}_{i = 1}^n$ can be computed as 
\begin{eqnarray}
    \delta^{\bphi}_{j,q} &=& \max_{1 \le i \le n}\max_{k,l}\left\| \mathcal{G}_{i,k}(\FF_i)\mathcal{G}_{i,l}(\FF_i) - \mathcal{G}_{i,k}(\FF_{i,i-j})\mathcal{G}_{i,l}(\FF_{i,i-j})\right\|_q \crcr
    & \le & \max_{1 \le i \le n}\max_{k,l}\left\| \mathcal{G}_{i,k}(\FF_i)[\mathcal{G}_{i,l}(\FF_i) - \mathcal{G}_{i,l}(\FF_{i,i-j})]\right\|_{q} \crcr 
    & + &  \max_{1 \le i \le n}\max_{k,l}\left\| [\mathcal{G}_{i,k}(\FF_i) - \mathcal{G}_{i,k}(\FF_{i,i-j})]\mathcal{G}_{i,l}(\FF_{i,i-j})\right\|_q \crcr
    & \le & \left(\max_{i=1}^n \|x_i\|_{2q}\right) \cdot \delta_{j,2q}.
\end{eqnarray}
By taking $q = 2$, we have $\delta^{\bphi}_{j,2} \lesssim \delta_{j,4} \lesssim \delta_{j,p}$. 
%\textcolor{OliveGreen}{(Change the notation of dependence measure to make it compatible with the whole article.)}\\
Therefore, 
\begin{equation}
\begin{aligned}
\E\left|\frac{1}{n}X^\top X - \frac{1}{n}\E(X^\top X) \right|_F^2 & = \E\left|\frac{1}{n}\sum_{i=1}^n \left(x_ix_i^{\top} - \E(x_ix_i^{\top})\right)\right|_F^2  = \E\left|\frac{1}{n}\sum_{i=1}^n(\bphi_i - \E\bphi_i)\right|_F^2 \\
& = \frac{1}{n^2}\sum_{j=1}^{d^2} \E\left(\sum_{i=1}^n (\phi_{ij} - \E\phi_{ij})\right)^2  \lesssim \frac{d^2}{n^2}\cdot n \left(\sum_{j = 0}^{\infty}  \delta^{\bphi}_{j,2}\right)^2 \lesssim \frac{d^2}{n},
\end{aligned}
\end{equation}
where the first inequality follows from \citet[Lemma 6]{zhou2013heteroscedasticity}.

Further, note that 
\begin{eqnarray*}
\begin{aligned}
    & \quad \left| \left(\frac{1}{n}X^{\top}X\right)^{-1}- \left(\frac{1}{n}\E X^{\top}X  \right)^{-1} \right |_F \\
    & = \left|\left(\frac{1}{n}\E X^{\top}X  \right)^{-1} \left[\frac{1}{n}X^{\top}X- \frac{1}{n}\E X^{\top}X   \right]\left(\frac{1}{n}X^{\top}X\right)^{-1}\right|_F \\
    & \le \left|\left(\frac{1}{n}\E X^{\top}X  \right)^{-1}\right| \left|\frac{1}{n}X^{\top}X- \frac{1}{n}\E X^{\top}X   \right|_F\left|\left(\frac{1}{n}X^{\top}X\right)^{-1}\right| \\
    & \lesssim \left|\frac{1}{n}X^{\top}X- \frac{1}{n}\E X^{\top}X   \right|_F \mbox{ in probability},
\end{aligned}
\end{eqnarray*}
which implies that \[ \left| \left(\frac{1}{n}X^{\top}X\right)^{-1}- \left(\frac{1}{n}\E X^{\top}X  \right)^{-1} \right |^2_F = \bigO_\pr(d^2/n).\]
\end{proof}

\begin{lemma}\label{consistency 2}
Assume $\E(\epsilon_i \mid x_i) = 0$ almost surely, and $Y_i:=x_i\epsilon_i$ satisfies \cref{asm:moment}(a) and \cref{asm:dependence}(a) with $p \ge 4$, then 
\begin{equation}
\E \left|\frac{1}{\sqrt{n}} \sum_{i = 1}^n  x_i {\epsilon}_i\right|^2_F = \bigO({d})
\end{equation}
\end{lemma}
\begin{proof}
% First, we could compute the dependence measure of $Y_i := x_i \epsilon_i$, which gives that for $\forall 1 \le j \le d$, 
% \[\theta_{m,j,2}^Y \le \max_{i=1}^n{\|x_{i,j}\|_4}\cdot\theta_{m,1,4}^\epsilon  + \max_{i=1}^n{\|\epsilon_{i}\|_4}\cdot\theta_{m,j,4}^x \lesssim \theta_{m,j,4}^x,\]
% where $\theta^\epsilon$ and $\theta^x$ denote the physical dependence measure of $\{\epsilon_i\}$ and $\{x_i\}$, respectively. 
Denote the physical dependence measure of $\{Y_i\}$ under $\mathcal{L}^p$ norm regarding the $j$-th entry as $\theta^Y_{m,j,p}$, then we have
\begin{align*}
    \E\left| \frac{1}{\sqrt{n}} \sum_{i = 1}^n  x_i {\epsilon}_i  \right|^2_F
    & = \frac{1}{{n}}  \sum_{j= 1}^d \E \left( \sum_{i= 1}^n x_{ij} \epsilon_i  \right)^2 \\
    & = \frac{1}{{n}}  \sum_{j= 1}^d \left( \sum_{i= 1}^n \E [x_{ij}^2 \epsilon_i^2] + 2\sum_{1 \le k < l \leq n } \E[x_{kj}x_{lj}\epsilon_k\epsilon_l ]  \right) \\
    & \leq \frac{1}{{n}}  \sum_{j= 1}^d \left( \sum_{i= 1}^n \E [x_{ij}^2 \epsilon_i^2] + \sum_{k = 1 }^{n-1} \sum_{l>k}\theta^Y_{l-k,j,2}  \right) \\
    & \lesssim \frac{1}{{n}}  \sum_{j= 1}^d \left( \sum_{i= 1}^n \E [x_{ij}^2 \epsilon_i^2] \right) \\
    & = \bigO(d).
\end{align*} 
%\textcolor{OliveGreen}{If we are only interested in the magnitude of the sample mean itself, we can directly use \citep[Lemma 6(i)]{zhou2013heteroscedasticity} to get $||\frac{1}{\sqrt{n}}\sum_{i=1}^n x_{ij} \epsilon_i||_4 = \bigO(1)$ for each $1 \le j \le d$, and so $\frac{1}{\sqrt{n}}\sum_{i=1}^n x_{i} \epsilon_i = \bigO_\pr({d^{1/4}}).$}
\end{proof}

\begin{lemma}\label{consistency 3}
Assume %the discrepancy measure of $Y_k = x_k \epsilon_k $ satisfies $ \delta_{k,2q} = \bigO(\chi^{-k})$ for some $\chi \in (0,1)$ and $q > 2$, 
$x_i$ and $Y_i=x_i\epsilon_i$ satisfies \cref{asm:dependence}(a), \cref{asm:moment}(a) with $p > 4$ , the smallest eigenvalue $\lambda^*$ of the covariance matrix $\E\left(\frac{1}{n} X^TX\right)$ is bounded away from 0 uniformly in $n$ and $d^2/n\rightarrow 0$,  then \[ \max_{1 \leq i \leq n} | \hat{\epsilon}_i - \epsilon_i| = \bigO_\pr(dn^{ \frac{2}{p} - \frac{1}{2}} ). \]
\end{lemma}

\begin{proof}
Let $E_i$ be the $i$-th standard basis in $\R^n$. Note that, by the proof of Lemma \ref{consistency 1},
the probability that the smallest eigenvalue of $X^\top X/n$ is greater or equal to $\lambda_*/2$ goes to 1 as $n\rightarrow\infty$. Therefore,
\begin{align*}
  \max_i|\hat{\epsilon}_i - \epsilon_i|
  & =  \max_i\left|E_i^\top  ( \hat{\epsilon}  - \epsilon )\right| = \max_i\left| E_i^\top X  (\beta-  \hat{\beta} )\right| =  \max_i \left|E_i^\top X  (X^\top X )^{-1} X^\top\epsilon \right|\\ 
  & \lesssim \max_i \left|E_i^\top X \frac{X^\top \epsilon}{n}\right| \mbox{ in probability.}
\end{align*}
On the other hand, simple calculations yield that
$$\left\|E_i^\top X \frac{X^\top \epsilon}{n}\right\|_{\frac{p}{2}}\le \|E_i^\top X\|_{p} \left\|\frac{X^\top \epsilon}{n}\right\|_p=O(d/\sqrt{n}), $$
%\textcolor{OliveGreen}{Modified the bound, original one: $\| \max_i | \hat{\epsilon}_i - \epsilon_i| \|_q = \bigO(d n^{-1/2 + 1/q}) $. (but need to double check what are the assumptions on $X$ on earth.) \\}
which implies that $\left\|\max_i\left|E_i^\top X \frac{X^\top \epsilon}{n}\right|\right\|_{\frac{p}{2}}=\bigO\left(dn^{\frac{2}{p} - \frac{1}{2}}\right)$.\\
Therefore, we have that
$\max_i | \hat{\epsilon}_i - \epsilon_i|= \bigO_\pr( d n^{ \frac{2}{p} - \frac{1}{2}})$.

%{\color{red} Shiqi: I am trying to correct the proof here. I cannot get the $O(d^{1/p}n^{2/p-1/2})$ result you got here. Instead, I have a $O(dn^{2/p-1/2})$ bound. Can you check? In particular, can you check your calculations for $\|E_i^\top X\|_{p} \|\frac{X^\top \epsilon}{n}\|_p$? What I got was $\|E_i^\top X\|_{p}=O(\sqrt{d})$ and $\|\frac{X^\top \epsilon}{n}\|_p=O(\sqrt{d}/\sqrt{n})$. (Actually, intuitively there is no reason why those norms decrease when $p$ increases like what you derived.) If my calculations are correct, can you correct the bounds in the subsequent proofs of Theorems 5.1 and 5.2 where this lemma is used? Please note that for a random vector $Z$, $\|Z\|_p=\E([z^2_1+\cdots+z^2_d]^{p/2})^{1/p}$. It is not $[\E(|z_1|^p+\cdots+|z_d|^p)]^{1/p}$.}

\end{proof}

%\textcolor{OliveGreen}{Notes on the next lemma: there are 2 ways to control this bound. The first method is to follow the previous idea using \citet[Lemma 4.2]{Fang2015}, while the second method is to use the new \cref{lemma_convex_bootstrap}. Both versions are provided to see which gives a better upper bound. (Seems we will go with the first method.)}

\begin{lemma}\label{consistency 4}
Assume conditions of \cref{convex:consistency} hold, then 
%\[ d_c(\hat{G}_n \mid \{(x_i, y_i)\}_{i=1}^n, \tilde{G}_n \mid \{(x_i, y_i)\}_{i=1}^n) = \bigO_{\pr}\left(\frac{d^{\frac{7}{8}} + d^{\frac{3}{8}+\frac{1}{2p}}n^{\frac{1}{p}}}{n^{\frac{1}{4}}}\right).\]
\[ d_c(\hat{G}_n \mid \{(x_i, y_i)\}_{i=1}^n, \tilde{G}_n \mid \{(x_i, y_i)\}_{i=1}^n) = \bigO_{\pr}\left(d^{\frac{7}{8}} n^{\frac{1}{p}-\frac{1}{4}}\right).\]
\end{lemma}
\begin{proof}
Note that conditional on $\{(x_i, y_i)\}_{i=1}^n$, $\hat{G}_n$, $\tilde{G}^\prime_n$ and $\tilde{G}_n$ are mean $0$ Gaussian vectors, according to %\cref{lemma_convex_bootstrap}, 
\citet[Lemma 4.2]{Fang2015}, for $\forall \epsilon_1 > 0$,
% \textcolor{OliveGreen}{\begin{eqnarray}\label{eq:Gaussian_comparison_Frobenius}
% \begin{aligned}
%     & \quad d_c(\hat{G}_n \mid \{(x_i, y_i)\}_{i=1}^n, \tilde{G}_n \mid \{(x_i, y_i)\}_{i=1}^n) \\ & \lesssim \left|\Cov\left(\hat{G}_n \mid \{(x_i, y_i)\}_{i=1}^n\right) -  \Cov\left(\tilde{G}_n \mid \{(x_i, y_i)\}_{i=1}^n\right)\right|_F, 
% \end{aligned}
% \end{eqnarray}}
\begin{eqnarray}\label{eq:Gaussian_comparison_Frobenius}
\begin{aligned}
    & \quad d_c(\hat{G}_n \mid \{(x_i, y_i)\}_{i=1}^n, \tilde{G}_n \mid \{(x_i, y_i)\}_{i=1}^n) \\ & \le 4d^{\frac{1}{4}}\epsilon_1 + \frac{2}{\epsilon_1}\E\left|\left(\hat{G}_n \mid \{(x_i, y_i)\}_{i=1}^n\right) -  \left(\tilde{G}_n \mid \{(x_i, y_i)\}_{i=1}^n\right)\right|_F, 
\end{aligned}
\end{eqnarray}
by optimizing $\epsilon_1$ we have
\[d_c(\hat{G}_n \mid \{(x_i, y_i)\}_{i=1}^n, \tilde{G}_n \mid \{(x_i, y_i)\}_{i=1}^n) \le 4\sqrt{2}d^{\frac{1}{8}}\left(\E\left|\left(\hat{G}_n \mid \{(x_i, y_i)\}_{i=1}^n\right) -  \left(\tilde{G}_n \mid \{(x_i, y_i)\}_{i=1}^n\right)\right|_F\right)^{\frac{1}{2}}.\]
Note that
\begin{eqnarray*}
\begin{aligned}
& \quad \left|\left(\hat{G}_n \mid \{(x_i, y_i)\}_{i=1}^n\right) -  \left(\tilde{G}_n \mid \{(x_i, y_i)\}_{i=1}^n\right)\right|_F \\ & \le \left|\left(\hat{G}_n \mid \{(x_i, y_i)\}_{i=1}^n\right) -  \left(\tilde{G}^\prime_n \mid \{(x_i, y_i)\}_{i=1}^n\right)\right|_F + \left|\left(\tilde{G}^\prime_n \mid \{(x_i, y_i)\}_{i=1}^n\right) -  \left(\tilde{G}_n \mid \{(x_i, y_i)\}_{i=1}^n\right)\right|_F,
\end{aligned}
\end{eqnarray*}
by elementary calculations and Jensen's inequality,
\begin{eqnarray*}
\begin{aligned}
        &\quad \E\left|\left(\hat{G}_n \mid \{(x_i, y_i)\}_{i=1}^n\right) -  \left(\tilde{G}_n \mid \{(x_i, y_i)\}_{i=1}^n\right)\right|_F \\ & \le \sqrt{\E\left((\hat{G}_n - \tilde{G}^\prime_n)^\top(\hat{G}_n - \tilde{G}^\prime_n) \mid \{(x_i, y_i)\}_{i=1}^n \right)} + \sqrt{\E\left((\tilde{G}^\prime_n - \tilde{G}_n)^\top(\tilde{G}^\prime_n - \tilde{G}_n) \mid \{(x_i, y_i)\}_{i=1}^n \right)} \\ & := I_1 + I_2.
\end{aligned}
\end{eqnarray*}
\begin{align*}
I_1^2 &= \E\left\{\left[\frac{\sum_{i=1}^{n-L+1}\left(\frac{1}{\sqrt{L}}\sum_{j=i}^{i+L}(\hat{z}_j - \tilde{z}^\prime_j)^\top B_i\right)}{\sqrt{n-L+1}}\right]\left[\frac{\sum_{i=1}^{n-L+1}\left(\frac{1}{\sqrt{L}}\sum_{j=i}^{i+L}(\hat{z}_j - \tilde{z}^\prime_j)B_i\right)}{\sqrt{n-L+1}}\right]\Bigg| \{(x_i, y_i)\}_{i=1}^n\right\} \\
& = \frac{1}{L(n-L+1)}\sum_{i=1}^{n-L+1}\left(\sum_{j=i}^{i+L}(\hat{z}_j - \tilde{z}^\prime_j)^\top\right)\left(\sum_{l=i}^{i+L}(\hat{z}_l - \tilde{z}^\prime_l)\right) \\
& = \frac{1}{L(n-L+1)}\sum_{i=1}^{n-L+1}\left|\sum_{l=i}^{i+L}(\hat{z}_l - \tilde{z}^\prime_l)\right|^2_F \\
& =  \frac{1}{L(n-L+1)}\sum_{i=1}^{n-L+1} \left|\left(   \left( \frac{1}{n} X^T X \right)^{-1} - \left( \frac{1}{n} \E X^T X \right)^{-1}  \right) \sum_{l=i}^{i+L} x_l  \hat{\epsilon}_l \right|_F^2 \\
& \le \frac{1}{n-L+1}\sum_{i=1}^{n-L+1} \left|   \left( \frac{1}{n} X^T X \right)^{-1} - \left( \frac{1}{n} \E X^T X \right)^{-1} \right|^2_F \left|\frac{1}{\sqrt{L}}\sum_{l=i}^{i+L} x_l  \hat{\epsilon}_l \right|_F^2 \\
& \lesssim \frac{d^2}{n} \left(\left|\frac{1}{\sqrt{L}}\sum_{l=i}^{i+L} x_l (\hat{\epsilon}_l-{\epsilon}_l) \right|_F^2 + \left|\frac{1}{\sqrt{L}}\sum_{l=i}^{i+L} x_l  {\epsilon}_l \right|_F^2\right) \\
& \lesssim \frac{d^2}{n} \cdot \left(d^{3}n^{\frac{4}{p}-1} + d\right)= \bigO_\pr\left(\frac{d^3}{n} + d^{5}n^{\frac{4}{p}-2} \right), %\left(\sqrt{d}\cdot \textcolor{OliveGreen}{\frac{d\log(n)}{\sqrt{n}}}\right)^2\cdot d  = \bigO_\pr\left(\frac{d^4\log^2(n)}{n}\right), 
\end{align*}
where the last inequality follows from \cref{consistency 1}, \cref{consistency 2}, and  \cref{consistency 3}.\\
%\textcolor{OliveGreen}{PS: the bound for $\left|\sum_{j=i}^{i+L} x_i  \hat{\epsilon}_i \right|_F^2$ should be dominated by the bound of $\left|\sum_{j=i}^{i+L} x_i  {\epsilon}_i \right|_F^2$.}\\
Similarly, 
\begin{align*}
I_2^2 &= \E\left\{\left[\frac{\sum_{i=1}^{n-L+1}\left(\frac{1}{\sqrt{L}}\sum_{j=i}^{i+L}(\tilde{z}_j - \tilde{z}^\prime_j)^\top B_i\right)}{\sqrt{n-L+1}}\right]\left[\frac{\sum_{i=1}^{n-L+1}\left(\frac{1}{\sqrt{L}}\sum_{j=i}^{i+L}(\tilde{z}_j - \tilde{z}^\prime_j)B_i\right)}{\sqrt{n-L+1}}\right]\Bigg| \{(x_i, y_i)\}_{i=1}^n\right\} \\
& = \frac{1}{L(n-L+1)}\sum_{i=1}^{n-L+1}\left|\sum_{l=i}^{i+L}(\tilde{z}_l - \tilde{z}^\prime_l)\right|^2_F \\
& =  \frac{1}{L(n-L+1)}\sum_{i=1}^{n-L+1} \left| \left( \frac{1}{n} \E X^\top X \right)^{-1}   \sum_{l=i}^{i+L} x_l  (\epsilon_l - \hat{\epsilon}_l) \right|_F^2 \\
& \le \frac{1}{n-L+1}\sum_{i=1}^{n-L+1} \left|\left( \frac{1}{n} \E X^\top X \right)^{-1} \right|^2\left|\frac{1}{\sqrt{L}}\sum_{l=i}^{i+L} x_l (\epsilon_l - \hat{\epsilon}_l) \right|_F^2 \\
%& \lesssim d^{\frac{2}{p}+1}n^{\frac{4}{p}-1}  = \bigO_\pr\left(d^{\frac{2}{p}+1}n^{\frac{4}{p}-1}\right),
& \lesssim d \cdot d^{2}n^{\frac{4}{p}-1}  = \bigO_\pr\left(d^{3}n^{\frac{4}{p}-1}\right),
\end{align*}
where the last inequality follows from \cref{consistency 3}. Combining the results above, we have
\begin{align*}
   & \quad d_c(\hat{G}_n \mid \{(x_i, y_i)\}_{i=1}^n, \tilde{G}_n \mid \{(x_i, y_i)\}_{i=1}^n)\\ & \le 4\sqrt{2}d^{\frac{1}{8}}\left(\E\left|\left(\hat{G}_n \mid \{(x_i, y_i)\}_{i=1}^n\right) -  \left(\tilde{G}_n \mid \{(x_i, y_i)\}_{i=1}^n\right)\right|_F\right)^{\frac{1}{2}}\\
   & = \bigO_{\pr}\left(d^{\frac{7}{8}} n^{\frac{1}{p}-\frac{1}{4}}\right).
   %& = \bigO_{\pr}\left(\frac{d^{\frac{7}{8}} + d^{\frac{3}{8}+\frac{1}{2p}}n^{\frac{1}{p}}}{n^{\frac{1}{4}}}\right). %\bigO_{\pr}\left(\frac{d^{\frac{9}{8}}\sqrt{\log n} + d^{\frac{5}{8}+\frac{1}{4q}}n^{\frac{1}{2q}}}{n^{\frac{1}{4}}}\right).
\end{align*}

\end{proof}

\begin{lemma}\label{consistency 5}
Assume the conditions of \cref{convex:consistency} hold, consider the event $A_n = \{ \Delta_n \leq  (d \sqrt{L/n} + d/L  )h_n \}$, where $h_n \to \infty$ at an arbitrarily slow rate, then $\Pr(A_n) = 1 - o(1)$, and on the event $A_n$, 
\[ d_c( \tilde{\Gamma}_n, \tilde{G}_n \mid \{(x_i,y_i)\}_{i=1}^n) = \bigO\left(d^{\frac{7}{4}}n^{\frac{9}{2p}-\frac{1}{2}}(\log n)^2+d\left(\left(\frac{L}{n}\right)^{\frac{1}{2}}
+\frac{1}{L}\right)h_n\right). \]
\end{lemma}
\begin{proof}
Denote $\{v_i\}_{i=1}^n$ as $\{x_i\epsilon_i\}_{i=1}^n$'s Gaussian counterparts. Then based on \cref{main_convex}, we can approximate $ \tilde{\Gamma}_n $ %(and similarly $\tilde{G}_n \mid \{(x_i, y_i)_{i=1}^n\}$) 
with a Gaussian version constructed by $\{v_i\}_{i=1}^n$ at the cost of $\bigO_{\pr}(d^{\frac{7}{4}}n^{\frac{9}{2p}-\frac{1}{2}}(\log n)^2)$. 

Combining \cref{lemma_convex_bootstrap} with \cref{main_convex} we have
\begin{equation}\label{eq:d_tilde_gamma_G}
\begin{aligned}
    & \quad d_c( \tilde{\Gamma}_n, \tilde{G}_n \mid \{(x_i,y_i)\}_{i=1}^n) \\ & \leq d_c\left(\tilde{\Gamma}_n, \frac{1}{\sqrt{n}}\sum_{i=1}^n\left(\frac{1}{n}\E X^\top X\right)^{-1}v_i\right) + d_c\left(\frac{1}{\sqrt{n}}\sum_{i=1}^n\left(\frac{1}{n}\E X^\top X\right)^{-1}v_i, \tilde{G}_n \mid \{(x_i,y_i)\}_{i=1}^n \right)\\
    %+ \bigO \left(\left|\Cov\left[ \tilde{\Gamma}_n \right] -\Cov\left[ \tilde{G}_n | \{(x_i,y_i)\}_{i=1}^n \right]  \right|_F\right) \\ & +  d_c\left(\frac{1}{\sqrt{n-L+1}\sqrt{L}}\sum_{i=1}^{n-L+1}\sum_{j=i}^{i+L}\left(\frac{1}{n}\E X^\top X \right)^{-1}v_i, \tilde{G}_n \mid \{(x_i,y_i)\}_{i=1}^n) \right)\\
       & = \bigO(d^{\frac{7}{4}}n^{\frac{9}{2p}-\frac{1}{2}}(\log n)^2)  %+\bigO \left(d^{1/8} \left|\Cov\left[ \tilde{\Gamma}_n \right] -\Cov\left[ \tilde{G}_n | \{(x_i,y_i)\}_{i=1}^n \right]  \right|_F^{1/2} \right).
       + C\left(\left|\Cov\left[ \tilde{\Gamma}_n \right] -\Cov\left[ \tilde{G}_n | \{(x_i,y_i)\}_{i=1}^n \right]  \right|_F \right)
\end{aligned}
\end{equation}
for some positive and finite constant $C$.
%\textcolor{OliveGreen}{Similarly, the result can be replaced with the covariance itself, based on the new Lemma 4.4. The other option is $\left|\Cov\left[ \tilde{\Gamma}_n \right] -\Cov\left[ \tilde{G}_n | \{(x_i,y_i)\}_{i=1}^n \right]  \right|_F$.}
Expand out covariance matrix of $\tilde{\Gamma}_n$ and $\tilde{G}_n \mid \{(x_i,y_i)\}_{i=1}^n$ we get:
\begin{align*}
\Cov\left[ \tilde{\Gamma}_n \right] %& =  \frac{1}{n}\E\left[ \left(\sum_{i =1}^n \tilde{z}_i \right) \left(\sum_{i = 1}^n \tilde{z}_i \right)^T\right] \\
& =  \left(\frac{1}{n} \E X^\top X\right)^{-1} \frac{1}{n}  \left[  \sum_{ 1 \leq k, l \leq n} \E[ {\epsilon}_k {\epsilon}_l x_k  x_l^\top]  \right] \left(\frac{1}{n} \E X^\top X\right)^{-1} 
\end{align*}
and
\begin{align*}
\Cov\left[ \tilde{G}_n | \{(x_i, y_i)\}_{i=1}^n \right] 
%& = \frac{1}{m(n-m) } \sum_{i=1}^{n-m}(\sum_{j=i}^{i+m}\tilde{z}_j) (\sum_{j=i}^{i+m}\tilde{z}_j)^T \\
%& = \frac{1}{m(n-m) }\sum_{i=1}^{n-m} \left[\sum_{k=i}^{i+m}\left(\frac{1}{n} \E X^TX\right)^{-1}x_k {\epsilon}_k \right] \left[\sum_{l=i}^{i+m} \left(\frac{1}{n} \E X^TX\right)^{-1}x_l{\epsilon}_l \right]^T \\
& =  \left(\frac{1}{n} \E X^TX\right)^{-1} \frac{1}{L(n-L+1) }  \sum_{i=1}^{n-L+1} \left[  \sum_{ i\leq k, l \leq i+L} {\epsilon}_k {\epsilon}_l x_k  x_l^\top  \right] \left(\frac{1}{n} \E X^\top X\right)^{-1} .
\end{align*} 
\begin{eqnarray*}
\begin{aligned}
\Delta_n & = \left| \Cov\left[ \tilde{\Gamma}_n \right]  -  \Cov\left[ \tilde{G}_n | \{(x_i, y_i)\}_{i=1}^n \right] \right|_F\\ 
  & \leq \left|\Cov\left[ \tilde{\Gamma}_n \right]  -  \E \Cov\left[ \tilde{G}_n | \{(x_i, y_i)\}_{i=1}^n \right] \right|_F + \left|\Cov\left[ \tilde{G}_n | \{(x_i, y_i)\}_{i=1}^n \right] -  \E \Cov\left[ \tilde{G}_n | \{(x_i, y_i)\}_{i=1}^n \right] \right|_F \\
%& \leq  \left| \left(\frac{1}{n} \E X^TX\right)^{-1}   \right|_F^2 \cdot \left| \frac{1}{n}  \left(  \sum_{ 1 \leq k, l \leq n} \E[ {\epsilon}_k {\epsilon}_l x_k  x_l^\top ]  \right) -  \frac{1}{m(n-m ) }  \sum_{i=1}^{n-m} \left(  \sum_{ i\leq k, l \leq i+m} \E[{\epsilon}_k {\epsilon}_l x_k  x_l^\top ] \right) \right|_F \\
	%& \quad +  \left| \left(\frac{1}{n} \E X^\top X\right)^{-1}   \right|_F^2 \cdot \left| \frac{1}{m(n-m ) }  \sum_{i=1}^{n-m}  \sum_{ i\leq k, l \leq i+m} \left(  {\epsilon}_k {\epsilon}_l x_k  x_l^\top  -   \E[{\epsilon}_k {\epsilon}_l x_k  x_l^\top ] \right) \right|_F\\
 & := I_3 + I_4. 
\end{aligned}
\end{eqnarray*}
Denote $1\leq s,t\leq d$ as indices of the coordinates, then by similar arguments of \citet[Lemma 9]{Cui_Zhou_2023}, one can show that 
\[\left\|\left(\Cov\left[ \tilde{\Gamma}_n \right]  -  \E \Cov\left[ \tilde{G}_n | \{(x_i, y_i)\}_{i=1}^n \right]\right)_{s,t}\right\|_2 = \bigO(\sqrt{L/n}), \]
and
\[\left\|\left(\Cov\left[ \tilde{G}_n | \{(x_i, y_i)\}_{i=1}^n \right] -  \E \Cov\left[ \tilde{G}_n | \{(x_i, y_i)\}_{i=1}^n \right] \right)_{s,t}\right\|_2 = \bigO(1/L), \]
\begin{comment}
\begin{align*}
	& \left| \frac{1}{n}  \left(  \sum_{ 1 \leq k, l \leq n} \E[ {\epsilon}_k {\epsilon}_l] \E [x_k  x_l^T]  \right) -  \frac{1}{m(n-m ) }  \sum_{i=1}^{n-m} \left(  \sum_{ i\leq k, l \leq i+m} \E[{\epsilon}_k {\epsilon}_l] \E[x_k  x_l^T ] \right) \right|_{st} \\
	& \leq \left| \sum_{  |i - j| < m } \left(  \frac{1}{n} - \frac{ ( m - |i-j| ) }{(n - m +1) m }  \right)  \E Y_i' Y_j    \right|_{st} + \frac{1}{n} \left| \sum_{  |i - j| \geq m }  \E Y_i' Y_j    \right|_{st} \\
	& \leq \left( \sum_{  k = 1 }^{m - 2} \left|\frac{1}{m} - \frac{( m - k ) }{(n - m +1) m } \right| (n - k)  \delta_2(k)   \right)_{st}  +  O\left( \frac{m}{n} \right)+ O(1/m) \\
	& = O\left( m/n \right) + O(1/m).
\end{align*} 
Thus we have $ (\ref{cov1}) =   O\left(d m/n \right) + O(d/m)$
\end{comment}
which leads to
\[\|I_3\|_2 = \bigO(d\sqrt{L/n}) \quad \text{and} \quad \|I_4\|_2 = \bigO(d/L).\]
%Recall the definition of $A_n$, we have $\Pr(A_n) = 1 - o(1)$, and under the event $A_n$,
Recall $h_n \to \infty$ at an arbitrarily slow rate, under $H_0$, we have by Markov's inequality that
$P(A_n)=1-o(1)$.
%\[\left|\Cov\left[ \tilde{\Gamma}_n \right] -\Cov\left[ \tilde{G}_n | \{(x_i,y_i)\}_{i=1}^n \right]  \right|_F = \bigO\left(d\left(\sqrt{\frac{L}{n}} + \frac{1}{L}\right)h_n\right),\]
On the event $A_n$, the result follows from \cref{eq:d_tilde_gamma_G}.
\end{proof}

\begin{lemma}\label{consistency 6}
Assume the conditions of \cref{convex:consistency} hold, then \[ d_c({\Gamma}_n, \tilde{\Gamma}_n) = \bigO\left(d^{\frac{7}{4}}n^{\frac{9}{2p}-\frac{1}{2}}(\log n)^2+d^{\frac{7}{8}}n^{-\frac{1}{4}} \right) .\]
\end{lemma}
\begin{proof}
Following the proof in \citet[Section 4]{Fang2015}, we define $A^\varepsilon$ and $A^{-\varepsilon}$ for $A \in \mathcal{A}$ as
\[A^{\varepsilon} = \{x \in \R^d: \text{dist}(x, A) \le \varepsilon\}, \qquad A^{-\varepsilon} = \{x \in \R^d: \text{dist}(x, \R^d \backslash A) > \varepsilon \},\]
where $\mathcal{A}$ is the collection of all convex sets in $\R^d$, and $\text{dist}(w,A) = \inf_{v \in A}|w - v|_F$. One can also verify that $A^\varepsilon, A^{-\varepsilon} \in \mathcal{A}$. Elementary calculations show that
\begin{equation*}
\begin{aligned}
  d_c \left(\Gamma_n, \tilde{\Gamma}_n\right) &= \sup_{A \in \mathcal{A}}\left|\pr \left(\Gamma_n \in A\right) - \pr \left(\tilde{\Gamma}_n \in A\right) \right| \\
  & \le \pr(|\Gamma_n - \tilde{\Gamma}_n| > \varepsilon) + \sup_{A \in \mathcal{A}}\{\pr(\tilde{\Gamma}_n \in A^{\varepsilon} \backslash A), \pr(\tilde{\Gamma}_n \in A \backslash A^{-\varepsilon})\}  \\
  & := I_5 + I_6. 
\end{aligned}
\end{equation*}
\begin{equation}\label{eq:I5}
    \begin{aligned}
        I_5 &\le \frac{1}{\varepsilon}\cdot \E\left|\left[\left(\frac{1}{n} X^TX  \right)^{-1} - \left(\frac{1}{n} \E X^TX  \right)^{-1}\right]\frac{1}{\sqrt{n}}\sum_{k =1}^n x_k \epsilon_k \right|_F \\
        & \le  \frac{1}{\varepsilon}\cdot \left(\E\left| \left( \frac{1}{n} X^TX  \right)^{-1} - \left( \frac{1}{n} \E X^TX \right)^{-1}\right|^2_F \right)^{1/2} \left(\E \left|\frac{1}{\sqrt{n}}\sum_{k =1}^n x_k \epsilon_k \right|^2_F\right)^{1/2}\\
        & \lesssim  \frac{d^{3/2}}{\sqrt{n}\varepsilon}.
    \end{aligned}
\end{equation}
Observe that 
\[\pr(\tilde{\Gamma}_n \in A^{\varepsilon} \backslash A) = \pr(\tilde{\Gamma}_n \in A^{\varepsilon}) - \pr(\tilde{\Gamma}_n \in A),\]
then for $\Gamma^{G}:= \frac{1}{\sqrt{n}}\sum_{i=1}^n z_i^G$ where $\{z_i^G\}_{i=1}^n$ is $\{\tilde{z}_i\}_{i=1}^n$'s Gaussian counterparts, we have 
\begin{equation}
\begin{aligned}
  & \quad \pr(\tilde{\Gamma}_n \in A^{\varepsilon} \backslash A) - \pr(\Gamma^{G}_n \in A^{\varepsilon} \backslash A) \\
  & = \pr(\tilde{\Gamma}_n \in A^{\varepsilon}) - \pr(\Gamma^{G}_n \in A^{\varepsilon}) - \left(\pr(\tilde{\Gamma}_n \in A) - \pr(\Gamma^{G}_n \in A)\right) \\
  & \le |\pr(\tilde{\Gamma}_n \in A^{\varepsilon}) - \pr(\Gamma^{G}_n \in A^{\varepsilon})| + |\left(\pr(\tilde{\Gamma}_n \in A) - \pr(\Gamma^{G}_n \in A)\right)| \\
  & \le 2 d_c\left(\tilde{\Gamma}_n, \Gamma_n^G \right),
\end{aligned}
\end{equation}
Similarly \[\pr(\tilde{\Gamma}_n \in A \backslash A^{-\varepsilon}) - \pr(\Gamma^{G}_n \in A\backslash A^{-\varepsilon}) \le 2 d_c\left(\tilde{\Gamma}_n, \Gamma_n^G \right).\]
This implies that
\begin{equation}\label{eq:I6}
    \begin{aligned}
        I_6 & \le \sup_{A \in \mathcal{A}}\{\pr(\Gamma^{G}_n \in A^{\varepsilon} \backslash A), \pr(\Gamma^{G}_n \in A \backslash A^{-\varepsilon})\} + 4\cdot d_c\left(\tilde{\Gamma}_n, \Gamma_n^G \right) \\
        & \lesssim d^{1/4}\varepsilon+d_c\left(\tilde{\Gamma}_n, \Gamma_n^G \right).
    \end{aligned}
\end{equation}
Combining \cref{eq:I5} and \cref{eq:I6} and optimizing the bound, we have 
\begin{equation}
    d_c(\Gamma_n, \tilde{\Gamma}_n) \lesssim d^{\frac{7}{8}}n^{-\frac{1}{4}} + d^{\frac{7}{4}}n^{\frac{9}{2p} - \frac{1}{2}}(\log n)^2,
\end{equation}
which leads to the result. 
\end{proof}

\emph{Proof \cref{convex:consistency}:}
Define the event $A_n = \{ \Delta_n \leq  (d \sqrt{L/n} + d/L  )h_n \}$, where $h_n \to \infty$ at an arbitrarily slow rate. Note that $\pr(A_n) = 1 - o(1)$. Since
\begin{equation}
    \begin{aligned}
        & \quad d_c(\Gamma_n, \hat{G}_n \mid \{(x_i, y_i)\}_{i=1}^n) \\ & \le d_c(\Gamma_n, \tilde{\Gamma}_n) + d_c(\tilde{\Gamma}_n, \tilde{G}_n \mid \{(x_i, y_i)\}_{i=1}^n) + d_c(\tilde{G}_n \mid \{(x_i, y_i)\}_{i=1}^n, \hat{G}_n \mid \{(x_i, y_i)\}_{i=1}^n),
    \end{aligned}
\end{equation}
then by \cref{consistency 4}, \cref{consistency 5} and \cref{consistency 6}, for every $\omega\in A_n$, we have 
\begin{align*}
& \quad d_c(\Gamma_n, \hat{G}_n \mid \{(x_i, y_i)\}_{i=1}^n) \\
& = \bigO\left(d^{\frac{7}{8}}n^{\frac{1}{p}-\frac{1}{4}}+d^{\frac{7}{4}}n^{\frac{9}{2p}-\frac{1}{2}}(\log n)^2+d\left(\sqrt{\frac{L}{n}}
+\frac{1}{L}\right)h_n\right).   
\end{align*}
% \begin{align*}
% & \quad d_c(\Gamma_n, \hat{G}_n \mid \{(x_i, y_i)\}_{i=1}^n) \\
% & = \bigO_\pr\left(d^{\frac{5}{8}+\frac{1}{2p}}n^{\frac{1}{p}-\frac{1}{4}}+d^{\frac{7}{4}}n^{\frac{9}{2p}-\frac{1}{2}}(\log n)^2+d^{\frac{5}{8}}\left(\left(\frac{L}{n}\right)^{\frac{1}{4}}
% +\left(\frac{1}{L}\right)^{\frac{1}{2}}\right)h_n^{\frac{1}{2}} + d^{\frac{11}{8}}n^{-\frac{1}{4}}\right).   
% \end{align*}
Hence the theorem follows from the fact that $h_n$ can approach $\infty$ arbitrarily slowly.

\section{Proof  \cref{consistency_coupling}}\label{proof-consistency_coupling}

Recall the definition of $T$ and $T^B$ in \cref{section:r2}, we aim to control the magnitude of $T - T^B$. The threshold function $\delta_{\lambda}(x)$ is Lipstchiz continuous in the sense that \[|\delta_{\lambda}(x) - \delta_{\lambda}(y)| \le |x - y|, \quad \forall x, y \in \R.\] 
\begin{comment}
Using the Lipstichiz continuity of $\delta_{\lambda}(\cdot)$, we have
\begin{eqnarray}
    |T - T^B| &=& \Bigg|\sqrt{n}\left|\hat{\beta}^{ST} - \beta_0\right|_{\max} - \sqrt{n}\left|\left(\beta_0 + \hat{G}_n/\sqrt{n}\right)^{ST} - \beta_0\right|_{\max} \Bigg| \crcr
    & \le& \sqrt{n}\left|\hat{\beta}^{ST} - \left(\beta_0 + \hat{G}_n/\sqrt{n}\right)^{ST}\right|_{\max} \crcr
    & =& \max_{1 \le j \le d} \left|\sqrt{n}\left(\delta_{\lambda_j}(\hat{\beta}_j) - \delta_{\lambda_j}({\beta}_{0,j} + \hat{G}_{n,j}/\sqrt{n})\right)\right| \crcr
    & \le & \max_{1 \le j \le d} \left|\sqrt{n}\left(\hat{\beta}_j - {\beta}_{0,j} - \hat{G}_{n,j}/\sqrt{n})\right)\right| \crcr 
    &=& \left|\sqrt{n}(\hat{\beta} - \beta_{0}) - \hat{G}_n\right|_{\max}. 
\end{eqnarray}
Under $H_0: \beta = \beta_0$, we can rewrite $\sqrt{n}(\hat{\beta} - \beta_{0})$ as $\Gamma_n = \frac{1}{\sqrt{n}}\sum_{i=1}^n\left(\frac{1}{n}X^\top X\right)^{-1}x_i\varepsilon_i$, then the goal breaks down to control the magnitude of $|\Gamma_n - \hat{G}_n \mid \{(x_i,y_i)\}_{i=1}^n|_{\max}$.\\
\end{comment}
Under $H_0: \beta = \beta_0$, we can rewrite $\sqrt{n}(\hat{\beta} - \beta_{0})$ as $\Gamma_n = \frac{1}{\sqrt{n}}\sum_{i=1}^n\left(\frac{1}{n}X^\top X\right)^{-1}x_i\varepsilon_i$, which implies that \[\hat{\beta} = \beta_0 + \Gamma_n/\sqrt{n}.\] 

\emph{Proof \cref{consistency_coupling}:}
Following the similar framework as in \cref{proof-convex:consistency}, we may control $T-T^B$ by splitting it into 3 parts: 
\begin{eqnarray}
    T - T^B &=& \sqrt{n}\left|\hat{\beta}^{ST} - \beta_0\right|_{\infty} - \sqrt{n}\left|\left(\beta_0 + \hat{G}_n/\sqrt{n}\right)^{ST} - \beta_0\right|_{\infty}  \crcr
    &= & \sqrt{n}\left|\left(\beta_0 + \Gamma_n/\sqrt{n}\right)^{ST} - \beta_0\right|_{\infty} - \sqrt{n}\left|\left(\beta_0 + \tilde{\Gamma}_n/\sqrt{n}\right)^{ST} - \beta_0\right|_{\infty}  \crcr
    & +&  \sqrt{n}\left|\left(\beta_0 + \tilde{\Gamma}_n/\sqrt{n}\right)^{ST} - \beta_0\right|_{\infty} - \sqrt{n}\left|\left(\beta_0 + \tilde{G}_n/\sqrt{n}\right)^{ST} - \beta_0\right|_{\infty}  \crcr
    &+& \sqrt{n}\left|\left(\beta_0 + \tilde{G}_n/\sqrt{n}\right)^{ST} - \beta_0\right|_{\infty}  - \sqrt{n}\left|\left(\beta_0 + \hat{G}_n/\sqrt{n}\right)^{ST} - \beta_0\right|_{\infty} \crcr
    &:=& I_1 + I_2 + I_3,
\end{eqnarray}
which further suggests that \[|T - T^B| \le |I_1| + |I_2| + |I_3|.\]
Observe that 
\begin{eqnarray*}
    |I_1| &\le& \sqrt{n}\left|\left(\beta_0 + \Gamma_n/\sqrt{n}\right)^{ST} - \left(\beta_0 + \tilde{\Gamma}_n/\sqrt{n}\right)^{ST} \right|_{\infty} \crcr 
    & = & \sqrt{n}\max_{1 \le j \le d} \left|\left(\delta_{\lambda_j}({\beta}_{0,j} + \Gamma_{n,j}/\sqrt{n}) - \delta_{\lambda_j}({\beta}_{0,j} + \tilde{\Gamma}_{n,j}/\sqrt{n})\right)\right| \crcr
    & \le & \max_{1 \le j \le d} \left|\Gamma_{n,j} - \tilde{\Gamma}_{n,j}\right| = \left|\Gamma_n - \tilde{\Gamma}_n\right|_{\infty},
\end{eqnarray*}
where the last inequality comes from the Lipstchiz inequality of $\delta_{\lambda}(\cdot)$ for any fixed $\lambda$. By intermediate results in \cref{consistency 6}, we have
\[\E|\Gamma_n - \tilde{\Gamma}_n|_F \lesssim \frac{d^{3/2}}{\sqrt{n}},\]
which leads to 
\begin{equation}\label{application:borel:I1}
    |I_1| \le |\Gamma_n - \tilde{\Gamma}_n|_{\infty} \le |\Gamma_n - \tilde{\Gamma}_n|_{F} = \bigO_{\Pr}\left(\frac{d^{3/2}}{\sqrt{n}}\right).
\end{equation}
Similarly using the Liptchiz continuity of $\delta_{\lambda}(\cdot)$, we have 
\begin{eqnarray*}
|I_3| &\le& \sqrt{n}\left|\left(\beta_0 + \tilde{G}_n/\sqrt{n}\right)^{ST} - \left(\beta_0 + \hat{G}_n/\sqrt{n}\right)^{ST} \right|_{\infty} \crcr 
& \le & \left|\hat{G}_n - \tilde{G}_n\right|_{\infty}.
\end{eqnarray*}
By the proof of \cref{consistency 4}, we also obtain that, conditional on $\{(x_i, y_i)\}_{i=1}^n$, 
\[\E|\hat{G}_n - \tilde{G}_n|_{\infty} \le \E|\hat{G}_n - \tilde{G}_n|_F \lesssim d^{\frac{3}{2}}n^{\frac{2}{p} -\frac{1}{2}}, \]
hence
\begin{equation}\label{application:borel:I3}
|I_3| = \bigO_{\Pr}\left(d^{\frac{3}{2}}n^{\frac{2}{p} -\frac{1}{2}}\right).
\end{equation}
To control the second term, without loss of generality, we assume $\tilde{G}_n \mid \{(x_i,y_i)\}_{i=1}^n$ and $\tilde{\Gamma}_n$ are defined on the same probability space. By \cref{main_borel}, we could construct $V_n = \frac{1}{\sqrt{n}}\sum_{i=1}^n{v_i}$, where $(v_1, \cdots, v_n)^\top$ is $(\tilde{z}_1, \cdots, \tilde{z}_n)^\top$'s Gaussian counterpart. Assume $(v_1, \cdots, v_n)^\top$ is on the above probability space as well. Following the similar steps as above, we have
\begin{eqnarray*}
|I_2| &\le& \sqrt{n}\left|\left(\beta_0 + \tilde{\Gamma}_n/\sqrt{n}\right)^{ST} - \left(\beta_0 + \tilde{G}_n/\sqrt{n}\right)^{ST} \right|_{\infty} \crcr 
& \le & \left|\tilde{\Gamma}_n - \tilde{G}_n\right|_{\infty} \le \left|\tilde{\Gamma}_n - \tilde{G}_n\right|_F \crcr
& \le & \left|\tilde{\Gamma}_n - V_n\right|_F  + \left|V_n - \tilde{G}_n\right|_F.
\end{eqnarray*}
The GA result from \cref{main_borel} implies that
\[|\tilde{\Gamma}_n - V_n|_F = \bigO_{\pr}(d\log(n){n^{\frac{1}{p}-\frac{1}{2}}}).\]
Further note that $V_n$ and $\tilde{G}_n \mid \{(x_i, y_i)\}_{i=1}^n$ are both Gaussian vectors, and by intermediate results in \cref{consistency 5}, we have
\[\Delta_n := \left| \Cov\left[ V_n \right]  -  \Cov\left[ \tilde{G}_n | \{(x_i, y_i)\}_{i=1}^n \right] \right|_F = \bigO_{\Pr}\left(d\sqrt{L/n} + d/L\right).\]
%by similar arguments as the proof of \cref{lemma_bootstrap_delta},
%\[\Delta_n := \left| \Cov\left[ V_n \right]  -  \Cov\left[ \tilde{G}_n | \{(x_i, y_i)\}_{i=1}^n \right] \right|_{\max} = \bigO_{\Pr}\left(d^{4/p}\sqrt{L/n} + L/n + 1/L\right).%\bigO_{\Pr}\left(d\sqrt{m/n} + d/m\right).
%\]
By \cref{key_lemma3}, conditional on $\{(x_i, y_i)\}_{i=1}^n$, %on the event $A_n$ defined in \cref{consistency_coupling}, we have
for any $h_n \to \infty$ arbitrarily slow, we have
% \[|Y_n - \tilde{G}_n|_F = \bigO_{\Pr}\left(d^{7/4}\left(\left(\frac{m}{n}\right)^{\frac{1}{4}}
% +\left(\frac{1}{m}\right)^{\frac{1}{2}}\right)\right).\]
\[|V_n - \tilde{G}_n|_F = \bigO_{\Pr}\left(d\left(\sqrt{\frac{L}{n}} + \frac{1}{L}\right)h_n\right).\]
%\textcolor{OliveGreen}{Or we can use the relationship between Frobenius norm and maximum norm to simply bound this by the Frobenius norm, which is an intermediate result in \cref{consistency 5} and the bound will be $\bigO_{\Pr}\left(d\sqrt{L/n} + d/L\right)$.}
Therefore we have
\begin{equation}\label{application:borel:I2}
    |I_2| = \bigO_{\Pr}\left(d\log(n){n^{\frac{1}{p}-\frac{1}{2}}} + d\left(\sqrt{\frac{L}{n}} + \frac{1}{L}\right)h_n\right).
\end{equation}
Combining \cref{application:borel:I1,application:borel:I3,application:borel:I2}, the result follows. 

\section{Preliminary results}
%\textcolor{OliveGreen}{Modified: added subscript $F$ to represent Frobenius norm and to be consistent with the notion in the main text.}
\begin{lemma}\label{ineq-vanHemmen-Ando}
    \cite[Corollary 4.2]{van1980inequality} 
    Let $\Sigma_1$ and $\Sigma_2$ be two positive definite matrices with the smallest eigenvalue bounded below by $\lambda_*>0$. Then for every $0<r\le 1$, we have 
    \[
    |\Sigma_1^r-\Sigma_2^r|_F \le \left(\frac{1}{\lambda_*}\right)^{1-r}|\Sigma_1-\Sigma_2|_F.
    \]
\end{lemma}

\end{document}